\documentclass[12pt,a4paper]{article}
\usepackage{amsmath,amsfonts,amssymb,amsthm,latexsym} 
\usepackage{graphicx,stmaryrd}
\usepackage[all]{xy}
\newtheorem{Def}{Definition}[section]
\newtheorem{Prop}[Def]{Proposition}
\newtheorem{Thm}[Def]{Theorem}
\newtheorem{Lem}[Def]{Lemma}
\newtheorem{Cor}[Def]{Corollary}

\theoremstyle{remark}
\newtheorem{Rem}[Def]{Remark}

\newtheorem{Ex}[Def]{Example}

\renewcommand{\labelenumi}{\theenumi)}

\renewcommand{\[}{\begin{equation}}
\renewcommand{\]}{\end{equation}}

\newcommand{\ere}{\hfill\mbox{$\Diamond$}\end{Rem}}
\newcommand{\eex}{\hfill\mbox{$\Diamond$}\end{Ex}}

\newcommand{\bsb}{\begin{Sb}}
\newcommand{\esb}{\end{Sb}}

\newcommand{\Z}{{\mathbb Z}}

\newcommand{\C}{{\mathbb C}}

\newcommand{\te}{\theta}

\newcommand{\sH}{{\mathcal H}}

\newcommand{\sC}{{\mathcal C}}
\newcommand{\sD}{{\mathcal D}}

\newcommand{\sO}{{\mathcal O}}

\newcommand{\ot}{\otimes}

\newcommand{\tens}{\otimes}

\newcommand{\id}{\mathrm{id}}

\newcommand{\im}[1]{\mathrm{im}\,#1}

\def\Rhom#1#2#3{{{\rm Hom}\sp{#1}(#2,#3)}}
\def\Lhom#1#2#3{{{}\sp{#1}{\rm Hom}(#2,#3)}}

\newcommand{\can}{\operatorname{\it can}}

\newcommand{\Me}{\mathbf{M}_A^C(\psi)}
\newcommand{\Mp}{\mathbf{M}_P^C(\psi)}
\newcommand{\Es}{{(A,C, {\psi})}}
\newcommand{\Ep}{{(P,C,\psi)}}
\newcommand{\Op}{{\Omega^1 P}}
\newcommand{\Ob}{{\Omega^1 B}}
\newcommand{\Ker}{\operatorname{Ker}}

\newcommand{\Hom}{\operatorname{Hom}}

\def\sw#1{{\sb{(#1)}}}
\def\sco#1{{\sp{(#1)}}} 
\def\su#1{{\sp{[#1]}}} 
\def\eps{\varepsilon}
\def\DC{\Delta_{\sC}}
\def\eC{\eps_\sC}
\def\DH{\Delta_{\sH}}
\def\eH{\eps_\sH}
\def\ut{\otimes}

\def\act{\triangleleft}
\def\dl{{}_P\Delta}
\def\csop{\mbox{$(C^*)^{op}$}}
\def\lfa{{\mbox{$\forall\;$}}}
\def\fa{\forall\,\,}
\def\Rhom#1#2#3{{{\rm Hom}\sp{#1}(#2,#3)}}
\def\ten#1{\underset{#1}{\otimes}}
\def\coten#1{\underset{#1}{\square}}

\def\note#1{}
\def\wegdamit#1{}
\def\CM{{\mathbf M}}
\def\st{such that}
\def\eps{{\epsilon}}
\def\vt{\triangleright}

\newcommand{\ff}{faithfully flat}

\renewcommand{\phi}{\varphi}

\newcommand{\Llp}{\mbox{\Large $($}}
\newcommand{\Lrp}{\mbox{\Large $)$}}
\newcommand{\llp}{\mbox{\large $($}}
\newcommand{\lrp}{\mbox{\large $)$}}

\renewcommand{\epsilon}{\varepsilon}
\newcommand{\ra}{\rightarrow}
\newcommand{\ci}{\circ}
\newcommand{\cc}{\!\ci\!}

\newcommand{\lra}{\longrightarrow}

\newcommand{\inc}{\mbox{$\,\subseteq\;$}}
\renewcommand{\subset}{\inc}

\newcommand{\dr}{\mbox{$\Delta_{P}$}}
\newcommand{\dsr}{\mbox{$\Delta_{\Omega^1P}$}}

\newcommand{\hd}{\mbox{$\delta$}}
\newcommand{\he}{\mbox{$\varepsilon$}}

\newcommand{\hD}{\mbox{$\Delta$}}

\newcommand{\ho}{\mbox{$\omega$}}
\newcommand{\hO}{\mbox{$\Omega$}}

\renewcommand{\tilde}{\widetilde}

\def\hge{Hopf-Galois $H$-extension}
\def\cge{coalgebra-Galois $C$-extension}
\def\cg{coalgebra-Galois}

\def\st{\stackrel}
\def\<{\langle}
\def\>{\rangle}
\def\d{\mbox{$\mathop{\mbox{\rm d}}$}}
\def\id{\mbox{$\mathop{\mbox{\rm id}}$}}
\def\im{\mbox{$\mathop{\mbox{\rm Im$\,$}}$}}
\def\ker{\mbox{$\mathop{\mbox{\rm Ker$\,$}}$}}
\def\coker{\mbox{$\mathop{\mbox{\rm Coker$\,$}}$}}
\def\hom{\operatorname{Hom}}

\def\map{\operatorname{Map}}
\def\te{\otimes}
\def\tens{\otimes}

\def\sO{{\mathcal O}}

\def\C{{\mathbb C}}

\def\Z{{\mathbb Z}}

\def\sH{{\mathcal H}}

\def\im{{\rm Im}}
\def\id{{\rm id}}

\newcounter{zlist}
\newenvironment{zlist}{\begin{list}{(\arabic{zlist})}{
\usecounter{zlist}\leftmargin2.5em\labelwidth2em\labelsep0.5em
\topsep0.6ex
\parsep0.3ex plus0.2ex minus0.1ex}}{\end{list}}

\newcounter{blist}
\newenvironment{blist}{\begin{list}{(\alph{blist})}{
\usecounter{blist}\leftmargin2.5em\labelwidth2em\labelsep0.5em
\topsep0.6ex
\parsep0.3ex plus0.2ex minus0.1ex}}{\end{list}}

\newcounter{rlist}
\newenvironment{rlist}{\begin{list}{(\roman{rlist})}{
\usecounter{rlist}\leftmargin2.5em\labelwidth2em\labelsep0.5em
\topsep0.6ex
\parsep0.3ex plus0.2ex minus0.1ex}}{\end{list}}

\newcommand{\M}{\operatorname{\bf M}}

\newcommand{\twa}{{}^{[1]}}
\newcommand{\twb}{{}^{[2]}}

\renewcommand\o{\otimes}

\renewcommand\epsilon\varepsilon

\newfont{\bbb}{msbm10 scaled\magstep1}

\def\kdots{\mathinner{\mkern1mu\raise0pt\vbox{\kern7pt\hbox{.}}\mkern2mu\raise4pt\hbox{.}\mkern2mu\raise8pt\hbox{.}}}

\newcommand{\Rend}[2]{\mathrm{End}\sp{#1}(#2)}
\newcommand{\rhom}[3]{\mathrm{Hom}\sb{#1}(#2,#3)}

\newcommand{\LRhom}[4]{{}^{#1}\mathrm{Hom}\sp{#2}(#3,#4)}

\def\Mc{{\mathbf M}}

\def\t{\theta}

\def\ra{\rightarrow}

\def\text{\hbox}

\def\ra{\rightarrow}

\def\Hom{\mathop{\rm Hom}\nolimits}

\def\Ker{\mathop{\rm Ker}\nolimits}

\def\build#1_#2^#3{\mathrel{\mathop{\kern 0pt#1}\limits_{#2}^{#3}}}

\begin{document} 

\title{\vspace*{-15mm}\LARGE \bf GALOIS-TYPE EXTENSIONS AND EQUIVARIANT PROJECTIVITY}
\label{2chapter4}
\author{\vspace*{-0mm}\large\sc Tomasz Brzezi\'nski \\
Department of Mathematics, Swansea University, \\
  Singleton Park,   Swansea SA2 8PP, U.K.\\ 
E-mail: T.Brzezinski@swansea.ac.uk \and
\vspace*{-0mm}\large\sc
Piotr M.~Hajac\\
\vspace*{-0mm}\large
Instytut Matematyczny, Polska Akademia Nauk\\
\vspace*{-0mm}\large
ul.\ \'Sniadeckich 8, Warszawa, 00-956 Poland
\vspace{0mm}\\
\large\sl
http://www.impan.gov.pl/$\!\widetilde{\phantom{m}}\!$pmh\\
\vspace*{0mm}\large
and\\
\vspace*{-0mm}\large
Katedra Metod Matematycznych Fizyki, Uniwersytet Warszawski\\
\vspace*{-0mm}\large
ul.\ Ho\.za 74, Warszawa, 00-682 Poland}
\date{}
\maketitle
\begin{abstract}\normalsize
The theory of general Galois-type extensions  is presented, including the interrelations between coalgebra extensions and algebra (co)extensions,  properties of corresponding (co)translation maps, and rudiments of entwinings and factorisations. To achieve broad perspective, this theory is placed in the context of far reaching generalisations of the Galois condition to the setting of corings. At the same time,
to bring together $K$-theory and general Galois theory, the equivariant projectivity of extensions is assumed resulting in the centrepiece concept of a {\em principal extension}. 
Motivated by noncommutative geometry, we employ such extensions as replacements
of principal bundles. This brings about the notion of a strong connection and yields finitely generated projective associated modules, which play the role of noncommutative vector bundles.  Subsequently, the theory of strong connections is developed. It is purported as
a basic ingredient in the construction of the Chern character for Galois-type extensions
(called the Chern-Galois character).
\end{abstract}

\begin{center}{\bf Acknowledgements}\end{center}

This work was partially supported by the European Commission grant
MKTD-CT-2004-509794 and the Polish Government grants
115/E-343/SPB/6.PR UE/DIE 50/2005 - 2008 and N201 1770 33.
\newpage
{\footnotesize\tableofcontents}

\section{Introduction}

Taking advantage of Peter-Weyl theory, principal comodule algebras (faithfully flat Hopf-Galois extensions with bijective antipodes) have been shown  \cite{bh} to generalise compact principal bundles in the sense of Henri Cartan (no local triviality assumed). On the other hand, there are examples of quantum spaces which, classically, correspond to principal bundles, yet do not fit the Hopf-Galois framework. More specifically, a natural source of examples of principal bundles is  provided by homogeneous spaces. These can always be defined as quotients of a group by its subgroup. In the case of Hopf algebras  understood as quantum groups, however, there is a rather limited
number of quantum subgroups (given by  surjections of Hopf algebras). 
As a result, not every quantum homogeneous space is a quotient of a quantum group
by its quantum subgroup. For example, only one member
of the family of quantum 2-spheres defined in \cite{p-p87} can be
obtained as a quotient of $SU_q(2)$ by $U(1)$. The theory of
Hopf-Galois extensions
can only describe quantum homogeneous
 spaces that are quotients of quantum groups by quantum subgroups. 

Thus
  it appears necessary to consider a wider class of extensions that, on one hand,  would be close enough to principal comodule algebras, yet general enough to include examples coming from  quantum homogeneous spaces. The basic idea is to replace
 a Hopf
algebra in a Hopf-Galois extension by a coalgebra. This point of view for the first time was taken seriously in \cite{bm98a}, where the studies of {\em coalgebra principal bundles} were initiated. Over the recent years and in a significant number of papers, the theory of coalgebra principal bundles or coalgebra-Galois extensions \cite{bh99} has been developed and refined both in purely algebraic and differential geometric directions. On the algebraic side it has led to revival of the coring theory and provided new points of view on areas such as noncommutative descent theory~\cite{bw03}. On the differential geometric side, it has culminated in the introduction of {\em principal extensions} as noncommutative objects most closely describing principal bundles, and in the development of Chern-Weil theory for such extensions~\cite{bh04}. Most importantly, the abstract theory of principal extensions generalising principal comodule algebras was supported by new interesting examples such as noncommutative or quantum instanton bundles going beyond Hopf-Galois theory. 

It seems that the theory of coalgebra-Galois and principal extensions has achieved a  level of maturity at which  it could be profitable to review recent progress  and present it in a unified manner. This is the aim of the current article. The article consists of two parts. In the first part we analyse the algebraic side of coalgebra-Galois extensions. We give basic definitions and properties, we look at dual ways of defining Galois-type extensions (by algebras or by coalgebras), we also put Galois-type extensions in a wider framework of corings and quantum groupoids. 

The second part is devoted to geometry motivated aspects of the coalgebra-Galois theory. In particular, we define modules associated to Galois-type extensions via corepresentations of their structure comodule coalgebras. They can be understood as modules of sections of associated noncommutative vector bundles. We describe basic elements of the theory of connections and strong connections, and derive consequences of the definition of a principal extension. The key idea here is that the concept of equivariant projectivity
replaces that of faithful flatness used in Hopf-Galois theory. These two
concepts are equivalent in the Hopf-Galois setting (bijective antipode assumed)
but only the implication
``equivariant projectivity" $\Rightarrow$ ``faithfull flatness" is known in general.  Therefore, we build our theory on equivariant projectivity which guarantees that the aforementioned associated modules are finitely generated projective  for any finite-dimensional corepresentation of the structure
coalgebra. This way we arrive at the $K$-theory of the coaction-invariant
subalgebra. Now, we can apply the noncommutative Chern character mapping
the $K_0$-group  to the even cyclic homology. 

 Furthermore, strong connections give explicit formulae for idempotents. Although these formulae depend on the choice of  strong connections, corresponding elements of the $K_0$-group  are connection independent. Thus
 we obtain an explicit map from the Grothendieck group of isomorphism classes of finite-dimensional corepresentations of the structure coalgebra to the even cyclic homology of the coaction-invariant subalgebra. We call it the Chern-Galois character, and view as noncommutative Chern-Weil theory.

\subsection{General conventions and standing assumptions}

All (co)algebras are (co)unital and over a field $k$. We use the standard Heynemann-Sweedler notation (with the summation symbol suppressed) for coproducts and coactions, and $*$ for the convolution product of maps from
a coalgebra to an algebra. The coproduct, counit, multiplication,
and antipode are denoted by $\hD$, $\he$, $m$,
and $S$, respectively. The kernel of the multiplication map $A\otimes A\rightarrow A$
is written as $\hO^1A$, and called the space of universal differential 1-forms.
The formula $\d a:=1\otimes a-a\otimes 1$ defines the universal differential
$A\rightarrow\hO^1A$.

Our typical notation for a left and a right coaction on a vector space $V$ is ${}_V\hD$
and $\hD_V$, or ${}^V\!\varrho$  and $\varrho^V$, respectively. For actions on $V$, we
use symbols like $\mu_V$ or $m_V$.
For an algebra $B$  and a coalgebra $C$, the symbol ${}_B\Mc^C$ stands for the
category of left $B$-modules that are also right $C$-comodules
with $B$-linear coactions. Morphisms in  ${}_B\Mc^C$ are left
$B$-linear right $C$-colinear maps. The space of all colinear
homomorphisms is denoted by $\Hom^C$. Analogous symbols denote other
categories of left (co)modules right (co)modules with the left and right 
structures being compatible and other homomorphism spaces.

\subsection{Equivariant projectivity}

The notion of \textit{equivariant projectivity} of a (left)
$B$-module $P$ occurs whenever $P$ has additional algebraic
structure, compatible with the $B$-module structure. In this case
we might like to require the properties of projectivity (such as
the splitting of the product map) to respect this additional
structure. A typical situation of key importance to the theory of
principal extensions can be described as follows.

 As in \cite{bh04}, an
object  $P\in {}_B\Mc^C$ is called a {\em $C$-equivariantly
projective left
 $B$-module} if for any two objects $M,N$ and
morphisms $\pi:M \ra N$, $f:P \ra N$ in ${}_B\Mc^C$, together with
a right $C$-colinear splitting $i: N \ra M$ of $\pi$ there exists
a morphism $g:P \ra M$ in ${}_B\Mc^C$ such that the following
diagram commutes:
\[
\xymatrix{
M \ar@<1ex>[r]^{\pi} & N \ar@<1ex>[l]^{i} \\
& P \ar@{-->}[lu]^{\exists g} \ar[u]^{f}
}
\]
Similarly to projective modules, the $C$-equivariant projectivity
can be fully characterised by the splitting property of the
multiplication map.
\begin{Lem}\label{2:lemma.e.p}
An object $P\in {}_B\Mc^C$ is a $C$-equivariantly projective left
 $B$-module if and only if there exists a left $B$-module right $C$-comodule
 section $s$ of the product map $B\otimes P\ra P$. Here $B\otimes P$ is a right
 $C$-comodule with the tensor product coaction $\id_B\otimes \Delta_P$.
 \end{Lem}
\begin{proof}
Given a section $s$ of the multiplication map $m_P:B\otimes P\ra
P$, and $M$, $N$, $f$, $i$ and $\pi$ as in the diagram above, one
defines the map $g: P\ra M$ by $g =m_M\circ(\id_B\otimes (i\circ
f))\circ s$, where $m_M: B\otimes M\ra M$ is the
$B$-multiplication map for $M$.
 Conversely,
in the defining diagram of a $C$-equivariantly projective
$B$-module $P$ take $M=B\otimes P$, $N=P$, $\pi = m_P$, $i: P\ra
B\otimes P$, $p\mapsto 1_B\ot  p$ and $f$ the identity map. Then
$g$ constructed through such diagram is the required splitting of
the multiplication map.
\end{proof}

In an analogous way, one calls a
 $(B,A)$-bimodule $P$  an {\em $A$-equivariantly
projective left
 $B$-module} if for any two $(B,A)$-bimodules $M,N$ and
$(B,A)$-bilinear maps $\pi:M \ra N$, $f:P \ra N$ in ${}_B\Mc_A$,
together with a right $A$-linear splitting $i: N \ra M$ of $\pi$
there exists a $(B,A)$-bilinear map $g:P \ra M$ such that
$\pi\circ g =f$. This is equivalent to the existence of a
$(B,A)$-bilinear splitting of the multiplication map $B\otimes
P\ra P$.

Since any right $C$-comodule is a left module of the convolution
algebra $C^*$, any object $P\in {}_B\Mc^C$ is a $(B,A)$-bimodule,
where $A=C^{*op}$. In this case, $P$ is a $C$-equivariantly
projective left $B$-module if and only if it is an
$A$-equivariantly projective left $B$-module (since there is a
bijective correspondence between $C$-colinear and $A$-linear
maps).

The notion of equivariant projectivity should be contrasted with
that of  {\em relative projectivity}\index{relatively projective
module}. Given an algebra map $\iota: A\ra B$, any left $B$-module
is also a left $A$-module via $\iota$ and the multiplication in
$B$. In this situation, one often says that $B$ is an {\em
$A$-ring}\index{$A$-ring}  or an {\em algebra over
$A$}\index{algebra over $A$} and that $P$ is a module over an
$A$-ring. The product map $B\otimes P\ra P$ descends to the map
$m_{P\mid A}:B\otimes_A P$. $P$ is called an {\em $A$-relatively
projective left $B$-module} provided the map $m_{P\mid A}$ has a
left $B$-linear section.

An equivariantly projective left $B$-module (be it
$A$-equivariantly or $C$-equivariantly) is always a projective
left $B$-module (a $(B,A)$-linear splitting of the multiplication
map is, in particular, left $B$-linear). Not every
$(B,A)$-bimodule $P$ that is projective as a left $B$-module is an
equivariantly projective module. For an $A$-ring $B$, a projective
left $B$-module is always an $A$-relatively projective left
$B$-module, but the relative projectivity of $P$ does not imply
the projectivity of $P$ (however, when $A$ is a separable algebra
the $A$-relative projectivity is equivalent to the projectivity of
$P$).

\section{Galois-type extensions and coextensions}

This section is devoted to the definition and description of basic algebraic properties of general Galois-type extensions. We start in Section~\ref{2:subs.def} by introducing the notion of equivariant projectivity, then give the definition of coalgebra-Galois extensions and two other types of algebra-Galois (co)extensions. Every such extension is determined by the existence of a {\em (co)translation map}, the properties of which are studied in Section~\ref{2:subs.trans}. Furthemore, any coalgebra-Galois extension or an algebra-Galois coextension gives rise to an algebraic structure, which encodes the symmetries of extension and is known as an {\em entwining structure}. This is closely related (by semi-dualisation) to {\em factorisation} of algebras. Both are described in Section~\ref{2:subs.ent.fac}.  Section~\ref{2:subs.princ} is devoted to the definition of a {\em principal extension} \cite{bh04}   which generalises the concept of a faithfully flat Hopf-Galois extension with bijective antipode and forms a cornerstone of  the theory of noncommutative principal bundles. Representations of entwining structures are given in terms of {\em entwined modules}. These unify many categories of modules studied previously in Hopf algebra theory. Rudimentary properties of entwined modules are described in Section~\ref{2:subs.coring}. In this section it is also shown, how the properties of such modules and Galois-type extensions can be derived from the properties of corings and their comodules. The latter provide a conceptual and algebraic framework for Galois-type extensions.

\subsection{Definitions and basic properties}\label{2:subs.def}

\subsubsection{Coalgebra-Galois extensions}

Let $C$ be a coalgebra over a
field $k$ and $P$ a $k$-algebra
and right $C$-comodule with a comodule structure
map $\hD_P:P \ra P \otimes C$.  In attempting to define a
coalgebra-Galois extension one first has to address the problem of 
defining the coaction invariants.

Recall that for Hopf-Galois extensions coinvariant
elements are defined as $p \in P$ such
that $\hD_P(p)=p \otimes 1$, using the fact
that the unit of a Hopf algebra is group-like.
Since there might not necessarily exist such a
group-like element in the coalgebra
$C$, we can no longer obtain coaction invariants of a
$C$-comodule $P$ in this way.
Instead, we define the  coaction invariants of $P$ by\footnote{
We owe this definition to M.\ Takeuchi.}
\begin{equation}
P^{coC}:=\{b \in P \mid \forall \  p \in P: \hD_P(bp)=b \hD_P(p)\}.
\end{equation}
First observe that $P^{coC}$ is a subalgebra of $P$. Indeed,
for all $b, b^{\prime} \in P^{coC}$ and $p \in P$, 
\begin{equation}
\Delta_P(b b^{\prime} p)
=b \Delta_P(b^{\prime} p)=b b^{\prime} \Delta_P(p).
\end{equation}
Thus $bb'\in P^{coC}$, and since $1\in P^{coC}$, we conclude
that $P^{coC}$ is a subalgebra of~$P$.

Another, and perhaps more intuitive, definition of coaction invariants is
possible, if 
 there  exists a group-like element $e$ in the coalgebra $C$
such that
$\Delta_P(1)=1 \otimes e$. (We call coactions enjoying this property $e$-coaugmented.)
Then  one
 can define the set of {\em $e$-coaction invariants} as
\begin{equation}
P_{e}^{coC}:=\{p \in P \mid \Delta_P(p)=p \otimes e\}.
\end{equation}
Note, however, that it is not always true that
 $P_{e}^{coC}$ is a subalgebra of $P$, although it is a subset of $P$
which contains 1.  These
two types of coaction invariants are related in the following way.

\begin{Lem} \label{2:coinvariants1}
Let $C$ be a coalgebra with a group-like element $e$, and let $P$ be an
algebra and a
right $C$-comodule such that $\Delta_P(1) = 1\otimes e$.
Then $P^{coC} \subseteq P_{e}^{coC}$.
\end{Lem}
\begin{proof}
If $b \in P^{coC}$, then
$\Delta_P(b)=\Delta_P(b \cdot 1)=b \Delta_P(1)=b \cdot (1 \otimes e)=b
\otimes e$, 
i.e.~$b \in P_e^{coC}$.
\end{proof}

Although this is not immediately apparent, both definitions of
coaction invariants  are related to a group-like element. This is,
however, not a group-like element in $C$ but a group-like element in
$P\otimes C$, understood as a coalgebra over $P$ or a {\em coring}. More
information about
corings is given below, and the role of group-like elements is explained
in Remark~\ref{2:rem.grouplike} (cf.\ Proposition~\ref{2:prop.ecoin=coin}).

We call an extension of algebras $B\inc P$ a {\em $C$-extension} if
$B=P^{coC}$. 
The definition of $P^{coC}$ immediately
implies that the coaction of a right $C$-comodule $P$
is a left $P^{coC}$-linear map.
This observation 
allows us to define when a coaction of a coalgebra on an algebra is Galois,
and thus to
generalise the notion of a Hopf-Galois extension.

\begin{Def}[\cite{bh99}]\label{2:def.cgalois}
Let $C$ be a coalgebra and $B\inc P$  a $C$-extension of algebras. We call
the left $P$-module and
right $C$-comodule homomorphism
\begin{equation}
can:P \otimes_B P \lra P \otimes C, \;\;\;  p \otimes p^{\prime} \longmapsto p
\Delta_P(p^{\prime}),
\end{equation}
the canonical map of the $C$-extension $B \subseteq P$. We say that this extension
is a {\em coalgebra-Galois extension}  if the canonical map  is bijective. Furthermore, if
there exists a group-like element $e$ such that $\Delta_P(1)=1\otimes e$, we call $B\inc P$
an {\em $e$-coaugmented \cge}.
\end{Def}
A straightforward generalisation of \cite{h-pm96} provides us with 
an alternative definition of a coalgebra-Galois extension.
\begin{Prop}\label{2:diagram}
Let $C$ be a coalgebra and $B\inc P$  a $C$-extension of algebras. The extension
is a coalgebra-Galois extension {\em if and only if} the following sequence is exact:
\begin{equation}\label{2:sequence}
0\lra P(\hO^1B)P\lra \hO^1P\st{\widetilde{can}|_{\Omega^1P}}{\lra}P\ot C^+
\lra 0.
\end{equation}
Here $\widetilde{can}:P\ot P\ra P\ot_B P\st{can}{\ra}P\ot C$
 is the natural lifting of the canonical map, and $C^+:=\ker\,\he$ is the
augmentation ideal of~$C$.
\end{Prop}
\begin{proof}
 Consider first the following commutative diagram (of left $P$-modules)
with exact rows and columns:
\[
\xymatrix{
0  \ar[r]  & \ker\,\widetilde{can}|_{\Omega^1P}  \ar[r] \ar[d] & \ker\,\widetilde{can}  \ar[r] \ar[d] &  0 \ar[d]\\
0  \ar[r] & \hO^1P  \ar[r]  \ar[d]_{\widetilde{can}|_{\Omega^1P}} & P\ot P  \ar[r]^{m}  \ar[d]_{\widetilde{can}} & P \ar[r] \ar[d] & 0 \\ 
0  \ar[r] & P\ot C^+  \ar[r] \ar[d] & P\ot C  \ar[r]^{\rm{id}\ot\epsilon} \ar[d] & P  \ar[r] \ar[d] & 0 \\ 
 & \coker\,\widetilde{can}|_{\Omega^1P}  \ar[r] & \coker\,\widetilde{can}  \ar[r] & 0  \ar[r] & 0.\\
}
\]
\noindent
Applying the Snake Lemma to the above diagram we obtain the exact sequence
\begin{equation}\label{2:s1}
0\lra\mbox{Ker}\,\widetilde{can}|_{\Omega^1P}\lra\mbox{Ker}\,\widetilde{can}\lra 0
\lra\mbox{Coker}\,\widetilde{can}|_{\Omega^1P}\lra\mbox{Coker}\,\widetilde{can}\lra 0.
\end{equation}
It follows from  the exactness of this sequence that
\begin{equation}\label{2:tilde}
\mbox{Ker}\,\widetilde{can}|_{\Omega^1P}=\mbox{Ker}\,\widetilde{can},\;\;\;
\mbox{Coker}\,\widetilde{can}|_{\Omega^1P}=\mbox{Coker}\,\widetilde{can}.
\end{equation}
On the other hand, the Snake Lemma applied to


\[
\xymatrix{
0 \ar[r] & P(\hO\sp1B)P \ar[r] \ar[d] &\mbox{Ker}\,\widetilde{can} \ar[r] \ar[d] &\mbox{Ker}\,can \ar[d] \\
0 \ar[r] & P(\hO\sp1B)P \ar[r] \ar[d] & P\ot P \ar[r] \ar[d]_{\widetilde{can}} & P\ot\sb B P \ar[r] \ar[d]_{can} & 0\\
0 \ar[r] & 0 \ar[r] \ar[d] & P\ot C \ar[r] \ar[d] & P\ot C \ar[r] \ar[d] & 0\\
& 0 \ar[r]  &\mbox{Coker}\,\widetilde{can} \ar[r] &\mbox{Coker}\,can \ar[r] & 0
}
\]

\noindent
yields the following exact sequence:
\begin{equation}\label{2:s3}
0\ra P(\hO\sp1\! B)P\ra\mbox{Ker}\,\widetilde{can}\ra\mbox{Ker}\,can\ra 0\ra
\mbox{Coker}\,\widetilde{can}\ra\mbox{Coker}\,can\ra 0.
\end{equation}

Assume now that $B\inc P$ is a coalgebra-Galois $C$-extension . Then 
$\mbox{Ker}\,can=0=\mbox{Coker}\,can$, and, from the exactness of 
(\ref{2:s3}), we can infer that $\mbox{Coker}\,\widetilde{can}=0$ 
and $\mbox{Ker}\,\widetilde{can}=P(\hO\sp1\! B)P$. 
Combining this with (\ref{2:tilde}), we conclude that (\ref{2:sequence}) is exact.

Conversely, assume that the sequence (\ref{2:sequence}) is exact. Then 
$\mbox{Ker}\,\widetilde{can}=\mbox{Ker}\,\widetilde{can}|_{\Omega^1P}=P(\hO\sp1\! B)P$,
and $\mbox{Coker}\,\widetilde{can}=\mbox{Coker}\,\widetilde{can}|_{\Omega^1P}=0\,$. 
Consequently, again due to the exactness of
(\ref{2:s3}), we have that $\mbox{Ker}\,can=0=\mbox{Coker}\,can\,$,
i.e.~$B\inc P$ is a coalgebra-Galois extension.
\end{proof}

Let $X$ and $X'$ be total spaces of principal bundles with the same base and structure group.
Recall that any map $X\ra X'$ inducing identity on the base and commuting with the group
action has to be bijective. We  end this section with a coalgebra-Galois incarnation of this fact.
It is a straightforward generalisation of \cite{s-hj90a}.
\begin{Lem}\label{2:311}
Let $P$ and $P'$ be  \cge s of $B$, and let $P'$ be right \ff\ over $B$.
Then any left $B$-linear right $C$-colinear map  $F:P\ra P'$ is an
isomorphism.
\end{Lem}
\begin{proof}
Consider $P'$ as a right module over $P$ via $F$. The composition
\[
P'\ot_PP\ot_BP\lra P'\ot_BP\st{\id\ot F}{\lra}
P'\ot_BP'\st{can'}{\lra}P'\ot C\lra P'\ot_PP\ot C
\]
coincides with 
\[
\id\ot_P can:P'\ot_PP\ot_BP\lra P'\ot_PP\ot C,
\]
where $can$ and $can'$ are the respective canonical maps.
Hence $\id\ot_BF:P'\ot_BP\ra P'\ot_BP'$ is an isomorphism. Therefore, so is
$F$ by the right faithful flatness of $P'$ over $B$.
\end{proof}

\begin{Rem}
As a special case of coalgebra-Galois extensions, obtained
by replacing Hopf algebras in Hopf-Galois extensions
by braided groups one can consider braided Hopf-Galois
 extensions. These provide an intermediate step in between the
$H$- and $C$-Galois,  and allows one to develop a
braided group gauge theory \cite{m-s97}.
\end{Rem}

\subsubsection{Quotient-coalgebra and homogeneous Galois extensions}\label{2:qspaces}

Though it is demonstrated in the previous and following sections that one can get away with
the lack of a group-like element in defining and developing some general aspects
of coalgebra-Galois theory, throughout this section all extensions will be coaugmented
by some group-like element~$e$. The reason is that we are not aware of interesting
examples of non-coaugmented extensions, and co-augmentation seems 
indispensable to prove some of the desired technical results.

On the other hand, a very interesting class of examples comes from the theory of
Hopf-algebra quotients that is elaborated in \cite{ss}, and that already in 1990
gave birth to coalgebra-Galois theory~\cite{s-hj90a}. The setting is as follows.
Let $H$ be a Hopf algebra, $P$ be a right $H$-comodule algebra and $I$ a right
ideal coideal of $H$. Then the composite map
\begin{equation}
P\st{\Delta_P}{\lra}P\otimes H\lra P\otimes (H/I)
\end{equation}
defines a right coaction of the quotient coalgebra $H/I$ on~$P$. Demanding
this coaction to be Galois defines a $\bar{1}$-coaugmented coalgebra-Galois
$H/I$-extension~\cite{s-hj90a}. (Here $\bar{1}$ is the class of $1$
in $H/I$.)
Thus the coaugmentation of such extensions comes automatically
from the Hopf-algebra symmetry that is fundamental in this definition.
We call such extensions {\em quotient-coalgebra Galois extensions}.

The above construction is parallel to what happens in differential geometry. Let us
explain it on the example of the principal instanton bundle $S^7\ra S^4$. The sphere
$S^7$ is a homogeneous space of $SU(4)$. Viewing $SU(2)$ as a block-diagonal
subgroup of $SU(4)$ gives an action of $SU(2)$ on $S^7$ that defines the principal
instanton fibration: $S^7/SU(2)\cong S^4$. The most sophisticated example of
a quotient-coalgebra Galois extension that we know of is a noncommutative deformation
of the instanton bundle~\cite{bcdt04}. 
Here one starts with the Soibelman-Vaksman quantum sphere $S^7_q$
\cite{sv90}, which is a homogeneous space of $SU_q(4)$, and then, following the insight
given by  Poisson geometry \cite{bct02}, one constructs a coideal right ideal $I$ of
the Hopf algebra $\sO(SU_q(4))$ such that the canonical surjection 
$\sO(SU_q(4))\ra \sO(SU_q(4))/I$ corresponds to the block-diagonal inclusion of $SU(2)$
in $SU(4)$ and the induced coaction is Galois yielding $\sO (S^4_q)$ as the coaction invariant
subalgebra~\cite{bcdt04}.

One of the reasons why this example is interesting is that it uses the full generality
of quotient-coalgebra Galois extensions, i.e., we have $P\neq H$ and $I\neq 0$. Observe
that for $I=0$ we recover as a special case Hopf-Galois theory, whereas for  $P= H$
we obtain what is called {\em homogeneous coalgebra-Galois extensions}. We devote
the remainder of this section to the latter case. This is the case which deals with 
quantum homogeneous spaces or left coideal subalgebras (thus justifying the name ``homogeneous coalgebra-Galois extension"). The aim is to try and reproduce in the general noncommutative setting a classical construction in which a homogeneous space $M$ of a group $G$ is viewed as a base for a principal bundle with the total space $G$.

Let $P$ be a Hopf algebra and $I$ a coideal right ideal of $P$, so that $P/I$ is a coalgebra and a right $P$-module. View $P$ as a right $P/I$-comodule via the induced coaction
\[
\hD_P:=(\id\ot\pi_I)\ci\hD,\;\;\; P\st{\pi_I}{\lra}P/I.
\]
 The corresponding $P/I$-extension $B\inc P$ is called   a {\em homogeneous}
  $P/I$-extension. 
The importance of extensions of this type stems from the fact that $B$ is a {\em quantum homogeneous space} or a left coideal subalgebra of $B$. Let us disucss this in more detail.

Since $1$ is a group-like element in a Hopf algebra $P$, its coalgebra projection $\pi_I(1)$ 
is a group-like element in $P/I$. Furthermore, $\Delta_P(1) = 1\otimes \pi_I(1)$. Observe then
that for a homogeneous $P/I$-extension, the coaction-invariant subalgebra
 $B$ is equal to the
subalgebra of  $\pi_I(1)$-coaction invariants
 $P^{P/I}_{\pi_I(1)} = \{ b\in P\; | \; \Delta_P(b) = 
b \otimes \pi_I(1)\}$. Indeed, $B\inc P^{P/I}_{\pi_I(1)}$ by Lemma~\ref{2:coinvariants1}. Conversely,
if $\Delta_P(b) = b \otimes \pi_I(1)$, then, for all $p\in P$,
\begin{eqnarray*}
\Delta_P(bp) &=& b\sw 1p\sw 1 \ot \pi_I(b\sw 2 p\sw 2) = b\sw 1p\sw 1 \ot \pi_I(b\sw 2) p\sw 2\\
& =&
 bp\sw 1 \ot \pi_I(1) p\sw 2 = bp\sw 1 \ot \pi_I( p\sw 2) = b\Delta_P(p),
\end{eqnarray*}
where we have used the fact that $\pi_I$ is a right $P$-module map. Thus, $b\in B$ as required. 
Next, using the coassociativity of the coaction $\Delta_P$ and the description of $B$ as $\pi_I(1)$-coaction invariants, apply $\Delta\ot \id$ to  equation $b\sw 1 \ot \pi_I(b\sw 2) = b\ot \pi_I(1)$ to deduce that  for all $b\in B$,
$
((\id\ot\Delta_P)\circ \Delta)(b) = \Delta(b)\ot \pi_I(1).
$
This implies that 
\begin{equation}
\forall b\in B, \qquad b\sw 1\otimes
b\sw 2 \in P\otimes B,
\label{2:homg.eq}
\end{equation}
 i.e., $\Delta(B)\inc P\ot B$, so that $B$ is a left coideal subalgebra of $P$ or a quantum homogeneous space of $P$.

Thus homogeneous $P/I$-extensions provide one with a suitable set-up for principal bundles over quantum homogeneous spaces. To exploit this fully, however, we need to address a question when a homogeneous $P/I$-extension is a coalgebra-Galois extension. The answer turns out to determine the structure of $I$ completely (cf.~Lemma~5.2 in \cite{bm93}). 

\begin{Thm}\label{2:b+p}
Let $B\inc P$  be a homogeneous  $P/I$-extension. 
 Then this extension is Galois
{\em if and only if} $I=B^+P$, where $
B^+:=B\cap\mbox{\em Ker}\,\he$.
\end{Thm}
\begin{proof}
Assume first that $I=B^+P$. Taking advantage of (\ref{2:homg.eq}), for any
$b\in B^+$, $p\in P$, we compute:
\begin{equation}\label{2:sbp}
S(b\sw1 p\sw1)\ot_Bb\sw2 p\sw2=S(b\sw1 p\sw1)b\sw2\ot_Bp\sw2=S(p\sw1)\he(b)\ot_Bp\sw2=0.
\end{equation}
Hence there is a well-defined map 
\[
T:P\ot(P/I)\lra P\ot_BP,\;\;\;
T(p\ot [p']_I):=pS(p'\sw1)\ot_Bp'\sw2.
\]
 It is straightforward to verify
that $T$ is the inverse of the canonical map $\can$. Consequently,
$P$ is a coalgebra-Galois $P/I$-extension.

To show the converse, let us first prove the following:

\begin{Lem}\label{2:homo}
Let $P$, $I$ and $B$ be as above. Then $B\inc P$  is a 
coalgebra-Galois $P/I$-extension
{\em if and only if}
$\llp\pi\sb B\circ(S\ot \id )\circ\hD\lrp(I)=0$, where 
$\pi\sb B:P\ot P\ra P\ot\sb B P$ is the canonical surjection.
\end{Lem}
\begin{proof}
By Proposition~\ref{2:diagram}, $P$ is a coalgebra-Galois $P/I$-extension of $B$
if and  only if the following   sequence 
\begin{equation}\label{2:seq}
0\lra P(\hO\sp1\! B)P\lra
P\ot P\st{\widetilde{can}}{\lra}P\ot P\slash I\lra 0
\end{equation}
is exact. One can check
that $\llp \widetilde{can}\circ(S\ot \id )\circ\hD\lrp(I)=0$. Hence, it follows from 
the exactness of (\ref{2:seq}) that 
$\llp (S\ot \id )\circ\hD\lrp(I)\inc P(\hO\sp1\! B)P$. Consequently,
$\llp \pi\sb B\circ(S\ot \id )\circ\hD\lrp(I)=0$ due to the exactness of 
the sequence
\[
0\lra P(\hO\sp1\! B)P\lra
P\ot P\st{\pi\sb B}{\lra}P\ot\sb B P\lra 0 .
\]
To prove the converse, one can proceed as in the considerations 
preceding this lemma.
\end{proof}

\begin{Cor}\label{2:bpco}
Let $B\inc P$ be a coalgebra-Galois $P/I$-extension
 as above. Then the translation map $\tau:=can^{-1}(1\otimes \cdot)$
is given by the formula: $\tau([p]_I):=S(p\sw1)\ot_Bp\sw2\,$.
\end{Cor}

Assume now that $P$ is a coalgebra-Galois $P/I$-extension of $B$. 
It follows from the above
corollary and (\ref{2:sbp}) that $\tau([B^+P]_I)=0$. 
Hence, by the injectivity
of $\tau$, we have $B^+P\inc I$. Furthermore, there is a well-defined
map 
\[
can':P\ot_BP\lra P\ot(P/B^+P),\;\;\; p\ot_Bp'\longmapsto pp'\sw1\ot [p'\sw2]_{B^+P}\, .
\]
Indeed, taking again  advantage of (\ref{2:homg.eq}), we obtain
\begin{eqnarray}
p\ot bp'&\mapsto& pb\sw1 p'\sw1\ot [(b\sw2-\he(b\sw2))p'\sw2+\he(b\sw2)p'\sw2]_{B^+P}\nonumber\\
&=& pb\sw1 p'\sw1\ot\he(b\sw2)[p'\sw2]_{B^+P}\nonumber\\
&=& pbp'\sw1\ot [p'\sw2]_{B^+P}\, .
\end{eqnarray}
On the other hand,
 $pb\ot p'\mapsto pbp'\sw1\ot [p'\sw2]_{B^+P}\,$. Reasoning as in the first part
of the proof, we can conclude that $can'$ is bijective. Next, consider the following
commutative diagram:
\[
\begin{array}{ll}
P\ot_BP  & \mathop{-\hspace{-6pt}\longrightarrow}\limits^{can'} P\ot(P/B^+P)  \\

~\, ^{\id } \Big\downarrow  &  \,~~~~~~~~~~\Big\downarrow\ ^{\id \ot\imath} \\ 

P\ot_BP & \mathop{-\hspace{-6pt}\longrightarrow}\limits^{can} P\ot(P/I)\, , \\
\end{array}
\]
where $\imath([p]_{B^+P}):=[p]_{I}$. (Recall that we have already showed that
$B^+P\inc I$, so that $\imath$ is well-defined.) 
It follows from the commutativity of the diagram that $\id \ot\imath$
is bijective. Hence, as the tensor product is over a field, also $\imath$
is bijective.
In particular, $\imath$ is injective, and therefore
$I\inc B^+P$, as needed.
\end{proof}

From the discussion preceding  Theorem~\ref{2:b+p}, we know that homogeneous 
coalgebra-Galois extensions give rise to quantum homogeneous spaces. On the other 
hand, given a left coideal subalgebra $B$ of a Hopf algebra $P$, one can ask 
when $B\inc P$ is a homogeneous coalgebra-Galois $P/I$-extension for a suitable 
coideal $I$. By Theorem~\ref{2:b+p}, the ideal  $I$ must be of the form $B^+P$, so 
this question makes sense, provided $B^+P$ is a coideal for any left coideal subalgebra 
$B$. This is the case indeed, since
 for a typical
 element $x=bp \in I$, i.e., $p \in P$ and
$b \in B$ with $\varepsilon(b)=0$, we obtain
\begin{eqnarray*}
\Delta(bp)&=& b\sw 1 p\sw 1 \otimes b\sw 2 p\sw 2 =
 b\sw 1 p\sw 1 \otimes b\sw 2 p\sw 2 - bp\sw 1 \otimes p\sw 2 + bp\sw 1
\otimes p\sw 2\\
&=&
 b\sw 1 p\sw 1 \otimes (b\sw 2 - \varepsilon(b\sw 2))p\sw 2 + bp\sw 1
\otimes p\sw 2 
\in P \otimes I + I \otimes P.
\end{eqnarray*}
Now, take any $b\in B$ and $p\in P$ and, using the fact that $B$ is a left coideal subalgebra of $P$, so that equation (\ref{2:homg.eq}) holds,
compute
$$
\Delta_P(bp) = b\sw 1p\sw 1\otimes\pi_I(b\sw 2p\sw 2) = b\sw 1p\sw 1\otimes
\varepsilon(b\sw 2)\pi_I(p\sw 2) = b\dr(p).
$$
Therefore $B\subseteq P^{coP/I}$. In view of Theorem~\ref{2:b+p}, the sufficient and necessary condition for a quantum homogeneous space to be a base of a homogeneous coalgebra-Galois extension is that $B\inc P^{coP/I}$. This
is the case if, for example, $P$ is a faithfully flat left or right $B$-module (cf.\  
\cite[Theorem~1]{t-m79}, \cite[Lemma~1.3]{s-hj92}).

\subsubsection{Algebra-Galois coextensions}

The dual version of  Hopf-Galois extensions can be viewed as a direct 
noncommutative generalisation of the theory of quotients 
of formal schemes under free actions of formal group schemes
(cf.~\cite{s-hj90a}). Spaces are replaced by coalgebras and their fibred products
by  cotensor products. Groups are turned into group rings and replaced by
algebras. Thus there is no dualisation involved  in this generalisation
of group actions on spaces. The ``noncommutativity of a space" is encoded
in the non-cocommutativity of a coalgebra replacing it. Roughly speaking, we treat
the classical space as a coalgebra with the trivial coproduct. Much as for
group rings, the underlying vector space is spanned by the elements of the
space thought of as basis vectors. The coalgebra structure on this vector space
is defined by declaring the basis vectors to be group-like:
\[
C_X=\bigoplus_{x\in X}kx,\;\;\;\hD x=x\ot x,\;\;\;\he(x)=1.
\]
This is what we mean by a {\em classical-space coalgebra}. If $X$ is also a 
$G$-space, i.e., there is an action $X\times G\ra X$, then we have the induced
action
$C_X\ot kG\ra C_X$. 

A step beyond classical spaces is to consider cocommutative
coalgebras that do not admit a basis whose all elements are group-like. As an
example, take the coalgebra dual to the algebra $A:=\C[\theta]/\<\theta^2\>$
of dual numbers. Denote by $\{1^*,\theta^*\}$ the dual basis of $A^*$. Then
the coalgebra structure on $A^*$ is given by the formulae:
\[
\hD 1^*=1^*\ot 1^*,\;\;\;\hD\theta^*=1^*\ot\theta^*+\theta^*\ot 1^*,\;\;\;
\he(1^*)=1,\;\;\;\he(\theta^*)=0.
\]
The aforementioned example belongs to the realm of super geometry:
it is neither classical nor noncommutative.

For a noncommutative example, let us consider the algebra of matrices
\[
M_n(\C)\widetilde{=}\C\< x,y\>/\<x^n-1,\, y^n-1,\, yx-e^{\frac{2\pi i}{n}xy}\>\; .
\]
We can take as generating matrices
\[
x=\left(\begin{array}{ccccc}
0&1&0&\ldots&0\\
\vdots&\vdots&\vdots&&\vdots\\
0&\ldots&&0&1\\
1&0&\ldots&&0\\
\end{array}\right),
\;\;\;
y=\left(\begin{array}{ccccc}
1&0&\ldots&&0\\
0&q&0&\ldots&0\\
\vdots&\vdots&\vdots&&\vdots\\
0&\ldots&&0&q^{n-1}\\
\end{array}\right),
\;\;\; q=e^{\frac{2\pi i}{n}}.
\]
Then $\{ x^ky^l\;|\; k,l\in\{ 0,...,n-1\}\}$ is a linear basis of $M_n(\C)$.
Let $C_n$ denote the dual coalgebra $M_n(\C)^*$. As before, we use the
notation $(x^ky^l)^*$ for the dual basis elements. A direct calculation yields
the coalgebra structure:
\[
\hD ((x^ky^l)^*)=\!\!\!\sum_{p,r,s,t=0}^{n-1}\!\!\!q^{rs}
\hd_{k\,,\,p+s\, \mbox{\scriptsize mod}\, n}\,
\hd_{l\,,\,r+t\, \mbox{\scriptsize mod}\, n}\,
x^py^r\ot x^sy^t,\;\;\;\he((x^ky^l)^*)=\hd_{k,0}\hd_{l,0}.
\]
To make this example more tangible, put $n=2$. Then we have explicitly
\begin{eqnarray}
\hD (1^*)&=&1^*\ot 1^*+x^*\ot x^*+y^*\ot y^*-(xy)^*\ot(xy)^*,\\
\hD (x^*)&=&x^*\ot 1^*+1^*\ot x^*+(xy)^*\ot y^*-y^*\ot(xy)^*,\\
\hD (y^*)&=&y^*\ot 1^*+1^*\ot y^*-(xy)^*\ot x^*+x^*\ot(xy)^*,\\
\hD ((xy)^*)&=&(xy)^*\ot 1^*+x^*\ot y^*-y^*\ot x^*+1^*\ot(xy)^*.
\end{eqnarray}

We can think of a coproduct as a set-theoretical map from the set of basis
elements $X$ to its Cartesian square times the ground field: $X\times X\times k$.
For the classical spaces, this map embeds $X$ in $X\times X\times k$ as the diagonal
in $X\times X$ times $\{1\}$. For the ``commutative spaces", the value of this map
on any element of $X$ is symmetric with respect to the plane
$\{(x,x,\alpha)\in X\times X\times k\;|\; x\in X,\;\alpha\in k\}$.
The noncommutativity of the space is measured by the lack of the aforementioned
symmetry. Thus we can visualise {\em the geometry of noncommutativity}. The finite
spaces fit perfectly this coalgebraic picture (the dual of any finite dimensional
algebra is a coalgebra), but a significant adjustment seems to be required
to accommodate the infinite case, especially if we want to go beyond the discrete
topology. However, this point of view is, hopefully, tenable through some
topological version of the concept of a coalgebra.

Since the coalgebra-Galois extensions are dual to the classical principal
bundles, they are also dual (for each object involved) to the algebra-Galois
coextensions. Therefore, in this section we dualise coalgebra-Galois extensions 
and derive results analogous to the results discussed in the previous 
section. As this dualisation in its full generality is somewhat involved,
it is helpful to consider first
the  definition from \cite[p.\ 3346]{s-hj93}  that
dualises the concept of a \hge.
Let $H$ be a Hopf algebra, $C$ a right $H$-module coalgebra with the 
action $\mu_C:C\otimes H\to C$. Then, since 
the action $\mu_C$ is a coalgebra map, i.e.,  
$\hD\ci\mu_C=(\mu_C\ot\mu_C)\ci(C\ot\mbox{flip}\ot H)\ci(\hD\ot\hD)$,
we have
\begin{eqnarray}
\hD(\mu_C(c,h)-\epsilon(h)c)
&=&
\mu_C(c\sw1,h\sw1)\ot\mu_C(c\sw2,h\sw2)-c\sw1\ot\he(h)c\sw2
\nonumber\\ &-&
\he(h\sw1)c\sw1\ot\mu_C(c\sw2,h\sw2)+\he(h\sw1)c\sw1\ot\mu_C(c\sw2,h\sw2)
\nonumber\\ &=&
\llp\mu_C(c\sw1,h\sw1)-\he(h\sw1)c\sw1\lrp\ot\mu_C(c\sw2,h\sw2)
\nonumber\\ &+&
c\sw1\ot\llp\mu_C(c\sw2,h)-\he(h)c\sw2\lrp,
\end{eqnarray} 
so that  
$I := \{\mu_C(c,h)-\epsilon(h)c\;|\; c\in C, h\in H\} $ is 
a coideal in $C$. Hence $B:=C/I$ is a coalgebra.
Using again the assumption that $\mu_C$ is a coalgebra map, 
it can be directly checked that 
$((C\ot\mu_C)\circ(\Delta\tens H))(C\te H)\inc C\Box_BC$. This way we arrive at:
\begin{Def}\label{2:hgc} 
 We say that 
$C\twoheadrightarrow B$ is a (right) {\em Hopf-Galois   
 $H$-coextension} if the  
canonical left $C$-comodule right $H$-module map 
$ 
cocan := (C\ot\mu_C)\circ(\Delta\tens H):C\te H \ra  
C\Box\sb BC 
$ 
is a bijection. 
\end{Def} 

Now, to obtain a dualisation of the general \cge, we replace
$H$ by an algebra $A$ and remove the condition that the action $\mu_C$ 
is a coalgebra map. At this level of generality, to formulate the definition
of a coextension, we first need:
\begin{Lem}[\cite{bh99}]\label{2:coa}
Let $A$ be an algebra and $C$ a coalgebra and right $A$-module with an action
$\mu_C:C\te A\ra C$. Denote by $I$ the vector space 
\[ 
{\rm span}\{\mu_C(c,a)\sw1\alpha(\mu_C(c,a)\sw2)-c\sw1\alpha(\mu_C(c\sw2,a))\;|
\; a\in A,\, c\in C,\,\alpha\in\mbox{\em Hom}(C,k)\}, 
\]
by $\pi:C\ra C/I$ the canonical surjection, and put $B:=C/I$. Then 
$I$ is a coideal of $C$, the
action $\mu_C$ is left $B$-colinear, i.e., 
$(\pi\te C)\ci\hD\ci\mu_C = (B\te\mu_C)\ci((\pi\te C)\cc\hD\ot A)$,
and
$((C\te\mu_C)\cc(\hD\ot A))(C\ot A)\inc C\Box_BC$.
\end{Lem}
Now we can conclude that we have a well-defined map
\[
cocan := (C\ot\mu_C)\circ(\Delta\tens A)\; :\;\; C\tens A \lra  
C\Box\sb BC, 
\]
and can consider: 
\begin{Def}[\cite{bh99}]\label{2:agc} 
Let $A$ be an algebra,  
$C$ a coalgebra and right $A$-module,  and $B=C/I$, 
where $I$ is the coideal of Lemma~\ref{2:coa}. We say 
that $C$ is a (right)
 {\em algebra-Galois $A$-coextension} 
of $B$ if the canonical left $C$-comodule right $A$-module  
map 
$
cocan := (C\ot\mu_C)\circ(\Delta\tens A):C\tens A \lra  
C\Box\sb BC
$ 
is bijective. An algebra-Galois coextension is said to be {\em $\kappa$-augmented} 
if there exists an algebra map $\kappa: A\to k$ such that
 $\eps_C\circ\mu_C = \eps_C\ot\kappa$.
\end{Def} 
To see more clearly that Definition~\ref{2:agc} dualises the notion of 
a Galois $C$-extension, one can notice that both 
$C\tens A$ and $C\Box\sb BC$ are objects in $\sp C\!\CM\sb A$, which 
is dual to $\sb A\CM\sp C$. The structure maps are 
$\Delta\tens C$, $C\tens m$ and $\Delta\Box\sb BC$, $C\Box\sb B\mu_C$, 
respectively. The canonical map $cocan$ is a morphism in $\sp C\!\CM\sb A$.  
The right $A$-coextension $C\twoheadrightarrow B$ is algebra-Galois if $C\tens A \cong 
C\Box_BC$ as objects in $\sp C\!\CM\sb A$ by the canonical map $cocan$.
(In what follows, we consider only right coextensions and omit ``right" for
brevity.)

Finally, note that for the group actions on spaces translated into
group-ring actions on classical-space coalgebras, $cocan$ can be reduced to
\[
F^G_X:X\times G\lra X\times_{X/G}X,\;\;\; F^G_X((x,g))=(x,xg).
\]
Here $X\times_{X/G} X$ is by definition the image of $F^G_X$. 
The bijectivity of $cocan$ is then equivalent to the bijectivity of $F^G_X$.
This guarantees that the action of $G$ on $X$ is free. However, to arrive
at principal actions, one needs to introduce topology and go beyond the map
$F^G_X$ to guarantee the properness of the action (see the preamble of the (co)translation map section).

\subsubsection{Algebra-Galois extensions}\label{aext}

\vspace*{-11mm}$\qquad\qquad\qquad\qquad\qquad\qquad\qquad\qquad\quad$\footnote{This section is based on joint work with P.~Schauenburg and H.-J.~Schneider.}

~\\
\noindent
The right action of a group $G$ on a space $X$ induces the right action of the
group ring $(kG)^{op}$ on a suitable algebra of functions on $X$:
$(f\act g)(x):=f(xg)$. This model is a prototype of our considerations here. Let
$P$ be an algebra and a right $A$-module via the action $P\ot A\st{\act}{\ra}P$.
Then one can define the invariant subalgebra
\[
P^A=\{ b\in P\;|\; (bp)\act a=b(p\act a),\,\fa p\in P,\, a\in A\}.
\]
If $P$ is a comodule via the map $\hD_P: P\ra P\ot C$ and the action of $(C^*)^{op}$
is given by the formula $p\act a=p\sw0 a(p\sw1)$, then $P^{(C^*)^{op}}=P^{co C}$
(cf.\ \cite{m-s93}). Now, if $A$ is a finite dimensional algebra, then the
pullback of the multiplication and the unit map turns $A^*:=\hom (A,k)$ into
a coalgebra. Similarly, any action $M\ot A\st{\act}{\ra}M$ gives rise to a
coaction 
\[
\rho:M\ra M\ot (A^{op})^*,\;\;\;\rho(m)=\sum_{i=1}^{\dim A}m\act e_i\ot e^i,
\]
where $\{e_i\}_{i\in\{1,...,\dim A\}}$ is a basis of $A$ and 
$\{e^i\}_{i\in\{1,...,\dim A\}}$ the dual basis. In this situation, algebras
and coalgebras as well as modules and comodules are equivalent concepts.
In particular, one directly translates the Galois condition for coactions into
an equivalent Galois condition for actions. This has been carried out in \cite{bm00}.
The aim of this section is to study the case $\dim A=\infty$, so that for
the details concerning  $\dim A<\infty$ we refer to \cite{bm00}. Here let us only
observe the following:
\begin{Rem}
In this remark we follow the convention and notation of \cite{bm00}.
An algebra action $A\ot P\ra P$ cannot be Galois in the sense of
\cite[Proposition~2.2]{bm00} when $\dim A=\infty$. Indeed,
suppose that $\dim A=\infty$ and action is Galois. Choose a linear basis $\{e_i\}$
of $A$, and write $\chi^{\#}(1)=\sum_{i=1}^n t_i\ot e_i$. Then
\begin{eqnarray}
1\ot e_{n+1}
&=&
((\chi\ot \id )\ci(\id \ot\chi^{\#}))(e_{n+1}\ot 1)
\nonumber\\ &=&
\sum_{i=1}^n (\chi\ot \id )(e_{n+1}\ot t_i\ot e_i)
\nonumber\\ &=&
\sum_{i=1}^n (e_{n+1}\triangleright t_i)\ot e_i\nonumber.
\end{eqnarray}
This contradicts  the linear independence of $\{e_i\}$. (In the finite
dimensional case one can take $n=\dim A$, and then there does not exist a linearly 
independent $e_{n+1}$.)
\end{Rem}

In the infinitely dimensional case, the point is that an action cannot always
be turned to a coaction~\cite[p.11]{m-s93}. Therefore, we need a definition of the
Galois property which avoids starting from a  coaction.
\begin{Def}\label{2:def.agalois.act}
 Let $P$ be an algebra and a right $A$-module. Let $V\subseteq A^*$ 
 and $j:P\ot V\ra \hom(A,P)$ be a linear map defined by 
 $j(p\ot v)(a):=pv(a)$.
  We say that
the action of $A$ on $P$ is Galois provided
\begin{rlist}
\item the map
$
Can: P\ot_{P^A}P\ra\hom(A,P)$, $Can(p\ot_{P^A}p')(a):=p(p'\act a)
$
is injective,
\item
there exists a subspace $V\inc A^*$ such that
\begin{blist}
\item $Can(P\ot_{P^A}P)=j(P\ot V)$
 ($V$ is sufficiently small),
\item $V(a)=0\Rightarrow a=0$ ($V$ is sufficiently big).
\end{blist}
\end{rlist}
A Galois action is said to be {\em $\kappa$-augmented} if there exists a character of $A$, $\kappa\in V$, such that $1_P\kappa(a) = 1_P\act a$
for all $a\in A$.
\end{Def}
We say that an extension of algebras $B\inc P$ is an {\em $A$-extension} if 
$B=P^A$, and we call it {\em algebra-Galois} if the action of $A$ on $P$ is Galois. Finally an algebra-Galois extension with a $\kappa$-augmented action is called a {\em $\kappa$-augmented} algebra-Galois extension.
Let us begin by extracting immediate properties of algebra-Galois extensions.
\begin{Lem}\label{2:peter}
Assume that the action $P\ot A\ra P$ satisfies condition (ii)(a) for some $V\inc A^*$.
Let $\gamma:P^*\ot P\ra A^*$ be defined by the formula
$
\gamma(f\ot p)(a)=f(p\act a).
$
Then $V=\im\gamma$.
\end{Lem}
\begin{proof}
Condition (ii)(a) entails that
\[
\fa v\in V\;\exists\,\mbox{$\sum_i$} p_i\ot p'_i\in P\ot P\;\fa a\in A:\; 
v(a)=\mbox{$\sum_i$} p_i(p'_i\act a).
\]
Choose $f\in P^*$ such that $f(1)=1$ and define $f_i$ by $f_i(p)=f(p_ip)$. Then
\begin{eqnarray}
v(a)&=& f(v(a))
\nonumber\\ &=& 
f(\mbox{$\sum_i$} p_i(p'_i\act a))
\nonumber\\ &=&
\mbox{$\sum_i$} f_i(p'_i\act a)
\nonumber\\ &=&
\gamma(\mbox{$\sum_i$} f_i\ot p'_i)(a),
\end{eqnarray}
as needed.
\end{proof}
Thus condition (ii)(a) uniquely determines $V$. The following lemma shows that it
also forces the action of $A$ to be locally finite. On the other hand, condition (ii)(b)
is responsible for the faithfulness of the action of $A$.
\begin{Lem}\label{2:hans}
The Galois action of $A$ on $P$ is always faithful ($P\act a=0\Rightarrow a=0$)
and locally finite ($\dim (p\act A)<\infty$ for any $p\in P$).
\end{Lem}
\begin{proof}
Let us prove first that $\dim (p\act A)<\infty$. It follows from  condition (ii)(a)
that
\[
\fa p\in P\;\exists\,\mbox{$\sum_i$} e_i\ot v_i\in P\ot V\;\fa a\in A:\; 
p\act a=\mbox{$\sum_i$} e_iv_i(a).
\]
Hence $p\act  A\inc \mbox{\rm span}\{ e_i\}_{i\in\mbox{\scriptsize finite set}}$, 
so that 
$\dim (p\act A)<\infty$. Next, 
the faithfulness assertion follows immediately from condition (ii)(b) and 
Lemma~\ref{2:peter}. Indeed, $$P\act a=0\Rightarrow V(a)=0\Rightarrow a=0.$$
\end{proof}
Furthermore, the axioms for $V$ turn out to be sufficiently strong to make it
a coalgebra and $P$ a $V$-comodule reflecting the $A$-module structure.
\begin{Lem}\label{2:peter2}
Assume that the action $P\ot A\ra P$ satisfies conditions 1 and (ii)(a). Then there exists
a coalgebra structure on $V$ and a coaction $\hD_P:P\ra P\ot V$ such that 
$(j\circ\hD_P)(p)(a)=p\act a$ for any $p\in P,\, a\in A$.
Moreover, $P$ is a coalgebra-Galois $V$-extension of $P^A$.
\end{Lem}
\begin{proof}
Since $Can(P\ot_{P^A}P)=j(P\ot V)$ and $j:P\ot V\ra \hom(A,P)$ is injective, there is
a homomorphism 
\[\label{2:cv}
\hD_P:P\lra P\ot V,\;\;\; \hD_P(p):=j^{-1}(Can(1\ot_{P^A}p)).
\]
To define the desired coproduct on $V$, let us consider the relationship between 
$\hD_P$ and $\gamma$. Taking advantage of the natural embedding of $V\ot V$ in
$(A\ot A)^*$, we obtain:
\begin{eqnarray}\label{2:spirit}
\llp(\gamma\ot\id)\circ(\id\ot\hD_P)\lrp(\phi\ot p)(a\ot a')&=&
\gamma(\phi\ot p\act a')(a)
\nonumber\\ &=&
\phi((p\act a')\act a)
\nonumber\\ &=&
\gamma (\phi\ot p)(a'a).
\end{eqnarray}
Consequently, if $t\in\ker\gamma$, then $(\id\ot\hD_P)(t)\in\ker(\gamma\ot\id).$
Therefore, we have a commutative diagram defining the coproduct:
\[
\xymatrix{
0  \ar[r]  & \ker\gamma  \ar[r] \ar[d] & P^*\ot P  \ar[r]^{\gamma} \ar[d]_{\id\ot\Delta_P} & V  \ar[r]  \ar[d]_{\Delta} & 0 \\
0  \ar[r] & \ker\gamma\ot V  \ar[r] & P^*\ot P\ot V  \ar[r]^{\gamma\ot\id} & V\ot V  \ar[r] &\, 0. \\ 
}
\]
In the spirit of (\ref{2:spirit}), we can verify that
\[\label{2:dv}
(\hD v)(a\ot a')=
\sum_i\llp(\gamma\ot\id)\circ(\id\ot\hD_P)\lrp(\phi_i\ot p_i)(a\ot a')=v(a'a).
\]
With the help of the natural embedding $V\ot V\ot V\inc (A\ot A\ot A)^*$, the above
formula entails the coassociativity of \hD. The counit \he\ is given by the 
evaluation 
at 1. The map $\hD_P$ is by construction compatible with the counit, i.e.,
$(\id\ot\he)\circ\hD_P=\id$, and the coassociativity of $\hD_P$ can be proven
the same way as the coassociativity of $\hD$. Thus $V$ is a coalgebra and $P$
is a comodule such that $p\sw0 p\sw1(a)=p\act a$ (see (\ref{2:cv})).
Since (by the injectivity of $j$)
\[
 (bp)\sw0\ot (bp)\sw1=bp\sw0\ot p\sw1\Leftrightarrow
(bp)\act a=b(p\act a), 
\]
we have $P^{coV}=P^A$. Finally, the coaction $\hD_P$ is 
clearly Galois as $can$ is surjective by condition (ii)(a) and is injective due to the 
injectivity of $Can$.
\end{proof}
We have just shown that, if an $A$-extension $B\inc P$ satisfies conditions
(i) and (ii)(a), then there exists a coalgebra $C\inc A^*$ such that $B\inc P$ is
a coalgebra-Galois $C$-extension and $Can=j\circ can$. Hence, behind any 
algebra-Galois extension, there is a coalgebra-Galois extension. On the other
hand, behind any coalgebra-Galois extension there is an algebra-Galois extension ---
just take $A=\csop$ and $V=i(C)$, where $i:C\hookrightarrow C^{**}=A^*$ is the
canonical embedding. The injectivity of $Can$ and condition (ii)(a) follow immediately 
from the definition of the action of \csop\ ($p\act a= p\sw0 a(p\sw1)$), and condition
(ii)(b) is automatic. We now want to employ condition (ii)(b) to show that the aforementioned
procedure of extracting the coalgebra-Galois structure from an algebra-Galois
extension retrieves the original coalgebra-Galois $C$-extension we started from.
\begin{Lem}\label{2:piotr}
 Let $P$ be an algebra and a right $C$-comodule. The $C$-extension
$B\inc P$ is coalgebra-Galois if and only if the induced action ($A:=(C^*)^{op}$,
$p\triangleleft a:=p\sw0a(p\sw1)$) is Galois.
\end{Lem}
\begin{proof}
If the $C$-coaction is Galois,
we take $V=i(C)$ and verify that all works. In the converse direction, assume that the action
of $A$ on $P$ is Galois. Since it is given by the formula $p\act a =p\sw0 a(p\sw1)$,
we have that $P^{co C}\inc P^A$ and 
\[\label{2:cc}
Can=j\circ(\id\ot i)\circ can.
\]
Hence, by condition (ii)(a), 
\[\label{2:im}
P\ot V=\im ((\id\ot i)\circ can)\inc P\ot i(C),
\]
so that $V\inc i(C)$. We also have that
 $can$ is injective due to the injectivity of $Can$.
 Furthermore,
now it follows from (ii)(a) that $i(C)\inc V$. Indeed, otherwise there exists 
$u\in i(C)$, $u\not\in V$. Let $\{c_i\}$ be a basis of $i^{-1}(V)$. Then
$i^{-1}(u)$ is linearly independent of $\{c_i\}$, and we can complete the set
$\{c_i\}\cup\{i^{-1}(u)\}$ to a linear basis of $C$. Using such a basis, define
$a_0$ to be 1 on $i^{-1}(u)$ and 0 on all other basis elements. Then $a_0\neq 0$
and $V(a_0)=a_0(i^{-1}(V))=0$, which contradicts (ii)(a). Therefore $V=i(C)$, and
(\ref{2:im}) implies that $P\ot C$ is the image of $can$. 
Combining this with the injectivity of $can$
 we conclude that $can$ is a bijection
from $P\ot_{P^{co C}}P$ to $P\ot C$, i.e., the extension
$P^{co C}\inc P$ is $C$-Galois.  
\end{proof}
We have shown that a coaction is Galois if and only if the induced action is Galois.
To complete the picture, let us note that the coalgebra structure on $V$ constructed
in Lemma~\ref{2:peter} coincides with the coalgebra structure on $C$. First, it is
straightforward to observe that their coactions on $P$ coincide:

\[\label{2:vc}
\xymatrix{
P  \ar[r] \ar[rd] & P\ot V \\
&  P\ot C. \ar[u]_{\id\ot i}
}
\]

Indeed, applying the injection $j$ to the coaction of $V$ yields by (\ref{2:cv})
$Can(1\ot_{P^A}p)\in \hom(A,P)$, and applying it to the composed map gives
$j(p\sw0\ot p\sw1)\in\hom(A,P)$. Evaluating these maps on an arbitrary $a\in A$,
one obtains
\[
Can(1\ot_{P^A}p)(a)=p\act a=p\sw0 a(p\sw1)=j(p\sw0\ot i(p\sw1))(a).
\]
Thus the diagram (\ref{2:vc}) is commutative as claimed. Next, let us choose $f\in P^*$ 
such that $f(1)=1$. Putting together the constructions from Lemma~\ref{2:peter2}
and embedding $i(C)\ot i(C)$ in $(A\ot A)^*$, we compute:
\begin{eqnarray}
(\hD\circ i)(c)(a\ot a')&=&\llp\gamma(f(c\twa \cdot)\ot_{P^A}{c\twb }\sw0)\ot i({c\twb }\sw1)\lrp(a\ot a')
\nonumber\\ &=&
f\llp(c\twa ({c\twb }\sw0\act a)\lrp a'({c\twb }\sw1)
\nonumber\\ &=&
f\llp(c\twa ({c\twb }\act a'\act a)\lrp
\nonumber\\ &=&
f\llp Can(c\twa \ot_{P^A}{c\twb })(a'a)\lrp
\nonumber\\ &=&
f\llp(j\circ(\id\ot i)\circ can)(c\twa \ot_{P^A}{c\twb })(a'a)\lrp
\nonumber\\ &=&
f\llp j(1\ot i(c))(a'a)\lrp
\nonumber\\ &=&
f\llp(a'a)(c)\lrp
\nonumber\\ &=&
(a'a)(c).
\end{eqnarray}
Here we abused the notation and denoted by $\gamma$ the map $P^*\ot_{P^A}P\ra A^*$.
On the other hand,
\[
\llp i(c\sw1)\ot i(c\sw2)\lrp (a\ot a')= a(c\sw1)a'(c\sw2)=(a'a)(c).
\]
The counitality of $i$ is also clear:
\[
(\he_V\circ i)(c)=i(c)(1)=\he_C(c).
\]
Hence $V$ and $C$ are isomorphic as coalgebras. This way we have shown that
indeed the procedure from Lemma~\ref{2:peter2} applied to the algebra-Galois
\csop-extension recovers the original coalgebra-Galois $C$-extension. 
\begin{Lem}
Let $P$ be an algebra and a  right  $C$-comodule. If the coaction is Galois,
then  it is isomorphic to 
the coaction of $V$ corresponding via Lemma~\ref{2:peter2} to the induced
algebra-Galois \csop-extension.
\end{Lem}
Now
one might ask what happens if we start from an algebra-Galois $A$-extension,
go to the coalgebra-Galois $V$-extension, and then to the algebra-Galois
$(V^*)^{op}$-extension. It turns out that $A$ is a subalgebra of $(V^*)^{op}$
via $A\hookrightarrow A^{**}$, and its action on $P$ factors via the action of
$(V^*)^{op}$. First, note that $V\inc A^*$ and, by condition (ii)(b), the pullback
of this inclusion composed with $A\hookrightarrow A^{**}$, i.e.,
$i_A: A\hookrightarrow A^{**}\ra V^*$, is injective. In fact, the injectivity
of this map is equivalent to condition (ii)(b).
Let us  check now that $i_A$   is an algebra homomorphism from $A$ to $(V^*)^{op}$.
To this end, we choose any $v\in V$, and compute:
\[
\llp i_A(a)i_A(a')\lrp(v)= v\sw1(a')v\sw2(a)
=
(\hD v)(a'\ot a)
=
v(aa')
=
\llp i_A(aa')\lrp(v).
\]
Here the penultimate identity follows from (\ref{2:dv}). The unitality of $i_A$ is
also clear:
\[
i_A(1_A)(v)=v(1_A)=\he(v)=1_{(V^*)^{op}}(v).
\]
Finally, the diagram
\[
\xymatrix{
P\ot A \ar[r] \ar[d]_{\id\ot i_A} & P \\
P\ot (V^*)^{op} \ar[ru] &
}
\]

is commutative because, by virtue of (\ref{2:cv}), we have
\begin{eqnarray}
p\act i_A(a) &=& p\sw0 i_A(a)(p\sw1)
\nonumber\\ &=&
p\sw0 p\sw1(a)
\nonumber\\ &=&
j(p\sw0\ot p\sw1)(a)
\nonumber\\ &=&
Can(1\ot_{P^A}p)(a)
\nonumber\\ &=&
p\act a.
\end{eqnarray}
We can summarise much of the above in the following:
\begin{Thm}\label{2:maina}
Let $P$ be an algebra and a right $A$-module. The action of $A$ on $P$ is
Galois {\em if and only if} there exists a coalgebra $V=:C\inc A^*$  coacting
on $P$ on the right and such that
\begin{rlist}
\item
the induced action of \csop\ on $P$  ($p\act f=p\sw0 f(p\sw1)$) is Galois,
\item 
$A$ is a subalgebra of \csop\ via the composition of the canonical 
embedding with the pullback of the above inclusion:
$i_A:A\hookrightarrow A^{**}\ra\csop$,
\item the action of $A$ factors through $i_A$ and the action of \csop\
($p\act a=p\act i_A(a)$).
\end{rlist}
\end{Thm}
\begin{proof}
If the action of $A$ on $P$ is Galois, then, by Lemma~\ref{2:peter2}, there
exists a right Galois coaction on $P$ by a coalgebra $C\inc A^*$. 
 On the other hand, by Lemma~\ref{2:piotr},
the induced action of \csop\ on $P$ is Galois. The remaining properties follow
from the discussion preceding the theorem.

Assume now that conditions (i)--(iii) are satisfied. It follows from condition~(i)
and Lemma~\ref{2:piotr} that the coaction of $C$ is Galois. Next, condition~(iii) entails
that
\[\label{2:entails}
p\act a=p\act i_A(a)=p\sw0 i_A(a)(p\sw1)=p\sw0 p\sw1(a)= j(p\sw0\ot p\sw1)(a).
\]
Since $C\inc A^*$, the map $j:P\ot C\ra \hom(A,P)$ is injective. 
Therefore, (\ref{2:entails}) implies that
\[
\fa p\in P,\,a\in A:\,(bp)\act a=b(p\act a)\;\Leftrightarrow\; 
\fa p\in P:\, (bp)\sw0\ot (bp)\sw1=bp\sw0\ot p\sw1\,.
\]
Consequently, $P^A=P^{coC}$. Now, using (\ref{2:entails}), one can immediately 
check that $Can=j\circ can$. Hence, the injectivity of $j$ and $can$
imply the injectivity of $Can$, and the surjectivity of $can$ implies condition~(ii)(a) in Definition~\ref{2:def.agalois.act}.
Finally, condition~(ii)(b) in Definition~\ref{2:def.agalois.act} follows from condition~(ii).
\end{proof}
To illustrate the foregoing theory, let us consider an example with trivial
invariants, so that we can focus on the subtelties particular to the Galois actions
of infinite dimensional algebras.
\begin{Ex}
Let $P$ be the Hopf algebra of Laurant polynomials $\C[z,z^{-1}]$ acted upon by
the group algebra $A=\C U(1)$ via the formula 
$z^\mu\act e^{i\theta}=e^{i\mu\theta}z^\mu$. We want to show that this action
is Galois. First, note that the invariant subalgebra $P^A$ is trivial:
$$
\left(\lfa\mu\in\Z,\,\theta\in[0,2\pi):\;(\sum_{k=-m}^na_kz^kz^\mu)
\act e^{i\theta}
=\sum_{k=-m}^na_kz^k(z^\mu\act e^{i\theta})\right)\Rightarrow
a_k=\hd_{k0}a_0.
$$
One can guess that the Galois coaction standing behind this action is simply
the coproduct on $P=\C[z,z^{-1}]$. (The canonical map has the form 
$p\ot p'\to pp'\sw1\ot p'\sw2$.) Therefore, in view of Lemma~\ref{2:peter2},
we take $V\cong\C[z,z^{-1}]$ and view it as a subset of $(\C U(1))^*$ via the
evaluation map:
$i(z^\mu)(e^{i\theta}):=e^{i\mu\theta}$. The injectivity
of the mapping $i$ from $\C[z,z^{-1}]$ to $(\C U(1))^*$ is clear, because a rational
function that is 0 at infinitely many points is the 0 function. Now, it is
straightforward to verify that
$j\circ(\id\ot i)\circ can=Can$, so that condition (ii)(a) holds due to the surjectivity 
of $can$. Condition 1 (injectivity
of $Can$) also holds, because $i$, $j$ and $can$ are injective. Finally, we need
to check that
$$
\left(\lfa\mu\in\Z,\,\mbox{finite subset $I$ of $U(1)$} :\;
i(z^\mu)(\sum_{g\in I}
\lambda_g g)=0\right)\;\Rightarrow\;\left(\lfa g\in I:\; \lambda_g=0\right).
$$
To this end,\footnote{We owe this argument to R.~Matthes.}
note that the left-hand side can be thought of as a system of infinitely many
linear equations, where the lambdas are variables and elments of $U(1)$ are
complex coefficients. Let $n$ be the number of elements in $I$. Since the
equation 
$
\sum_{j=1}^n\lambda_j{g_j^\mu}=0\in\C
$
has to be satisfied for all $\mu\in\Z$, it has to be satisfied for 
$\mu\in\{0,...,n-1\}$. This way we obtain a system of $n$ linear equations
with the coefficient matrix
\[
\left(\begin{array}{cccc}
1&1&\ldots&1\\
g_1&g_2&\ldots&g_n\\
\vdots&&&\vdots\\
g_1^{n-1}&g_2^{n-1}&\ldots&g_n^{n-1}\\
\end{array}\right).
\]
The (Vandermonde) determinant of this matrix is
$\prod_{n\geq j>l\geq 1}(g_j-g_l)$. It is
 non-zero because all the $g_j$s are pairwise different.
This proves that the linear system has only the zero solution, whence
$\sum_{g\in I}
\lambda_g g=0\in\C U(1)$. Thus we can conclude that the action of $\C U(1)$
on $\C[z,z^{-1}]$ is Galois. Therefore, according to Theorem~\ref{2:maina},
the group algebra $\C U(1)$ can be viewed as a subalgebra of the convolution
algebra $(\C[z,z^{-1}]^*)^{op}\cong\map(\Z,\C)$. The injective homomorphism
is given by the formula
\[
\C U(1)\ni\!\!\!\sum_{g\in I\subseteq U(1)}\!\!\!\lambda_g g\to f\in\map(\Z,\C),~~~
f(\mu)=\!\!\!\sum_{g\in I\subseteq U(1)}\!\!\!\lambda_g g^\mu.
\]
\end{Ex}
\note{
One can also consider a possibly more general  version of Galois actions.
(Translation map with coefficients $\tau:V\ra P\ot_{P^A}P\ot A$.)
Let $P$ be an algebra and a right $A$-module. We say that
the action of $A$ on $P$ is almost Galois if there exists a left $P$-linear 
injection $\widetilde{\alpha}: P\ot_{P^A}P\ot A\ra P\ot A^*\ot A$ satisfying
1) $(\id \ot\<\; ,\;\>)\ci\widetilde{\alpha}=m_P\ci(\id \ot_{P^A}\triangleleft)$,
where $\<\; ,\;\>: A^*\ot A\ra k$ is the canonical pairing.
2) There exists a subspace $V\inc A^*$ such that
~~~a) $V(a)=0\Rightarrow a=0$,
~~~b) $\widetilde{\alpha}(P\ot_{P^A}P\ot A)=P\ot V\ot A$.}

\subsection{The (co)translation map}\label{2:subs.trans}

Apart from its  value as a technical tool, the translation map has a
very nice geometric interpretation. In classical geometry a principal
bundle can also be defined in the following way
(cf.\ \cite[Section 4.2]{h-d94}). Consider a topological space
 $X$ with a free
action of a topological group $G$.
The freeness of action means that $x\cdot g =
x$ for $x\in X$, $g\in G$ implies that $g=1$. 
It guarantees that there is a function $\hat{\tau}: X\times_{X/G} X\ra G$
determined by the relation $x\cdot \hat{\tau}(x,x') = x'$. The function
$\hat{\tau}$ is called a {\em translation function}. One then says that $X$ is
a principal $G$-bundle over $X/G$ provided the translation function is continuous and $X\times_{X/G} X$ is closed in $X\times X$. These two conditions are equivalent to the action being
proper. Dualising the notions of a translation function
$\hat{\tau}$ and of a free action
one arrives at the translation map in Definition~\ref{2:transl.def}.

\subsubsection{Coalgebra extensions}

In the studies of coalgebra-Galois extensions an important role is
played by the notion of a translation map.
\begin{Def}
\label{2:transl.def}
For a coalgebra-Galois  $C$-extension $B\inc P$, the map
$$\tau: C \longrightarrow P \otimes _B P \ , \  c \longmapsto \can^{-1}(1 \otimes c),$$
is called a {\em translation map}.
For each $c \in C$, the image $\tau(c)$ is denoted by
$\tau(c):=c^{[1]} \otimes_B c^{[2]}$ (summation understood).
\end{Def}

\begin{Lem}[\sc Translation Map Lemma] \label{2:translation}
Let $C$-extension $B\inc P$ be a coalgebra-Galois extension. 
 For all $c \in C$ and $p \in P$, the translation map $\tau$
has the following properties:
\begin{rlist}
\item $c^{[1]}{c^{[2]}}_{(0)} \otimes {c^{[2]}}_{(1)} = 1 \otimes c$;
\item $c^{[1]}c^{[2]}=\varepsilon(c)1$;
\item $p_{(0)} {p_{(1)}}^{[1]} \ten{B} {p_{(1)}}^{[2]}= 1 \ten{B} p$;
\item $c^{[1]} \ten{B} {c^{[2]}}_{(0)} \otimes {c^{[2]}}_{(1)} = {c\sw
1}^{[1]}
 \ten{B} {c\sw 1}^{[2]} \otimes c\sw 2$;
 \item $c^{[1]} \ten{B} 1 \ten{B} {c^{[2]}} = {c\sw
1}^{[1]}
 \ten{B} {c\sw 1}^{[2]} c\sw 2\su 1\ten{B} c\sw 2\su 2$;
 \item Gauge invariance: for any  algebra and right $C$-comodule map $P\stackrel{F}{\ra} P$, 
$(F\ot_BF)\ci\tau=\tau$;
\item In the case of an $e$-coaugmented coalgebra-Galois extension $B\inc P$: 
$\  e^{[1]} \otimes_B e^{[2]}=
 1 \otimes_B 1$.
\end{rlist}
\end{Lem}
\begin{proof}
(i)  From
the definition of the translation map it follows that
$$c^{[1]}{c^{[2]}}_{(0)} \otimes {c^{[2]}}_{(1)}=(m \otimes \id)\circ
(\id \otimes \Delta_P)\circ \, \tau(c)=can \circ \tau(c)=1 \otimes c.$$

(ii) Apply $\id\ot\eps$ to property (i).

(iii) Applying $\can$ to both sides of the equation, on one hand we obtain
$$
\can(p_{(0)}{p_{(1)}}^{[1]} \otimes_B {p_{(1)}}^{[2]})=
p_{(0)} {p_{(1)}}^{[1]} {{p_{(1)}}^{[2]}}_{(0)}
\otimes {{p_{(1)}}^{[2]}}_{(1)}=
p_{(0)} \otimes p_{(1)}
$$ 
by property (i). On the other hand,  $\can(1 \otimes p)=
p_{(0)} \otimes p_{(1)}$. Since $\can$ is a bijection,
both arguments are equal.

(iv) We apply the isomorphism
$can \otimes \id_C$ to both
 sides and  use the fact that the canonical map $\can$ is
right $C$-colinear to obtain
 \begin{eqnarray*}
\can(c^{[1]} \otimes_B {c^{[2]}}_{(0)}) \otimes {c^{[2]}}_{(1)} &=&
\can(c^{[1]} \otimes_B {c^{[2]}})_{(0)}
\otimes \can(c^{[1]} \otimes_B {c^{[2]}})_{(1)}\\
&=&
1 \otimes c\sw 1 \otimes c\sw 2=\can({c\sw 1}^{[1]}
\otimes_B {c\sw 1}^{[2]}) \otimes c\sw 2.
\end{eqnarray*}

(v) Apply $\id\ot_B\can^{-1}$ to equality (iv) and then use property (iii).

(vi) For any $c\in C$ compute
$$
can (F(c\su 1)\ot_B F(c\su 2)) = F(c\su 1)F(c\su 2)\sw 0\ot F(c\su 2)\sw 1 = F(c\su 1c\su 2\sw 0)\ot c\su 2 \sw 1 =1\ot c.
$$
Now apply $\can^{-1}$ to deduce the assertion.

(vii) Since for an $e$-coaugmented coalgebra-Galois extension
$\Delta_P(e) =1\otimes e$, we can use property (iii) to compute
$$
e^{[1]} \otimes_B e^{[2]}=1_{(0)}{1_{(1)}}^{[1]}
\otimes_B {1_{(1)}}^{[2]}=
1 \otimes_B 1,
$$
as required. 
\end{proof}

The properties listed in the Translation Map Lemma are simply
dualisations of the properties of the classical translation function. For
example, Lemma~\ref{2:translation}(ii) corresponds to the classical
property $\hat{\tau}(x,x)=1$, while Lemma~\ref{2:translation}(v)  corresponds to the classical
transitivity property of the translation function,
$\hat{\tau}(x,x')\hat{\tau}(x',x'') = \tau(x,x'')$, for all $x,x',x''\in
X$. The properties of the translation map in the case of a Hopf-Galois
extension were first studied in \cite{s-hj90b}.

As the first application of the Translation Map Lemma we show that two
notions of coaction invariants
 introduced earlier coincide in the case of an
$e$-coaugmented coalgebra-Galois extension.

\begin{Prop}\label{2:prop.ecoin=coin}
$P_e^{coC}=P^{coC}$ for an $e$-coaugmented coalgebra-Galois 
$C$-extension $B\inc P$.
\end{Prop}

\begin{proof}
$P^{coC} \subseteq P_e^{coC}$ is proven in Lemma \ref{2:coinvariants1}.
For the opposite inclusion, let $b \in P_e^{coC}$, i.e.~$b\sw 0 \otimes b\sw
1=b \otimes e$.
Then using properties (iii) and (v) from Lemma \ref{2:translation} we
obtain,  for all 
$p \in P$,
$$
\Delta_P(bp)=b_{(0)} {b_{(1)}}^{[1]}\Delta_P({b_{(1)}}^{[2]}p)=
b e^{[1]} \Delta_P(e^{[2]}p)=b \Delta_P(p).
$$
This shows that $b \in P_e^{coC}$.
\end{proof}

\subsubsection{Algebra coextensions}

Assume that $C\twoheadrightarrow B$ is an algebra-Galois 
$A$-coextension. Then
$cocan$ is a bijection and there exists the 
{\it cotranslation map} $\check{\tau} : C\Box_B C \ra
A$, $\check{\tau} := (\eps\otimes \id_C)\circ cocan^{-1}$. By dualising
properties of the translation map 
(or directly from the definition of $\check{\tau}$),
one can establish the corresponding properties of the cotranslation map.

\begin{Lem}[\sc Cotranslation Map Lemma] \label{2:cotranslation} 
Let $C\twoheadrightarrow B$ be an algebra-Galois 
$A$-coextension. Then the cotranslation 
map $\check{\tau}$ has the following properties, for all $c\ot c'\in C\,\coten{B}\,C$ (summation implicit) and $a\in A$:
\begin{rlist}
\item $\check{\tau}\circ cocan = \eps\ot \id_A$;
\item $\check{\tau}\circ\Delta = \eps$;
\item $\mu_C\circ(\id_C\otimes\check{\tau})\circ(\Delta \,\coten{B}\,\id_C) =
\eps\,\coten{B}\, \id_C$ or explicitly $\mu_C(c\sw 1,\check{\tau}(c\sw 2,c')) = \eps(c)c'$;
\item  $\check{\tau}(c,\mu_C(c',a)) = \check{\tau}(c,c')a$;
\item $\check{\tau}\circ(\id_C\,\coten{B}\,\eps\,\coten{B}\,\id_C) = m_A\circ(\check{\tau}\ot\check{\tau})\circ(\id_C\,\coten{B}\, \Delta\,\coten{B}\, \id_C)$;
\item Gauge invariance: for any right $A$-linear coalgebra map $F:C\ra C$, 
$\check{\tau}\circ (F\Box_B F) = \check{\tau}$;
\item In the case of a $\kappa$-augmented algebra-Galois 
$A$-coextension 
$C\twoheadrightarrow B$, $\kappa\circ\check{\tau} = \eps\Box_B \eps$.
\end{rlist}
\end{Lem}


\subsubsection{Algebra extensions}

Let $B\subseteq P$ be an algebra-Galois extension by $A$
 (relative to $V\subseteq A^*$)  as in Definition~\ref{2:def.agalois.act}. 
Then for any $v\in V$ one can find a unique $v\su 1\ot_B v\su 2 \in P\ot_B P$ 
(summation assumed) such that, for any $p\in P$, $Can(v\su 1\ot_B v\su 2) = 
j(p\ot v)$. The assignment $v\mapsto v\su 1\ot_B v\su 2$ defines a map 
$\tau: V\ra P\ot_B P$ which is called a {\em translation map} for an
algebra-Galois $A$-extension $B\subseteq P$. In view of Theorem~\ref{2:maina}, the map $\tau$ can 
be seen to coincide with the translation map of the corresponding coalgebra-Galois
 extension, so that the Translation Map Lemma can be used to derive the following
 properties of $\tau$.
\begin{Lem}\label{2:a-translation}
Let $B\subseteq P$ be an algebra-Galois extension by $A$ 
(relative to $V\subseteq A^*$). Then the translation map 
$\tau: V\ra P\ot_B P$ has the following properties for all 
$a\in A$, $v\in V$ and $p\in P$
\begin{rlist}
\item $v\su 1(v\su 2\act a) = 1_P v(a)$.
\item $v\su 1 v\su 2 = 1_P v(1_A)$.
\item Let $\sum_i p_i\ot v_i = j^{-1} (Can(1_P\ot _B p))$. 
Then $\sum_i p_iv_i\su 1\ot_B v_i\su 2 = 1\ot_B p$.
\item Consider $V$ as a right $A$-module by $(v\cdot a)(a') = v(aa')$. 
Then $v\su 1\ot_B v\su 2\act a = (v\cdot a)\su 1\ot_B (v\cdot a)\su 2$, i.e.,
 $\tau$ is a right $A$-module map.
\item Gauge invariance: for any right $A$-linear algebra map $F:P\ra P$, 
$(F\ot_B F)\circ \tau = \tau$.
\item In the case of a $\kappa$-augmented algebra-Galois extension, 
$\kappa\su 1\ot_B \kappa\su 2 = 1\otimes_B 1$.
\end{rlist}
\end{Lem}
\begin{proof}
The properties listed in the lemma are simply properties of the translation map in
 Lemma~\ref{2:translation} translated to the algebra-Galois case with the help of
 Theorem~\ref{2:maina}. 
 We leave direct proofs to the reader as an exercise. Note only that (i) is 
simply the definition of a translation map.
\end{proof}

\subsection{Entwining and factorisation}\label{2:subs.ent.fac}

\subsubsection{Entwining}

In the definition of a Hopf-Galois extension one requires that $P$ is a
comodule algebra of a Hopf algebra $H$. Obviously, no such assumption is
possible in the case of a coalgebra-Galois extension, since a priori $C$
does not have any algebra structure.  This might appear to be a problem,
in particular when the differential geometry of coalgebra-Galois
extensions is considered, and some replacement for an algebra structure
on $C$ compatible with an algebra structure on $P$ is needed. Thus the
original point of view taken in \cite{bm98a} was to require a
compatibility between $P$ and $C$ in terms of an {\em entwining}, and
then define the canonical map and require its bijectivity
 within this setting. It has been realised in
\cite{bh99} that given a coalgebra-Galois extension as defined in
Definition~\ref{2:def.cgalois},
 there actually exists a relation between the coalgebra structure of
$C$ and the algebra structure of $P$
provided by an entwining.

\begin{Def} \label{2:ES}
An {\em entwining structure} consists of a triple $\Es$,
where $A$ is an algebra,
$C$ a coalgebra and $\psi: C \otimes A \ra A \otimes C$ a linear map
such that 
the following ``bow-tie''-diagram commutes:\vspace*{1mm}
\[
\xymatrix{
& C\ot  A\ot  A \ar[ddl]_{\psi\otimes \id_A} 
\ar[dr]^{\id_C\otimes m}  &  &             
C\ot  C\ot  A \ar[ddr]^{\id_C\otimes\psi}& \\
& & C\ot  A \ar[ur]^{\Delta\otimes \id_A} 
\ar[dr]^{\eps\ot \id_A}\ar[dd]^{\psi}  & &\\
A\ot C\ot  A \ar[ddr]_{\id_A\otimes\psi} & C \ar[ur]^{\id_C\otimes 
\eta} \ar[dr]_{\eta\otimes \id_C}& &A  & C\ot  A\ot  C 
\ar[ddl]^{\psi\otimes \id_C}\\
& &  A\ot  C \ar[ur]_{\id_A\otimes\eps} \ar[dr]_{\id_A\otimes\Delta} & &\\
& A\ot  A\ot  C \ar[ur]_{m\otimes \id_C} &  & A\ot C\ot  C & ,}
\]
where $m$ is the multiplication and $\eta$ the unit of $A$, and $\Delta$ is comultiplication and $\eps$ the counit of $C$. 
\end{Def}

The notion of an entwining structure was introduced in
\cite{bm98a}. The bow-tie diagram expresses the most natural compatibility
conditions
between algebra and coalgebra structures. The right part of the diagram
states that $\psi$ respects the coproduct (right pentagon) and the
counit (right triangle). The left part of the diagram expresses the
compatibility of $\psi$ with  the product (left pentagon) and the unit
(left triangle). Any algebra and a
coalgebra can be provided with an entwining structure with $\psi$ being
the usual flip of tensor factors (for obvious reasons this can be called
a {\em trivial} entwining structure). There are several extremely nice
properties of an entwining structure. Most importantly, the notion of an
entwining structure is self-dual in the following sense. The structure
described by the bow-tie diagram is invariant under the operation that
involves interchanging $A$ with $C$, $\Delta$ with $m$, $\varepsilon$
with $\eta$, and reversing all the arrows.  Another property is related
to tensor algebras and, consequently to universal differential algebras.
 Consider the tensor
algebra $A^\otimes = \oplus_n A^{\otimes n}$ with the product given by
the concatenation, i.e.,
$$
(a^1\otimes \ldots \otimes a^m)(a^{m+1}\otimes\ldots a^{m+n}) =
a^1\otimes \ldots \otimes a^ma^{m+1}\otimes\ldots a^{m+n}.
$$
Then $\Es$ induces an entwining structure $(A^\otimes, C,\psi{^\otimes})$
with 
$$
\psi^\otimes\mid_{C\otimes A^{\otimes n}} = (\id_A^{\otimes n-1}\otimes
\psi)\circ(\id_A^{\otimes n-2}\otimes
\psi\otimes\id_A)\circ\ldots \circ (\psi\otimes \id_A^{\otimes n-1}).
$$
As an exercise in the self-duality the reader can derive the
corresponding entwining structure for a tensor coalgebra $C^\otimes$.

To describe the action of $\psi$
we  use the following {\it $\alpha$-notation}:
$\psi(c\otimes a) = a_\alpha \otimes c^\alpha$ (summation
over a Greek index understood), for all $a\in
A$, $c\in C$, which proves very useful in concrete computations
involving $\psi$. The reader is advised to check that the bow-tie
diagram is
equivalent to the following four explicit relations:
\begin{eqnarray*}
&&\hspace{-.7cm} \mbox{left pentagon:}\;\; (aa')_\alpha \otimes c^\alpha =
a_\alpha a'_\beta \otimes
c^{\alpha\beta} ,  \\
&&\hspace{-.7cm}\mbox{left triangle:}\;\;
1_{\alpha}\otimes c^\alpha = 1\otimes c,\\
&&\hspace{-.7cm}\mbox{right pentagon:}\;
a_\alpha\otimes c^\alpha\sw 1\otimes c^\alpha\sw 2 =
a_{\beta\alpha}\otimes
c\sw 1^\alpha\otimes c\sw 2^\beta , \\
&&\hspace{-.7cm} \mbox{right triangle:}\; a_\alpha\eps(c^\alpha) =
a\eps(c),
\end{eqnarray*}
for all $a,a'\in A$, $c\in C$.

One may (or perhaps even should) think of an entwining map $\psi$ as a
twist in the convolution algebra ${\rm Hom}(C,A)$. Namely, given an
entwining structure, one can define  the map $*_\psi:
{\rm Hom}(C,A)\otimes{\rm Hom}(C,A)\ra{\rm Hom}(C,A)$ via $(f*_\psi
g)(c) = f(c\sw 2)_\alpha g(c\sw 1^\alpha)$, for all $f,g\in {\rm
Hom}(C,A)$ and $c\in C$. One can easily check that $({\rm
Hom}(C,A),*_\psi)$ is an associative algebra with unit $\eta\circ\eps$.
This algebra is known as a {\em $\psi$-twisted convolution algebra}.

Directly from the definition of an entwining structure one obtains the
following
\begin{Lem} \label{2:lemma.entw.gro}
Let $\Es$ be an entwining structure. If $e\in C$ is a group-like element,
then $A$ is a right $C$-comodule with the coaction
$$
\Delta_A: A\ra A\otimes C, \qquad a\mapsto \psi(e\otimes a).
$$
Dually, if $\kappa: A\ra k$ is a character (i.e., an algebra map), then
$C$ is a right $A$-module with the action $c\cdot a =
a_\alpha\kappa(c^\alpha)$.
\end{Lem}
\begin{proof} 
The fact that $\Delta_A$ is a right $C$-coaction follows from the right
part of the bow-tie diagram in Definition~\ref{2:ES}. In particular,
from the right pentagon and
using the fact that $e$ is a group-like element, we
have
\begin{eqnarray*}
(\Delta_A\otimes \id)\circ \Delta_A(a) &=& \Delta_A(a_\alpha)\otimes
e^\alpha = a_{\alpha\beta} \otimes e^{\beta}\otimes e^\alpha
= a_{\alpha\beta} \otimes e\sw 1^{\beta}\otimes e\sw 2^\alpha\\
& =&
a_\alpha\otimes e^\alpha\sw 1\otimes e^\alpha\sw 2 = (\id\otimes
\Delta)\circ
\Delta_A(a),
\end{eqnarray*}
while the right triangle implies that $a_\alpha\eps(e^\alpha) = a\eps(e)
=a$. 

The second statement follows from the self-duality of the notion of
an entwining structure.
\end{proof}

We now give two generic examples of entwining structures.

\begin{Ex} \label{2:Sentwined}
If $H$ is a bialgebra, then
$\psi: H \otimes H \ra H \otimes H \  , \  h \otimes g \mapsto g\sw 1
\otimes hg\sw 2$
defines an entwining structure with $C = A =H$.
Conversely, given an
algebra and coalgebra $H$
such that the map $\psi$ above is an entwining map,
 $H$ is a bialgebra.
The first statement can be checked by direct computations. Conversely,
suppose  that $H$ is an algebra and a coalgebra and that
$(H,H,\psi)$ is an entwining structure. We need to show that
the coproduct $\Delta$ and the counit $\eps$ of $H$ are algebra
maps. Evaluating the left pentagon identity
at $1\otimes h\otimes h'$ for all $h,h'\in H$, we immediately conclude
that $\Delta$ is a multiplicative map. Also,
evaluating the left triangle at $1\otimes 1$ we obtain $\Delta(1) =
1\otimes 1$. Thus $\Delta$ is an algebra map as required. Furthermore
the right triangle reads in this case $ h\sw 1\eps(h'h\sw 2) =
\eps(h')h$. Thus applying $\eps$ we immediately deduce that the counit
is a multiplicative map, as required.
\end{Ex}

The entwining constructed in Example~\ref{2:Sentwined} is known as a {\em
bialgebra entwining}. It justifies the statement that an entwining
structure is a generalisation of a bialgebra. A bialgebra entwining is a
special case of the following more general example.

\begin{Ex} \label{2:ex.Doi.Koppinen}
Given a bialgebra $H$, let  $C$ be a right $H$-module coalgebra and
$A$ a right $H$-comodule
algebra. Recall that this means that $C$ is a right $H$-module and
the module structure map $\mu_C :C\otimes
H\ra C$ is a coalgebra map, and that $A$ is a right $H$-comodule with
the coaction $\Delta_A: A\ra A\otimes H$ that is an algebra map.
Consider a $k$-linear map
$$
\psi: C \otimes A \ra A \otimes C ,\quad c \otimes a \mapsto a \sw 0
\otimes c \cdot a\sw 1.
$$
Then $\Es$ is an entwining structure.
   This is shown by direct calculations (cf.\ \cite[33.4]{bw03}). 
\end{Ex}

A triple $(H,C,A)$ satisfying conditions of 
Example~\ref{2:ex.Doi.Koppinen}
is known as a (right-right) {\em Doi-Koppinen datum} or
a {\em Doi-Koppinen structure},
the corresponding entwining structure is known as an {\em
entwining structure associated to a Doi-Koppinen datum}.
 The studies of Doi-Koppinen structures where initiated independently
by Doi in \cite{d-y92} and Koppinen in \cite{k-m94}. Doi-Koppinen
data as a separate entity first appeared in \cite{cr95}, and then
they were given a separate name in~\cite{cmz97}.  (Incidentally, they
were 
called Doi-Hopf data.) Various properties and applications
of Doi-Koppinen structures are studied in a  monograph
\cite{cmz02}. Finally, the entwining structure associated to a
Doi-Koppinen datum was first introduced in~\cite{b-t99}.

Clearly, a bialgebra $H$ itself forms a Doi-Koppinen datum $(H,H,H)$ in
which $H$ is viewed as an $H$-comodule algebra via the coproduct and $H$
is an $H$-module coalgebra via the multiplication. Therefore the
bialgebra entwining in Example~\ref{2:Sentwined} is a special case of the
Doi-Koppinen entwining. Several other special cases of the Doi-Koppinen
entwining are of particular interest. Most notably, the {\em relative Hopf
entwining} in which $C=H$, and  the {\em dual-relative Hopf entwining}
in which $A=H$. Another important example of an entwining, which again
is a special case of the Doi-Koppinen entwining, but is best verified
directly, is the {\em Yetter-Drinfeld entwining}. In this case $H$ is a
Hopf algebra, $A=C=H$,
and the entwining map $\psi :H\otimes H\ra H\otimes H$ is given by
$\psi: h'\otimes h \mapsto h\sw 2\otimes (Sh\sw 1)h' h\sw 3$. We
encourage the reader to verify that $\psi$ satisfies the bow-tie
diagram, and to verify that this entwining comes from the
Doi-Koppinen datum $(H^{op}\otimes H, H,H)$ (remember that
the first entry is the
bialgebra), with the following module and
comodule structures. The right multiplication by $H^{op}\otimes H$ is
given by $h\cdot (h'\otimes h'') =
h'hh''$, and the right coaction of $H^{op}\otimes H$ on $H$ is
$\varrho^{H}(h) = \sum h\sw 2\otimes Sh\sw 1\otimes h\sw 3$
(cf.\ \cite{cmz97}).

The $\psi$-twisted convolution algebras corresponding to all those
special cases of entwining structures were also studied. For example
the $\psi$-twisted convolution algebra corresponding to a Doi-Koppinen
datum was introduced in \cite[Definition~2.2]{k-m94}
and is also known as  {\em
Koppinen's smash product}. The $\psi$-twisted convolution algebra
corresponding to the relative-Hopf entwining was studied in
\cite{d-y84} and \cite{k-m92}.

A dual version of Example~\ref{2:ex.Doi.Koppinen}, in which $A$ is a left
$H$-module algebra and $C$ is a left $H$-comodule coalgebra
has been constructed by Schauenburg in \cite{s-p00}, and is termed an
{\em 
alternative Doi-Koppinen datum}. Although
Doi-Koppinen and alternative Doi-Koppinen data provide a rich source of
entwining structures, and one can show that, if either $A$ or $C$ is
finite dimensional, any  $\Es$ is  of an (alternative) Doi-Koppinen
type,
  they do not exhaust all possibilities. An example of
an entwining structure that does not come
from 
Doi-Koppinen data is constructed in \cite[Example~3.4]{s-p00}. This
construction is based on earlier work of Tambara \cite{t-d90} on
factorisations of algebras.

Given an entwining structure $\Es$,
a right $A$-module and right $C$-comodule $M$ with coaction $\varrho^M:M \to M\ot C$
is called an {\em entwined module} if for all $m \in M, a \in A$
$$\varrho^M(m \cdot a)=m\sw 0 \cdot \psi(m\sw 1 \otimes a) = m\sw 0\cdot
a_\alpha\otimes m\sw 1 ^\alpha \ .$$
The category of entwined modules together with right $A$-linear and right
$C$-colinear morphisms is denoted by
 $\Me$. The following example shows
that entwined modules unify and generalise various categories of Hopf
modules studied for the last 40 years.

\begin{Ex}\label{2:ex.entw.mod}
If the entwining structure $\Es$ comes from a Doi-Koppinen datum as
in example 
\ref{2:ex.Doi.Koppinen}, then $\Me=\Mc_A^C(H)$,
the category of Doi-Koppinen modules introduced in \cite{d-y92},
 \cite{k-m94}. 
In particular, Sweedler's Hopf modules \cite{s-me69} correspond to a
bialgebra entwining  in
Example~\ref{2:Sentwined}.
In a similar way, relative Hopf modules introduced in
\cite{t-m72}, \cite{d-y83} are simply entwined modules associated to
the relative Hopf entwining, dual-relative Hopf modules or $[C,A]$-Hopf
modules of \cite{d-y83} correspond to the dual-relative Hopf
entwining.
Finally,    Yetter-Drinfeld or crossed modules (cf.\ \cite{y-dn90},
\cite{rt93}), which play an important role in the representation
theory of quantum groups, are simply entwined modules of a
Yetter-Drinfeld entwining.

For either finite dimensional $A$ or $C$, any category of entwined modules
is equal to a category of (alternative) Doi-Koppinen modules.
\end{Ex}

Various properties of  entwined modules, in particular the
modules described in
Example~\ref{2:ex.entw.mod}, in various conventions are
studied in a monograph
\cite{cmz02}. The theory of entwined modules is extremely rich,
but of course, in these notes we are not able to cover it in any detail.
From the point of view of coalgebra-Galois extensions, and their
interpretation as generalised principal bundles of noncommutative
geometry, the following theorem is of the highest importance.

\begin{Thm} \label{2:canonical}
Let $C$-extension $B\inc P$ be a coalgebra-Galois extension.
Then there exists a unique entwining
structure $\Ep$, such that $P$ becomes an entwined module in $\Mp$ with
the  $P$-module structure given by the multiplication.
\end{Thm}

\begin{proof}
For a detailed proof of this fact we refer to \cite{bh99}.
We only remark that the  entwining map in this case is defined by
\[\label{2:canent}
\psi: C \otimes P \ra P \otimes C \ , \  c \otimes p \mapsto
\can(\tau(c) \cdot p),
\]
where $\tau$ is the translation map of $C$-extension $B\inc P$.
The bow-tie conditions from Definition \ref{2:ES} can now be verified using
the Translation Map Lemma~\ref{2:translation}, while the uniqueness
follows from the following simple argument.
Suppose that there is an entwining map $\tilde{\psi}$ such
that $P\in \Mc_P^C(\tilde{\psi})$ with structure maps $m_P$ and
$\Delta_{P}$.  
Then, for all $p\in P$, $c\in C$,
$$ 
\psi(c\otimes p) =  c\su 1(c\su 2 p)\sw 0\otimes (c\su 2 p)\sw 1  =
 c\su 1c\su 2\sw 0 \tilde{\psi}(c\su 2\sw 1\otimes p)
= \tilde{\psi}(c\otimes p),
$$ 
where we used the definition of the translation map to obtain the
last equality. 
\end{proof}

Although in principle there is no relation between the
coalgebra structure of $C$ and the algebra structure of $P$, the
definition of the coalgebra-Galois extension is rigid enough to produce
such a relationship in terms of an entwining. The entwining associated
to the $C$-extension $B\inc P$ in Theorem~\ref{2:canonical} is called the {\em canonical
entwining}. Its existence allows one to discuss symmetry properties
of coalgebra-Galois extensions, and to extend such symmetries to
(universal) differential structures on  the $C$-extension $B\inc P$. This is crucial for
the definition of a connection on  $B\inc P$.

To get a better feeling for canonical entwining structures it is
instructive to consider the following
\begin{Ex}\label{2:ex.can.Hopf}
Let $B\inc P$ be a quotient-coalgebra-Galois $H/I$-extension.
  In this case, using the Translation Map Lemma and the fact that  $H/I$ is a right $H$-module,
 we can explicitly compute the formula (\ref{2:canent}) for the canonical entwining:
\[
\begin{aligned}
\psi(\overline{h}\otimes p) 
&= can(\overline{h}^{[1]}\ot_B \overline{h}^{[2]}p)\\
&= \overline{h}^{[1]}{\overline{h}^{[2]}}\sw0 p\sw0\ot \overline{{\overline{h}^{[2]}}\sw1}\;p\sw1\\
&= p\sw0\ot \overline{h}p\sw1\,.
\end{aligned}
\]
Here $p\sw0\ot p\sw1$ is meant as the result of $H$-coaction on $p$. Observe that whenever the
antipode $S$ of $H$ is bijective, so is the above computed canonical entwining $\psi$.
Indeed, it is straightforward to verify that the formula $\psi^{-1}(p\ot\overline{h})=
\overline{h}S^{-1}(p\sw1)\ot p\sw0$ defines the inverse of~$\psi$.
Note also that in particular, for $I=0$, we get the canonical entwining of a Hopf-Galois extension,
which is a
relative Hopf entwining.
\end{Ex}

Very much as in the previous section, it turns out that every 
algebra-Galois $A$-coextension
is equipped with an entwining structure. More precisely, we have the following
dual version of Theorem~\ref{2:canonical}:
 
\begin{Thm}\label{2:main2}
Let $C$ be an algebra-Galois  $A$-coextension of $B$. Then there exists 
a unique map $\psi :C\otimes A\ra A\otimes C$ entwining $C$ with $A$  
and such that $C\in{\CM}\sp{C}_A(\psi)$ with the structure maps \hD\ and
$\mu_C$. (The map $\psi$  
is called the {\em canonical entwining map} associated to the 
algebra-Galois  
$A$-coextension $C\twoheadrightarrow B$.) 
\end{Thm}

Using the cotranslation map one defines a map $\psi :C\otimes A\ra A\otimes
C$ by
\begin{eqnarray}
&&\psi =
(\check{\tau}\otimes C)\circ(C\otimes\Delta)\circ cocan,\nonumber\\
&&\psi(c\otimes a ) =  \check{\tau}(c\sw 1,\mu_C(c\sw 2,a)\sw
1)\otimes\mu_C(c\sw 2,a)\sw 2\,\label{2:psi2} .
\end{eqnarray}

\subsubsection{Factorisation}

Given two algebras $A,P$, one can study all possible algebra structures
 on the tensor product
$A \otimes P$ with unit $1 \otimes 1$ and
with the property that the multiplication becomes an
$(A,P)$-bimodule map. It turns out (cf.\ \cite{t-d90},
\cite[pp.\ 299-300]{m-s95},
\cite{csv95}) that all such algebra structures
are in one-to-one correspondence with maps
$\Psi: P \otimes A \ra A \otimes P$ such that
\[ \Psi\circ (\mu_P\ut\id_A)=(\id_A\ut\mu_P)\circ (\Psi\otimes\id_P)\circ
(\id_P\ut\Psi),
\quad \Psi(1\ut a)=a\ut 1,
\quad\forall a\in A\]
\[ \Psi\circ
(\id_P\otimes\mu_A)=(\mu_A\otimes\id_P)\circ(\id_A\ut\Psi)\circ
(\Psi\ut\id_A) ,
\quad \Psi(p\otimes 1)=1\otimes p,
\quad\forall p\in P,\]
where $\mu_A$ is the product in $A$ and $\mu_P$ is the product in $P$.
A triple $(P,A,\Psi)$ is known as  a
{\em factorisation structure}.
For example, every braided tensor product of algebras gives rise to a factorisation (given by braiding). Furthermore, crossed product algebras including bicrossproducts \cite{m-s90} correspond to factorisations. 

There is a close relationship between factorisations and entwining
structures.  Let $\Ep$ be an entwining structure with a
finite-dimensional coalgebra $C$,
and let $A=(C^*)^{op}$, i.e., the dual algebra with multiplication given by
$(a  a^{\prime})(c)=a(c\sw 2)a^{\prime}(c\sw 1)$,
 for $a, a^{\prime} \in A$, $c\in C$. 
Define a map $\Psi: P \otimes A \ra A \otimes P \ ,
\ p \otimes a \mapsto \sum_i a_i \otimes p^i$, where
$\sum_i a_i(c)p^i=p_{\alpha}a(c^{\alpha})$ for all $c \in C$.
Then $\Psi$ is a factorisation, that is $A \otimes P$ has an
algebra structure 
given by $(a \otimes p)(a^{\prime} \otimes p^{\prime})=
\sum_i a a^{\prime}_i \otimes p^i p^{\prime}$.

The opposite statement holds for a factorisation
$\Psi$ together with a finite
dimensional algebra~$A$. Then one can construct an entwining structure for the  algebra $P$ and the coalgebra $C:=A^*$.

This relationship between entwining structures and factorisations as
well as the existence  of a canonical entwining structure associated to
a coalgebra-Galois extension allows one to develop the coalgebra-Galois
theory on a purely algebraic (not coalgebraic) level, in terms of
suitable factorisations. This point of view is taken and such a theory
is developed in \cite{bm00} (see Section~\ref{aext}).

Finally, we would like to remark that before the factorisations and
entwining structures appeared in the current setup, similar
structures were studied in category theory. In category theory in place
of algebras one uses monads and in place of coalgebras one uses
comonads. A structure corresponding to factorisation involves two
monads and is known as {\em distributive law} \cite{b-j69}
\cite{bb69}, while the
structure corresponding to an entwining involves a monad and a comonad
and is known as a {\em mixed distributive law} \cite{v-dh73}.

\subsection{Principal extensions}\label{2:subs.princ}

The concept of a faithfully flat Hopf-Galois extension with a 
bijective antipode 
is a cornerstone of Hopf-Galois theory. 
The following notion of a {\em principal extension} generalises this 
key concept 
in a way that 
it encompasses interesting examples escaping Hopf-Galois theory, 
yet still enjoys a number of crucial properties of the aforementioned 
class of 
Hopf-Galois 
extensions. It is an elaboration of the Galois-type extension 
\cite[Definition~2.3]{bh99} 
which evolved from 
\cite[p.182]{s-hj90a}, \cite{bm98a} and other papers. 
\\ 
\begin{Def} 
\label{2:def.principal} 
An $e$-coaugmented coalgebra-Galois $C$-extension  $B\inc P$ is said to be  {\em principal} 
 if 
\begin{rlist}
\item  the canonical entwining map  is bijective,
\item 
 $P$ is $C$-equivariantly projective as a left $B$-module.
\end{rlist}
\end{Def} 
%

 

\begin{Lem}  
Let $(A,C,{\psi})$ be an entwining structure such that 
$\psi$ is bijective. 
Assume also that there exists a 
group-like element $e\in C$ such that $A$ is a right $C$-comodule via 
$\psi\ci(e\ot\id)$ 
and a left $C$-comodule via 
$\psi^{-1}\ci(\id\ot e)$. Then 
 $A$ is coflat as a right (resp.\ left) $C$-comodule if and only if 
 there exists 
    $j_{R}\in \Rhom C CA$ (resp.\  $j_{L}\in \Lhom C CA$) such that 
    $j_{R}(e)=1$ 
(resp.\  $j_{L}(e)=1$). 
(Here $C$ is a $C$-comodule via the coproduct.) 
\label{2:prop.jrl} 
\end{Lem}
The comodule maps $j_R$ and $j_L$ are generalisations of total integrals of
Doi. The latter are used in \cite{s-hj90a} to prove coflatness results in the
context of Hopf-Galois extensions. Analogous results can be proven also in
our more general setting. The axioms of a principal extension guarantee that 
$(P,C,\psi)$ is an 
entwining structure satisfying the assumptions of the 
above lemma \cite[Theorem~2.7]{bh99}. Moreover, 
it can be shown that 
maps $j_L$ and 
$j_R$ as in Lemma~\ref{2:prop.jrl} can be constructed for any principal 
$C$-extension,  
and one can prove the 
following \cite{bh04}: 
\begin{Thm}\label{2:thm1} 
   Let $B\inc P$ be a principal $C$-extension. Then: 
\begin{zlist}   
\item  $P$ is a projective left and right $B$-module. 
\item $B$ is a direct summand of $P$ as a left and right $B$-module. 
\item $P$ is a faithfully flat left and right $B$-module. 
\item $P$ is a coflat left and right $C$-comodule. 
\end{zlist}
\label{2:cor.coflat} 
\end{Thm} 
In (4), the left $C$-comodule structure of $P$ is given by
$_P\hD(p)=\psi^{-1}(p\ot e)$.
Note that (3) follows from (1) and (2) by standard module-theoretic arguments.

\subsubsection{Extensions by coseparable coalgebras}\label{2:sec.ext.cosep}

It turns out that in many cases of interest, for example, in those in which
$C$ corresponds to a coalgebra structure of a matrix quantum group, the
injectivity of the canonical map implies its bijectivity. Recall that a coalgebra $C$ is said to be {\em coseparable} provided  
the coproduct has a retraction in the category of $C$-bicomodules.
Equivalently, $C$ is a coseparable coalgebra if there exists a {\em
cointegral}, i.e.\ a $k$-linear map $\delta :C\ot C\ra k$ such
that $\delta\circ\Delta = \eps$ and, for all $c,c'\in C$,
\begin{equation}\label{2:coint}
 c\sw 1\delta(c\sw 2\ot c') =  \delta(c\ot c'\sw 1)c'\sw 2.
\end{equation}
 Any cosemisimple coalgebra over an algebraically closed field is coseparable (cf.\ \cite[Proposition~2.5.3]{ss}).

\begin{Thm}[\cite{b-t05}, Theorem~4.6]\label{2:thm.main.princ}
Let  $(P,C)_\psi$ an entwining structure such that the map $\psi$ is bijective. Suppose that $e\in C$ is a group-like element and view $P$ as a right $C$-comodule with the coaction $\Delta_P:P\to P\ot C$, $p\mapsto \psi(e\ot p)$. If $C$ is a coseparable coalgebra and the lifted canonical map
$$
\widetilde{\can}: P\ot P{\longrightarrow} P\ot  C,\quad 
p\ot q\longmapsto p\Delta_P(q),
$$
is surjective, then $P$ is a principal $C$-extension of the 
coaction-invariant subalgebra $B=P^{co C}$.
\end{Thm}

Theorem~\ref{2:thm.main.princ} is proven in \cite{b-t05} as a special case of a
general structure theorem for {\em principal comodules} for a {\em coring} (see 
Theorem~\ref{2:thm.princ.field} below). It is also proven (over a commutative ring) in
\cite[Theorem~5.9]{ss05}. The discussion of this proof goes beyond the scope of these notes. 
On the other hand, a direct proof of Theorem~\ref{2:thm.main.princ} has been recently presented in \cite{bb08}. This uses the explicit form  of a strong connection which we describe in more detail in  Section~\ref{2:sec.cosep}. 

The  example of a noncommutative instanton bundle \cite{bcdt04}
described in Section~\ref{2:qspaces}
is principal. Since this is the situation of a quotient coalgebra extension, the
group-like element in the coalgebra $C=\sO(SU_q(4))/I$ is the image of $1$
under the canonical projection, and the invertibility of the canonical
entwining map follows from the invertibility of the antipode of the Hopf
algebra $\sO(SU_q(4))$ (see Example~\ref{2:ex.can.Hopf}). The Galois property has
been proven in \cite{bcdt04} by brute force. The hard part of this proof was to
verify the injectivity of the canonical map $can$. Since  the coalgebra
$C=\sO(SU_q(4))/I$ is cosemisimple, this part is a direct consequence of
Theorem~\ref{2:thm.main.princ}. The same theorem also gives equivariant projectivity.

\subsubsection{Hopf fibrations over the Podle\'s 
quantum 2-spheres}\label{2:qsphere.coalg}

 In our presentation we
follow \cite{bm00}, where the reader can find
detailed proofs of all the facts quoted.
The total space of the bundle
is given by the Hopf algebra of functions on the quantum group $SL_q(2)$,
$P=\sO(SL_q(2))$. Recall that $\sO(SL_q(2))$ is defined as a polynomial
algebra with unit,
generated by   $\alpha$, $\beta$, $\gamma$ and $\delta$,
with relations 
$\alpha \beta = q \beta \alpha$, $\alpha \gamma = q \gamma \alpha$,
$\beta\gamma= \gamma\beta$, $\alpha \delta =
\delta \alpha - (q-q^{-1})\beta\gamma$, $\gamma \delta = q\delta \gamma$,
$\beta \delta = q \delta \beta$, $\alpha\delta - q\beta\gamma =1$. Here
$q$ is any number which is not a root of unity. In the case of
$k=\mathbb{C}$, the algebra $\sO(SL_q(2))$ can be made into a
$C^*$-algebra $C(SU_q(2))$, provided $q\in (0,1)$.

Coideal subalgebras of $P$ or    (embeddable)
homogeneous spaces of the quantum group $\sO(SL_q(2))$ are known as {\em
quantum} or {\em Podle\'s spheres} and were introduced in
\cite{p-p87}. They are defined as subalgebras $B=S_{q,s}$, where $s\in
k$ is a parameter, 
which are  embedded in $P$
and  generated by 
\begin{eqnarray}
\xi &=& s(\alpha^2 - q^{-1}\beta^2)+(s^2 - 1)q^{-1} \alpha \beta , \nonumber\\
\eta &=& s(q \gamma^2 - \delta^2)+(s^2 - 1) \gamma \delta , \nonumber\\
\zeta &=& s(q \alpha \gamma - \beta \delta) + (s^2 - 1)q \beta \gamma .
\end{eqnarray}
In the case of the $C^*$-algebra $C(SU_q(2))$, $S_{q,s}$ can be made
into $C^*$-subalgebras of  $C(SU_q(2))$, provided $s\in [0,1]$.
The coideal $I = B^+P$ can be computed as
$I=\langle \xi - s , \eta + s , \zeta \rangle P$.
The corresponding quotient coalgebra $C=P/I$ is spanned by
 group-like elements
\begin{equation}
g_0 = \pi_I(1), \qquad g_n = \pi_I( \prod_{k=0}^{n-1} (\alpha + q^k s \beta)),\qquad 
g_{-n} = \pi_I( \prod_{k=0}^{n-1} (\delta + q^{-k} s \gamma)), \label{Hopf.fibr.gen}
\end{equation}
$n = 1,2,\ldots $, where the multiplication increases from  left to right
(e.g., $g_2 = \pi((\alpha +  s \beta)(\alpha + q s
\beta))$, etc.). The coalgebra $C$
can be equipped with  numerous Hopf-algebra structures. For example,
$C$ can be an algebra generated by two elements $Z$ and $Z^{-1}$
 such that
$ZZ^{-1} =1=
Z^{-1} Z$ by setting $Z^n = g_n$ and $Z^{-n} = g_{-n}$. In this way, $C$
can be viewed as an algebra of functions on the circle, i.e. $C =
\sO(S^1)$ (or $C = C(U(1))$ in the $C^*$-algebra case). The constructed
coalgebra-Galois extensions is therefore known as a {\em quantum Hopf
fibration}, since the classical Hopf fibration is a principal bundle
over the two-sphere with the circle as a fibre and the three-sphere
$SU(2)$ as the total space. Note that
 $\pi_I$ 
 is a Hopf-algebra map only if $s=0$. This corresponds to the {\em
standard Podle\'s sphere}.

\subsection{The Galois condition in the setting of corings}
\label{2:subs.coring}

In this section we outline some general properties of entwined modules, view them as a special case of comodules of corings, and then describe a generalisation of Hopf algebras in which the ground field is replaced by a noncommutative algebra. For details we refer to \cite{bw03}. This section has a purely algebraic flavour and can be skipped by a more geometrically oriented reader (mind, however, that 
Example~\ref{2:thm.coring.gauge} is certainly of geometric origin and interest).

\subsubsection{The structure theorems for entwined modules}

Suppose that we have an entwining structure $\Ep$ such that $P$ itself is
a $\Ep$-entwined module. This means in particular that $P$ is a right
$C$-comodule with the structure map $\Delta_P :P\ra P\otimes C$, and
one can consider the subalgebra of
coaction invariants
of $P$, $B = P^{coC}$. In this case the coaction invariants
 have
an alternative description.
\begin{Lem} \label{2:lemma.coinv}
Let $\Ep$ be an entwining structure such that $P$  is
a $\Ep$-entwined module. Then
$$
P^{coC} = \{ b\in P\; |\; \Delta_P(b) = b\Delta_P(1) = b 1\sw 0\otimes
1\sw 1\}.
$$
\end{Lem}
\begin{proof}
Clearly, if $b\in P^{coC}$, then $\Delta_P(b) = b\Delta_P(1)$.
Conversely, if $\Delta_P(b) = b\Delta_P(1)$, then
$$
\Delta_P(bp) = b\sw 0 p_\alpha\otimes b\sw 1^\alpha = b1\sw
0p_\alpha\otimes 1\sw 1^\alpha = b\Delta_P(1p)= b\Delta_P(p)
$$
 for all $p\in P$, as required.
\end{proof}

Note that in an $e$-coaugmented coalgebra-Galois extension,
$\Delta_P(1) = 1\otimes e$. Thus Lemma~\ref{2:lemma.coinv} combined with
Theorem~\ref{2:canonical} imply Proposition~\ref{2:prop.ecoin=coin}, i.e.,
 that in the case of an $e$-coaugmented
coalgebra-Galois $C$-extension both definitions of coaction invariants
 coincide.

Given    an entwining structure $\Ep$ such that $P \in \Mp$, one can
consider two functors between the categories of right
$B=P^{coC}$-modules and the entwined modules. First, there is the {\em
induction functor} $-\otimes_B P : \Mc_B \ra \Mp$, which to each right
$B$-module $N$ assigns a $\Ep$-entwined module $N\otimes_B P$, and to a
right $B$-module morphism $f$, a morphism of entwined modules
$f\otimes_B\id_P$. Here $N\otimes_B P$ is a right $P$-module by
multiplication in $P$, i.e., $(n\otimes_Bp)\cdot p' = n\otimes_B pp'$,
 and a right $C$-comodule with the coaction
$\varrho^{N\otimes_B P} = \id_N\otimes_B \Delta_P$. In the opposite
direction, there is a {\em coaction-invariants functor} $(-)^{coC} :\Mp\ra\Mc_B$
that to each $M\in\Mp$ assigns the right $B$-module
$$
M^{coC} := \{m\in M \;|\; \varrho^M(m) = m\Delta_P(1)\}.
$$
Note that $M^{coC}$ is a right $B$-module by the definition of $B$ as
subalgebra of coaction invariants of $P$. One
can easily show that the coaction-invariants functor is the right adjoint of
the induction functor, i.e., for any $N\in \Mc_B$ and $M\in \Mp$,
there is an isomorphism of vector spaces $\Hom_P^C(N\otimes_B P, M)
\cong \Hom_B(N, M^{coC})$, natural in $M$ and $N$. Here $\Hom_P^C(-, -)$
denotes all right $P$-module right $C$-comodule maps. Perhaps the easiest
way of seeing that this is so is to realise that the coaction invariants
functor can be identified with the $\Hom$-functor, $M^{coC} =
\Hom^C_P(M,P)$, and then use the standard Hom-tensor relations.
Explicitly, the unit  of the adjunction reads, for all $N\in
\Mc_B$, 
$$
\eta_N: N\ra (N\otimes_BP)^{coC}, \qquad n\mapsto n\otimes 1,
$$
while, for all $M\in \Mp$, the counit of the adjunction reads
$$
\sigma_M :M^{coC}\otimes_BP\ra M, \qquad m\otimes p\mapsto
m\cdot p.
$$
The structure theorems for entwined modules deal with the properties of the above
 adjoint pair of functors. The first theorem determines when the coaction-invariants 
functor is 
fully faithful, while the second  theorem determines when the above
adjunction is an equivalence of categories, i.e., when $\eta_N$ and
$\sigma_M$ are  isomorphisms (natural in $M$ and $N$). In this way one
obtains 
a generalised version of Schneider's theorem \cite[3.7 Theorem]{s-hj90a}. 
\begin{Thm}\label{2:thm.str.flat}
Let $\Ep$ be an entwining structure such that $P \in \Mp$, and define
$B:=P^{coC}$. 
Then the following are equivalent:
\begin{blist}
\item $B \subseteq P$ is a coalgebra-Galois $C$-extension, $\psi$ is the canonical
 entwining map associated to  $B\inc P$, and $_B P$ is  flat;
\item the functor $(-)^{coC}: \Mp\ra \Mc_B$ is fully faithful, i.e., for all 
$M\in \Mp$, the counit of adjunction $\sigma_M: M^{coC}\otimes_BP\ra M$ is an 
isomorphism.
\end{blist}
\end{Thm}
\begin{proof}
Assume that $B \subseteq P$ is a coalgebra-Galois $C$-extension,
 and that $P$ is a flat left
$B$-module.  For all $M\in\Mp$, one can view $M\otimes C$ as a right
$P$-module (as a $\Ep$-entwined module in fact) with the structure map
$(m\otimes c)\cdot p = m\cdot p_\alpha\otimes c^\alpha$. In particular,
$P\otimes C$ is a right $P$-module, and the reader can check that the
canonical map $can :P\otimes_B P\ra P\otimes C$ is a right $P$-module
map (i.e., it is a $(P,P)$-bimodule bijection). Now consider the following
commuting
diagram of right $C$-comodule right $P$-module maps
\[
\xymatrix{
0 \ar[r]  &  M^{coC}\otimes_BP  \ar[r] \ar[d]_{\sigma_{M}} & M\otimes_BP \ar[r] \ar[d]_{\id_M\otimes_{P}\can} & (M\otimes C)\otimes_B{P} \ar[d]_{(\id_{M}\otimes \id_C)\otimes_{{P}}\can} \\
0 \ar[r] & M  \ar[r]_{\varrho^M} &  M \otimes  C \ar[r]_{\ell_{MC}} & M\otimes C\otimes C.  \\
}
\]

The maps in the top row are the obvious inclusion and
\[
m\otimes p\mapsto
\varrho^M(m)\otimes_Bp - m\cdot \Delta_P(1)\otimes_Bp,
\] 
while 
$\ell_{MC}=\varrho^M\otimes \id_C - \id_M\otimes \Delta$
is the coaction equalising map.
The top row is exact since it is a
defining sequence of $M^{coC}$ tensored with $P$ over $B$, and the
functor 
$-\otimes_BP$ is
exact. The bottom row is exact too. Since the canonical map $\can$ is a
$(P,P)$-bimodule  and  right $C$-comodule, also the maps
$\id_M\otimes_P\can$ and $(\id_M\otimes \id_C)\otimes_P\can$ are right
$C$-comodule right $P$-module bijections. Therefore,
$\sigma_M$
is an isomorphism in $\Mp$, i.e., $(-)^{coC}$ is fully faithful, as required.

Conversely, assume that $(-)^{coC}$ is fully faithful.
Note that $P\otimes C$ is an entwined module with the action $(p\otimes
c)\cdot p' = pp'_\alpha\otimes c$, and the coaction $\varrho^{P\otimes
C} = \id_P\otimes \Delta$. Therefore, there is a corresponding
counit of adjunction
$\sigma_{P\otimes C}:(P\otimes C)^{coC}\otimes_BP\ra P\otimes C$,
and it is bijective. Next consider the map $\phi: P\ra (P\otimes
C)^{coC}$, given by $p\mapsto p\Delta_P(1)$. This map is well defined
since 
$$
(p1\sw{0'}\otimes 1\sw{1'})\cdot 1\sw 0 \otimes 1\sw 1 =
p1\sw{0'}{1\sw 0}_\alpha\otimes 1\sw{1'}^\alpha \otimes 1\sw 1
= p1\sw 0\otimes 1\sw 1\otimes 1\sw 2 = \varrho^{P\otimes C}(p1\sw
0\otimes 1\sw 1).
$$
Here $1\sw{0'}\otimes 1\sw{1'}$ denotes another copy of $\Delta_P(1)$,
and we used the definition of the $P$-action on $P\otimes C$ and the
fact that $P$ is an entwined module. Clearly, $\phi$ is a left
$P$-module map. It is also right $B$-linear since, for all $b\in B$ and
$p\in P$, we have
$$
\phi(p)\cdot b = (p1\sw 0\otimes 1\sw 1)\cdot b = p1\sw 0b_\alpha\otimes
1\sw 1^\alpha = pb\sw 0\otimes b\sw 1 = pb1\sw 0\otimes 1\sw 1 =
\phi(pb).
$$
Here we used that $P$ is an entwined module and the definition of
coaction invariants in $P$. Finally, $\phi$ is an isomorphism of
$(P,B)$-bimodules with
the inverse $\phi^{-1} = \id_P\otimes \eps$. Now take any $p,p'\in P$
and compute
\begin{eqnarray*}
\sigma_{P\otimes C}\circ (\phi\otimes_B\id_P)(p\otimes_B p') &=&
\sigma_{P\otimes C} (p1\sw 0\otimes 1\sw 1\otimes_B p')  =
(p1\sw 0\otimes 1\sw 1)\cdot p'\\
&=& p1\sw 0 p'_\alpha\otimes 1\sw 1^\alpha = pp'\sw 0\otimes p'\sw 1,
\end{eqnarray*}
i.e., $can = \sigma_{P\otimes C}\circ (\phi\otimes_B\id_P)$. Therefore,
the canonical map is a composition of isomorphisms and hence an
isomorphism as required. Thus the extension $B\subseteq P$ is Galois.

To prove that ${}_BP$ is flat one uses the following argument from the category 
theory. 
Note that both kernels and cokernels of
any morphism in $\Mp$ are  $\Ep$-entwined modules, i.e., $\Mp$
is an
Abelian category. In fact, one can prove that $\Mp$ is a Grothendieck category. 
Therefore, any sequence of $\Ep$-entwined module maps is
exact if and only if it is exact as a sequence of additive maps. Thus to prove that 
${}_BP$ is flat suffices it to show that that the functor
 $-\otimes_BP: \Mc_B\ra \Mp$ is exact. Note, however, that $-\otimes_BP$ is a left 
adjoint of a fully faithful functor, and since $\Mp$ is a Grothendieck category, 
the functor $-\otimes_BP$ is exact by the Gabriel-Popescu theorem 
(cf.\ \cite[Theorem~15.26]{f-c73}).
\end{proof}

\begin{Thm}\label{2:thm.structure}
Let $\Ep$ be an entwining structure such that $P \in \Mp$, and define
$B:=P^{coC}$. 
Then the following are equivalent:
\begin{blist}
\item $B \subseteq P$ is a coalgebra-Galois $C$-extension, $\psi$ is the canonical
 entwining map associated to $C$-extension $B\inc P$ and $_B P$ is faithfully flat;
\item the functor $- \otimes_B P: \Mc_B \ra \Mp$ is a category equivalence.
\end{blist}
\end{Thm}
\begin{proof}
Assume that $B \subseteq P$ is a coalgebra-Galois $C$-extension,
 and that $P$ is a faithfully flat left
$B$-module.  For all  $N\in\Mc_B$, consider the following commutative diagram of right
$B$-module maps
\[
\xymatrix{
0 \ar[r] & N \ar[r] \ar[d]_{\eta_N} & N\otimes_BP \ar[r] \ar@{=}[d] & N\otimes_BP\otimes_BP \ar[d]_{\id_N\otimes_B\can}\\
0 \ar[r] & (N\otimes _BP)^{coC} \ar[r] & N\otimes_BP \ar[r] & N\otimes_BP\otimes C.
}
\]

The maps in the top row are: $n\mapsto n\otimes_B 1$ and $n\otimes_B
p\mapsto n\otimes_B p\otimes_B  1-n\otimes_B 1\otimes_B p$, and the top
row is exact by the faithfully flat descent. The bottom row is the
defining sequence of
$(N\otimes_BP)^{coC}$ and hence is exact. This implies that the unit
of adjunction
$\eta_N$ is an isomorphism in $\Mc_B$. Since ${}_BP$ is  flat, also $\sigma_M$ is an 
isomorphism for all $M\in \Mp$ by Theorem~\ref{2:thm.str.flat}. Therefore
$(-)^{coC}$ and $-\otimes_BP$
are inverse
equivalences. 

Conversely, assume that $(-)^{coC}$ and $-\otimes_BP$
are inverse equivalences. Then by Theorem~\ref{2:thm.str.flat}  the extension 
$B\subseteq P$ is Galois, and ${}_BP$ is flat.  Since
 $-\otimes_BP$ is an equivalence, it also reflects exact
sequences. Therefore $P$ is a faithfully flat left $B$-module.
\end{proof}

\subsubsection{Corings and Galois comodules}

This part of the text is devoted to the introduction of an algebraic
structure which helps to understand better properties of entwining
structures and entwined modules. This structure will play no further
role in studies of geometric aspects of coalgebra-Galois extensions
presented below, but it is currently a subject of intensive studies. 
Interested readers are referred to \cite{b-t02} and
to the 
monograph \cite{bw03}.

\begin{Def}
Let $A$ be an algebra. An $(A,A)$-bimodule $\sC$ is called an
{\em $A$-coring}
if there exist $(A,A)$-bimodule maps
$$\Delta_{\sC}: \sC \ra \sC \otimes_A \sC \qquad \text{and}
\qquad \varepsilon_{\sC}:\sC \ra A
$$
such that
$$
(\DC\otimes _A\id_\sC)\circ\DC = (\id_\sC\otimes_A\DC)\circ\DC, \quad
(\eC\otimes _A\id_\sC)\circ\DC = \id_\sC = (\id_\sC\otimes_A\eC)\circ\DC,
$$
i.e,. $\DC$ and $\eC$  satisfy the axioms for a coproduct and counit.
\end{Def}

The term {\em coring} was coined by Sweedler in the context of a semi-dual version of the
Jacobson-Bourbaki theorem \cite{s-me75}. 
In
late seventies,
corings resurfaced under the
name of {\em bimodules over a category with a coalgebra structure}, or
{\em BOCS}'s for short, in the work of Rojter \cite{r-av80} and Kleiner
\cite{k-m81}
 on algorithms for 
matrix problems. Corings play particularly important role in the theory
of ring extensions, and the canonical example comes from such an
extension.

\begin{Ex}\label{2:ex.cor.Swe}
Consider an algebra extension
$B \hookrightarrow A$. Let $\sC:=A \otimes_B A$ with the obvious
$(A,A)$-bimodule structure. Then $\sC$ is an $A$-coring with
the coproduct
$$\Delta_{\sC}: \sC \ra \sC \otimes_A \sC \cong A\otimes_B A\otimes_B A
  , \qquad  a \otimes_B a^{\prime}
\mapsto a \otimes_B  1 \otimes_B a^{\prime},
$$ 
and the counit
$$
\varepsilon_{\sC}:\sC \ra A  ,
\qquad a \otimes_B a^{\prime} \mapsto a a^{\prime}.
$$
This coring is called the {\em canonical coring} associated to an extension
of algebras
 $B\hookrightarrow A$, or simply the {\em Sweedler coring}.
\end{Ex}

\renewcommand{\labelenumi}{\roman{enumi})}
 For a long time, essentially only two types of
examples of corings truly generalising coalgebras were known ---
one associated to a ring extension as in Example~\ref{2:ex.cor.Swe},
the other associated to a matrix problem. The latter example was also
studied in the context of
differential graded algebras. This lack of examples
obviously hampered the progress in general coring theory. At the end of the nineties,
however, M.\ Takeuchi  made a remarkable observation that connects
entwining structures with corings, and thus provides one with a reach,
new source of examples of corings. More precisely we have (cf.\
\cite[Proposition~2.2]{b-t02})

\begin{Thm}[Takeuchi] \label{2:Takeuchi}
Let $\Es$ be an entwining structure.
Then $\sC:=A \otimes C$ becomes an $A$-coring
with the following structure:
\begin{rlist}
\item $a \cdot (a^{\prime} \otimes c)\cdot a^{\prime \prime} =
    a a^{\prime} \psi(c \otimes a^{\prime \prime})$,
\item $\Delta_{\sC}: \sC \ra \sC \otimes_A \sC \ ,
    \  a \otimes c \mapsto a \otimes c\sw 1 \otimes_A 1_A \otimes c\sw 2$,
\item $\varepsilon_{\sC}: \sC \ra A \ , \  a \otimes c \mapsto \varepsilon(c)
a$.
\end{rlist}
Conversely, if $A$ is an algebra, $C$ a coalgebra, and $\sC=A \otimes C$ has
the structure of an $A$-coring, then an entwining map $\psi$ is given by
$\psi: C \otimes A \ra A \otimes C \ , \  c \otimes a \mapsto (1 \otimes c)
\cdot a$.
\end{Thm}
\begin{proof}
It is obvious that $A\otimes C$ is a left $A$-module with the
specified
action. The following simple calculations,
performed for any $a,a',a''\in A$
and $c\in C$,
$$
(a\otimes c)\cdot  (a'a'') =  a(a'a'')_\alpha\ut c^\alpha =
 aa'_\alpha
a''_\beta \ut c^{\alpha\beta} =  (aa'_\alpha\ut c^\alpha)
a'' =
((a\ut c)\cdot a') \cdot a''
$$
and
$$
(a\ut c)\cdot 1 =  a1_\alpha\ut c^\alpha =a\ut c
$$
prove that $A\otimes C$ is a right $A$-module. Note how the left
pentagon was used to derive the first result and the left triangle
to
obtain the second one. Thus $\sC$ is an $(A,A)$-bimodule.

Next one has to check that $\eC$ and $\DC$ are
$(A,A)$-bimodule maps. Clearly, they are left $A$-linear. Take any
$a,a'\in A$, $c\in C$, and compute
$$
\eC((a\ut c) \cdot a') =  \eC(aa'_\alpha\ut c^\alpha)
=
 aa'_\alpha\eps(c^\alpha) = aa'\eps(c) = \eC(a\ut c)a'.
$$
Here the right triangle was used to establish the penultimate
equality.
Furthermore,
\begin{eqnarray*}
\DC((a\ut c)\cdot a') & = &  aa'_\alpha\ut c^\alpha\sw
1\ut
c^\alpha\sw 2 =  aa'_{\alpha\beta}\ut c\sw 1^\beta\ut c\sw
2^\alpha \\
& = &  (a\ut c\sw 1) \cdot a'_\alpha\ut c\sw 2^\alpha =
 (a\ut c\sw 1)\otimes _A(1\ut c\sw 2) \cdot a' \\
& = &
\DC(a\ut c) \cdot a'.
\end{eqnarray*}
Here the  second equality follows from the right pentagon. Thus
$\eC$ and $\DC$ are $(A,A)$-bimodule morphisms, as
required.
Now, the coassociativity of $\DC$ follows immediately from the
coassociativity of $\Delta$, while the counit property of $\eC$
is
an immediate consequence of the fact that $\eps$ is a counit of $C$.

Conversely, let $\sC = A\otimes C$ be an
$A$-coring with structure maps given
in the theorem.  Denote $\psi(c\ut  a) =
(1\ut c) \cdot a =  a_\alpha\ut
c^\alpha$. Since $\psi$ is defined in terms of the right
multiplication one 
has
$\psi(1\ut c) = c\ut 1$ (the left triangle) and
\begin{eqnarray*}
\psi(c\ut aa') &=& (1\ut c)\cdot (aa') = ((1\ut c)\cdot
a)\cdot a' \\
&=&  (a_\alpha\ut c^\alpha)\cdot  a' = a_\alpha(1\ut
c^\alpha) \cdot a' =  a_\alpha a'_\beta \ut c^{\alpha\beta},
\end{eqnarray*}
(the left pentagon). Furthermore, since $\eC$ is $A$-linear,
$$
 a_\alpha\eps(c^\alpha) = \eC(a_\alpha\ut c^\alpha) =
\eC((1\ut c)\cdot a) =
\eC(1\ut  c)a = \eps(c) a
$$
(the right triangle). Finally, also $\DC$ is right $A$-linear,
so that
\begin{eqnarray*}
 a_{\alpha\beta}\ut c\sw 1^\beta\ut  c\sw 2^\alpha & = &
 (1\ut
c\sw 1)\cdot a_\alpha\ut c\sw 2^\alpha =  (1\ut c\sw 1)\ut
a_\alpha\ut c\sw 2^\alpha \\
& = &  (1\ut c\sw 1)\ut_A (1\ut c\sw
2)\cdot  a \\
&=&\DC(1\ut c) \cdot a = \DC((1\ut
c)\cdot 
a)\\
& = &  \DC(a_\alpha\ut c^\alpha) =  a_\alpha\ut
c^\alpha\sw
1\ut c^\alpha\sw 2
\end{eqnarray*}
(the right pentagon). Therefore, $\psi$ is an entwining map, as claimed.
\end{proof}

As an application of Theorem~\ref{2:Takeuchi} one obtains a quick proof of
the existence of the canonical entwining structure in
Theorem~\ref{2:canonical}. Indeed, for a coalgebra-Galois  $C$-extension $B\inc P$,
Sweedler's canonical coring $\sC =P \otimes_B P$
induces through the bijectivity of $\can:P \otimes_B P \ra P \otimes C$
a coring 
structure on $P \otimes C$. Then  Theorem \ref{2:Takeuchi} implies that
there is an entwining structure $\Ep$. This is precisely  the
canonical entwining structure $\Ep$ from Theorem~\ref{2:canonical}.

Similarly as for coalgebras, one can study the corepresentation theory of
corings. Given an $A$-coring $\sC$, a   {\em right $\sC$-comodule} is a
right $A$-module
$M$,  together with a right $A$-module map $\varrho^M:
M \ra M \otimes_A \sC$ satisfying the following axioms for a coaction:
\begin{eqnarray}
(\id_M \otimes_A \Delta_{\sC}) \circ \varrho^M &=&
(\varrho^M \otimes_A \id_{\sC}) \circ \varrho^M, \nonumber\\
(\id_M \otimes \varepsilon_{\sC}) \circ \varrho^M &=& \id_M\,.
\end{eqnarray}
A morphism of right $\sC$-comodules is a right $A$-module map $f:M\ra N$
such that $\varrho^N\circ f = (f\otimes_A\id_\sC)\circ \varrho^M$.
The category of $\sC$-comodules is denoted by $\Mc^{\sC}$.

The category of comodules of Sweedler's coring is familiar from the
(noncommutative) descent theory (cf.\ \cite{b-f94})

\begin{Ex}(\cite[25.3]{bw03})\label{2:ex.comodule.Sweedler}
The category of comodules over Sweedler's canonical coring $\sC = A \otimes_B
A$ 
is isomorphic to the category of (right) descent data
${\bf Desc}(A/B)$. 
The objects  
in 
category ${\bf Desc}(A/B)$, known as {\em descent data}, are pairs $(M,f)$,
where $M$ is a right
$A$-module, 
and $f: M\ra M\otimes_B A$ is a right $A$-module morphism satisfying
the 
following conditions. Let, for any $m\in M$, $f(m) =
\sum_i m_i\ut_B a_i$.
Then  
\begin{rlist}
    \item $ \sum_i f(m_i)\ut_B a_i = \sum_i m_i\ut_B 1\ut_B a_i$;
    \item $\sum_i m_i\cdot  a_i = m$.
\end{rlist}
A morphism $(M,f)\ra (M',f')$ in ${\bf Desc}(A/B)$ is a right
$A$-module map $\phi: M\ra M'$ such that $f'\circ\phi =
(\phi\otimes_{B}\id_A)\circ f$.
\end{Ex}

 The category ${\bf Desc}(A/B)$ is a noncommutative
generalisation \cite{c-m76} of the category of descent data
associated to an 
extension of commutative rings introduced by Knus and Ojanguren in
\cite{ko74},  and forms a backbone of the noncommutative
extension of the classical {\em descent theory} \cite{g-a5960a}
\cite{g-a5960b}. Descent theory provides answers to the
following types of 
questions. 
\begin{itemize}
    \item {\em Descent of modules}\index{descent of modules}: given an
    algebra   extension $B\ra A$ and a right $A$-module $M$, is
there a right $B$-module $N$ such that $M\cong N\otimes_B A$ as right
$A$-modules? 
\item {\em Classification of $A$-forms}\index{classification
of $A$-forms}: given a right $B$-module $N$,
classify all right $B$-modules $M$ such that $N\otimes_B A \cong
M\otimes_B A$ as right $A$-modules.
\end{itemize}
For recent developments in  the descent theory we refer to
\cite{n-p97}. Thus corings shed a new light and give new tools for studies
of the 
descent theory.

{}From the point of view of coalgebra-Galois extensions, more important is
the following
\begin{Prop}\label{2:ex.entw.cor}
For an entwining structure $\Es$ and its associated coring $\sC= A \otimes
C$,
 the category of $\sC$-comodules $\Mc^{\sC}$ is isomorphic to the category of entwined
 modules $\Me$. Hence, the theory of entwined modules can be viewed as a
special case of corepresentation theory of corings.
\end{Prop}
\begin{proof}
The key observation here is that, if $M$ is a right $A$-module,
    then $M\otimes C$ is a right $A$-module  with the multiplication
    $(m\ut c)\cdot a =  m \cdot a_{\alpha}\ut c^{\alpha}$. The statement
    $M$ is an $(A,C,{\psi})$-entwined module is equivalent to the
    statement that $\varrho^{M}$ is a right $A$-module map.
    Using the canonical identification $M\otimes  C \cong
    M\otimes_{A}A\otimes C = M\otimes_{A}\sC$, one can view a right
    $C$-coaction as a right $A$-module map  $M\ra M\otimes_{A}\sC$,
    i.e., as a right $\sC$-coaction.
    Conversely a right
    $\sC$-coaction can be viewed as a right $A$-module map
    $\varrho^{M}:M\ra M\otimes C$ thus providing a right $\sC$-comodule
    with the structure of an $(A,C,{\psi})$-entwined module.
\end{proof}

In view of Example~\ref{2:ex.entw.cor}, several properties of entwined
modules can be derived from corresponding properties of more general
comodules of a coring. For instance,
 Theorem~\ref{2:thm.str.flat} and Theorem~\ref{2:thm.structure} are
special cases of the structure theorems for
corings  which are isomorphic
to the Sweedler coring and known as {\em Galois corings} (cf.\  \cite[Theorem~5.6]{b-t02} and \cite[28.19]{bw03}).
These, in turn, are special cases of a more general construction which extends both
the canonical Sweedler coring and Galois-type extensions, and leads to structure 
theorems from which structure theorems for entwined modules, such as Theorem~\ref{2:thm.str.flat} and Theorem~\ref{2:thm.structure}, can be derived as corollaries.

\begin{Def}[\cite{eg03}]\label{2:comatrix}
Let $\sC$ be an $A$-coring,  $M$ be a right $\sC$-comodule and let $B=\Rend \sC M$ be the endomorphism ring of $M$.  View $\sC$ as a right $\sC$-comodule with the regular coaction $\DC$. The comodule
$M$ is called a {\em Galois (right) 
 comodule} if $M$ is a finitely generated and
projective right $A$-module, and the evaluation
 map  
 $$ \varphi_\sC: \Rhom \sC M\sC\ot_BM\lra \sC, \qquad  f\ot_B m\longmapsto f(m),$$
is an isomorphism of right $\sC$-comodules. 
\end{Def}

An equivalent definition of Galois comodules is obtained by first noting that $M$ is a $(B,A)$-bimodule and $\Rhom \sC M\sC \simeq M^* = \rhom A M A$ as $(A,B)$-bimodules. If $M_A$ is finitely generated projective, then $M^*\ot_B M$ is an $A$-coring with the coproduct $\Delta_{M^*\ot_B M}(\xi\ot_B m) =  \sum_i \xi\ot_B e^i \ot_A \xi^i\ot_B m$, where $\{e^i\in M,\xi^i\in M^*\}$ is a dual basis of $M_A$, and with the counit $\eps_{M^*\ot_B M}(\xi\ot_B m) =   \xi(m)$ (cf.\ \cite{eg03}). The coring $M^*\ot _B M$ is known as a {\em comatrix coring}. The map  $\varphi_\sC$ reduces to the {\em canonical} $A$-coring morphism
$$
\can_M: M^*\ot_B M\lra \sC, \qquad \xi\ot_B m\longmapsto \sum \xi(m\sw0)m\sw 1.
$$
A finitely generated projective right $A$-module $M$  is a Galois comodule if and only if the canonical map $\can_M$ is an isomorphism of corings. 

To see the relationship of Definition~\ref{2:comatrix} with the coalgebra-Galois theory, take $\sC = A\ot C$, the coring corresponding to an entwining structure $\Es$. Assume $A$ is an entwined module, hence a $\sC$-comodule, and take $M=A$. This is obviously a finitely generated projective right $A$-module. The endomorphism ring $B=\Rend\sC A$ can be easily identified with the coaction-invariant subalgebra, i.e., $B\simeq A^{coC}$. Since $A^*\simeq A$, the corresponding comatrix coring can be identified with the Sweedler coring $A\ot_B A$ of the extension $B\hookrightarrow A$, and the canonical map $\can_A$ with the canonical map of the coalgebra-Galois theory. Thus $A$ is a Galois comodule if and only if the $C$-extension $B\hookrightarrow A$ is Galois. In view of this, Theorem~\ref{2:thm.str.flat} and Theorem~\ref{2:thm.structure} are corollaries of the following Galois Comodule Structure Theorem which, in part, was first formulated in \cite[Theorem~3.2]{eg03}. 

\begin{Thm} [\sc The Galois Comodule Structure Theorem]
\label{2:thm.galoisc}
Let $\sC$ be an $A$-coring and $M$ be a right $\sC$-comodule that 
 is finitely generated and projective as a right $A$-module. Set $B=\Rend \sC M$. 

\begin{zlist}
\item The following are equivalent:
 \begin{blist}
 \item $M$ is a Galois comodule that is flat as a left $B$-module.
 \item $\sC$ is a flat left $A$-module and $M$ is a generator in $\Mc^\sC$.
\item $\sC$ is a flat left $A$-module and, for any $N\in \Mc^\sC$, the  evaluation map $\varphi_{N}: \Rhom \sC M N \ot_BM\to N$, $f\ot_B m\mapsto f(m)$,
is an isomorphism of right $\sC$-comodules.
 \item  $\sC$ is a flat left $A$-module  and the functor $\Rhom \sC M - : \Mc^\sC \to \Mc_B$ is fully faithful. 
 \end{blist}
\item The following are equivalent:
  \begin{blist}
 \item $M$ is a Galois comodule  that is faithfully flat as a left $B$-module. 
 \item $\sC$ is a flat left $A$-module  and $M$ is a projective generator in $\Mc^\sC$.
 \item  $\sC$ is a flat left $A$-module  and $\Rhom \sC M - : \Mc^\sC \to \Mc_B$ is an
        equivalence     
        with the inverse $-\ot_BM:\Mc_B\to \Mc^{\sC}$.
\end{blist}
\end{zlist}
\end{Thm}
For the proof of this theorem we refer to \cite[18.27]{bw03}. The assertion (d) in Theorem~\ref{2:thm.galoisc}~(1) is simply a rephrasing of the assertion (c), since the natural map $\varphi$ is the counit of the adjunction $(-\ot_BM \dashv \Rhom \sC M - )$.

Thus Galois comodules provide one with a very general and conceptually clear point of view on Galois-type extensions. Going further in this direction one intoduces the notion of a {\em principal comodule} as a Galois $\sC$-comodule $M$ which is projective as a (left) module over its coaction-invariant algebra $B=\Rend \sC M$. Principal extensions are examples of such comodules. Principal comodules are characterised by the following theorem proven 
in~\cite{b-t05}:
\begin{Thm}\label{2:thm.princ.field}
Let  $\sC$ be an $A$-coring and $M$ a right $\sC$-comodule that is finitely generated and projective as a right $A$-module. Set $B=\Rend \sC M$.
\begin{zlist}
\item If $M$ is a principal comodule, then it is faithfully flat as a left $B$-module.
\item 
 View  $M^*\ot M$ as a left $\sC$-comodule via ${}^{M^*\!\!}\varrho\ot \id_M$, where ${}^{M^*\!\!}\varrho: M^*\ra \sC\ot_AM^*$ is the left coaction induced from the right $\sC$-coaction $\varrho^M:M\ra M\ot_A\sC$. Then the following statements are equivalent:
\begin{blist}
\item The map
$
\widetilde{\can}_M :M^*\ot  M\ra \sC$,  $\xi\ot m\mapsto\sum \xi(m\sw 0)m\sw 1,
$
is a split epimorphism of left $\sC$-comodules.
\item $M$ is a principal right $\sC$-comodule.
\end{blist}
\end{zlist}
\end{Thm}

Several generalisations of Galois comodules discussed above have been introduced recently. Comatrix corings built out of modules which  are not  required to be finitely generated projective and the corresponding Galois comodules are discussed in \cite{eg04}. This led in a natural way to considering corings and Galois comodules over firm rings without a unit in \cite{gv07}. Going in a slightly different direction, the definition of a Galois comodule was proposed in \cite{w-r06}. It is required there that a certain morphism of functors,
$
\rhom A M - \otimes_B M \ra -\otimes_A\sC,
$
 where $\sC$ is an $A$-coring, $M$ is a right $\sC$-comodule,
  and $B$ is its endomorphism algebra, is an isomorphism. 
 
Yet another approach to the Galois condition for corings, deeply rooted in noncommutative (affine) algebraic geometry was recently developed by T.\ Maszczyk \cite{m-t}. The idea is to depart from considering comodules and  to start with a morphism of two $A$-corings
$
\widetilde{\gamma} : \sD\ra \sC.
$
With such a morphism one can view $\sD$ as a $\sC$-$\sC$ bicomodule in a natural way and then define a ring $B = \LRhom\sC\sC\sD\sC$. $\sD$ is a $B$-$B$ bimodule and one can consider the quotient $A$-coring $\sD/[\sD,B]$. The map $\gamma$ descends to this quotient, and one says that a Galois condition is satisfied if the resulting $A$-$A$ bimodule map $\gamma:\sD/[\sD,B]\ra \sC$ is bijective (an isomorphism of $A$-corings). In particular, if one starts with a right $\sC$-comodule $M$ that is finitely generated and projective as a right $A$-module and sets $\sD := M^*\otimes M$ and $\widetilde{\gamma}:=\widetilde{\can}_M$, then $B$ is isomorphic to $\Rend \sC M$, and  $\sD/[\sD,B]$ can be identified with $M^*\otimes_B M$. The  resulting $\gamma$ is the canonical map ${\can}_M$.

All these definitions of the Galois condition for corings can be understood as instances of a {\em Galois condition for comonads} in general categories. The latter is discussed in detail in \cite{g-tj06}.

\begin{Rem} \label{2:rem.grouplike}
We have already discussed the problem that one cannot use
group-like elements in a coalgebra $C$ when
defining invariants in the context of a coalgebra Galois
extension $C$-extension $B\inc P$.
On the other hand, studying the corresponding coring
$\sC= P \otimes C$, as given in Theorem~\ref{2:Takeuchi},
yields the advantage that there actually exists a
group-like 
element $g:=\Delta_P(1) \in P \otimes C$. A group-like element in a
general 
$P$-coring $\sC$ is defined in an analogous way as in a coalgebra, i.e.,
$g\in\sC$ is group-like, if $\DC(g) = g\otimes_P g$ and $\eC(g) =1$. One
easily checks that in the case $g=1\sw 0\otimes 1\sw 1$
\begin{eqnarray*}
g\otimes_P g &=& 1\sw {0'}\otimes 1\sw {1'}\otimes_P 1\sw 0\otimes 1\sw 1 =
(1\sw {0'}\otimes 1\sw {1'})\cdot 1\sw 0\otimes 1\sw 1\\
& =&
1\sw {0'}{1\sw 0}_\alpha\otimes 1\sw {1'}^\alpha \otimes 1\sw 1=
1\sw 0\otimes 1\sw 1\otimes 1\sw 2 \\
&=& \DC(1\sw 0\otimes 1\sw 1),
\end{eqnarray*}
where we used that $P$ is an entwined module, and, obviously, $\eC(1\sw
0\otimes 1\sw 1) = 1\sw 0\eps(1\sw 1) = 1$. Therefore $\Delta_P(1)$ is
a group-like element in a $P$-coring $P\otimes C$. In view of
Lemma~\ref{2:lemma.coinv},   the coaction invariants of $P$ are
defined as $\{b \in P \mid \Delta_P(b)=b \otimes_P g \}$.
\end{Rem}

Although the knowledge of corings is not essential for the studies of
(noncommutative) geometry of coalgebra-Galois extensions, and in what
follows we will make some simplifying assumptions which will make
working with coalgebra-Galois extensions and entwining structures more
pleasant, corings are a very useful device, which allows one to view
complicated algebraic structure through much simpler and more familiar
objects (our intuition trained on coalgebras will also work for corings
in the majority of cases). In the context of noncommutative geometry, it
seems to be worthwhile to mention an interesting relationship between
corings with a group-like element and differential graded algebras (see \cite{r-av80} or
\cite[Section~29]{bw03} for more details). Given an $A$-coring $\sC$ with a group-like
element $g$, one can introduce the structure of a graded differential
algebra $\Omega^\bullet$ with $\Omega^0 =A$ and $\Omega^1 = \Ker\eC$.
Given a right $A$-module $M$ one can study connections on $M$ with
values in this graded differential algebra (see below for the definition
of a connection). It turns out that $M$ is a right $\sC$-comodule if and
only if it admits a flat connection. Thus the representation-theoretic
notion of a comodule of a coring has a deep and somewhat unexpected
(noncommutative)
differential-geometric meaning.

\subsubsection{Quantum groupoids}

In noncommutative geometry, Hopf algebras are understood as  deformed algebras of functions 
on groups. In algebra and geometry (in particular in Poisson geometry 
and the geometry of foliations), 
in addtion to groups, one also considers groupoids (cf.\ \cite{m-k05}, \cite{mm03}), 
which are roughly 
defined as groups for which not every two points can be multiplied together. Precisely, 
a {\em groupoid} is a small category in which any morphism is an isomorphism. Thus 
underlying a groupoid are two sets: a set of points (objects) and a set of invertible arrows (maps). An arrow can be composed with another arrow only if the 
head of one of them coincides with the tail of the other. Thus not any pair of arrows can 
be composed together. On the other hand, if the groupoid has a single point (object), then the head of any arrow must be the same as its tail, hence any 
two maps can be composed with each other, and, since they are assumed to be isomorphisms, 
the set of arrows forms a group. A typical example of  a groupoid is a fundamental groupoid of a manifold: the points are points of the manifold, the arrows are homotopy classes of paths, and the product is induced from the concatenation of paths.

As can be seen from the above example, groupoids have a strong geometric flavour. In fact, one 
can associate a groupoid to any principal bundle (this is known as a {\em gauge} or 
{\em Ehresmann} groupoid).  An algebra of functions on a groupoid, however, is no longer a Hopf 
algebra. First, a groupoid is built on two sets and each of them gives rise to an algebra
 of functions. Second, recall that a product of elements of a group gives rise to a coproduct 
in the algebra of 
functions on it. But in the case of a groupoid, not every two arrows can be composed. As a result, 
the product does not provide the algebra of functions on a groupoid with a coalgebra structure, 
but with the structure of a coring over the algebra of functions on points. Thus an algebra of functions on a groupoid is also a coring with a suitable compatibility conditions between the product and coproduct. Noncommutative generalisation of this structure leads to the notions of a {\em bialgebroid}  and a {\em Hopf algebroid} or a {\em quantum groupoid}.

What makes the definition of a bialgebroid non-trivial is the fact that if $A$ is an 
algebra and $\sH$ is an $A$-bimodule and an algebra, the tensor product $\sH\ot_A \sH$ 
is not necessarily an algebra. Thus one cannot require that the coproduct 
$\DH: \sH\to \sH\ot_A \sH$ is an algebra map. A different compatibility condition must be used. Over the years, different authors have found different solutions to this problem. First Sweedler \cite{s-me75a} (in the case of a commutative base) and Takeuchi \cite{t-m77} (in the case of a general base algebra) proposed to restrict the image of $\DH$ to a subbimodule of   
$\sH\ot_A \sH$ on which the algebra structure is well-defined, and then to require it to 
be an algebra map. Later and independently, Lu \cite{l-jh96} proposed to replace the 
algebra condition for $\DH$ with a weaker algebraic condition. Most recently, 
Xu \cite{x-p01} suggested a different definition supported with an additional map
 (an anchor). Amazingly, all these seemingly different solutions lead to the same 
algebraic structure (see \cite{bm02}). The conceptual understanding of the definition 
of a bialgebroid $\sH$ in terms of a monoidal structure of the category of its (left) 
modules has been provided by Schauenburg \cite{s-p98}. Our presentation here 
is based on \cite[Section~31]{bw03}, where more details can be found.

Given a $k$-algebra $A$, an  {\em $A$-ring} or an {\em algebra over $A$}  is a pair
$(U, i)$, where $U$ is a $k$-algebra  and $i: A\to U$ is an algebra map.
If $(U, i)$ is an $A$-ring, then $U$ is an $(A,A)$-bimodule  with the
structure provided by the map $i$,
$au a':= i(a)ui(a')$. A map  of $A$-rings $f: (U, i)\to (V,j)$
is a
$k$-algebra map $f:U\to V$ such that $f\circ i=j$. 

Let $\bar{A}=A^{op}$ be the opposite algebra of $A$.
For $a\in A$, $\bar{a}\in \bar{A}$ is the same $a$ but now viewed as an
element in $\bar{A}$, that is, $a\mapsto \bar{a}$ is an (obvious)
anti-isomorphism of
 algebras. Let
$
A^{e}=A\ot \bar{A}
$
be the enveloping algebra of $A$.
Then a pair $(H,i)$ is  an $A^{e}$-ring if and
only if there exist an algebra map $s:A\to H$ and an anti-algebra map
$t: A\to H$,
 such that
 $s(a)t(b)=t(b) s(a)$,
for all $a,b\in A$. Explicitly, $s(a) = i(a\ot 1)$ and $t(a) =
i(1\ot \bar{a})$, and, conversely,  $i(a\ot \bar{b})=s(a)t(b)$. In the case of an $A^e$-ring $H$, 
$A$ is called a
{\em base algebra}, $H$ a {\em total
algebra}\index{total algebra}, $s$ a {\em source}\index{source map}
map and $t$ a
 {\em target} map\index{target map}. To indicate explicitly the source and target maps we write $(H,s,t)$.

Take a pair of $A^e$-rings $(U, s_U,t_U)$ and $(V, s_V,t_V)$, view $U$ as a right $A$-module via the left multiplication by the target map $t_U$, and view $V$ as a left $A$-module via the left multiplication by the source map $s_V$. A {\em Takeuchi $\times_A$-product} is then defined as
%
%
$$
\begin{array}{c}
 U \times_A V 
 :=
\{\sum\limits_{i} u_i\ot_A v_i \!\in\! M\!\ot_A\! N \; | \; \forall\, b\!\in\!
A, \;
 \sum\limits_{i} u_it_U(b) \ot v_i = \sum\limits_i u_i \ot v_i s_V(b) \} .
\end{array}
$$
%
The importance of the notion of the Takeuchi $\times_A$-product is a direct consequence of  the  
following observation (\cite[Proposition~3.1]{s-me75}, \cite[Proposition~3.1]{t-m77}). 
\begin{Lem}
\label{2:lemma.x-product}
 For any pair of $A^e$-rings $(U, i)$ and $(V, j)$, the  
$(A^e,A^e)$-bimodule 
$U\times_A V$ is an   $A^e$-ring with  
the  algebra map   $A\ot  \bar{A} \ra U\times_A V$,  
$a\otimes \bar{b}\to i(a) \ot  
j(\bar{b})$, the associative product  
$$ \begin{array}{c} 
(\sum_i u^i\ot v^i)(\sum_j \tilde u^j\ot \tilde v^j) =  
\sum_{i,j} u^i\tilde u^j\ot v^i  
\tilde v^j,  
\end{array}$$  
and the unit $1_U\ot 1_V$.  
\end{Lem}  

We are now ready to define bialgebroids.

\begin{Def}\label{2:def.Lu} 
Let $(\sH, s, t)$ be an $A^{e}$-ring. View $\sH$ as an
$(A,A)$-bimodule, with the left $A$-action given by the source map $s$,
and the right $A$-action
that descends from the left $\bar{A}$-action given by the target map
$t$, that is, 
$$
a h = s(a) h, \quad h a = t(a)h, \quad \mbox{ for all }a\in A, h\in \sH.
$$
We say that 
$(\sH, s, t, \DH, \eH)$ is
an {\em $A$-bialgebroid}\index{bialgebroid} if
\begin{zlist}
\item $(\sH,\DH, \eH)$ is an $A$-coring;
\item ${\rm Im}(\DH) \subseteq \sH\times_A\sH$ and the corestriction of
      $\DH$ to $\DH :\sH \to \sH\times_A\sH$ is an algebra map;
\item $\eH(1_\sH)=1_A$, and, for all $g$, $h\in \sH$,
$$
\eH(gh) =\eH \left( g s( \eH (h) ) \right) =
\eH \left( g t( \eH (h) ) \right).
$$
\end{zlist}
%
\end{Def}

There are many examples of bialgebroids. For instance, if $H$ is a bialgebra and $A$ is an algebra, then $\sH= A\ot H\ot \bar{A}$ is a bialgebroid over $A$ with the natural tensor algebra structure and:
 \begin{rlist}
	 \item the source map $s: a\mapsto a\ut 1_{H}\ut 1_{A}$;
	 \item the target map $t: a\mapsto 1_{A}\ut 1_{H}\ut \bar{a}$;
	 \item the coproduct $\DH : a\ut b\ut \bar{a'}\mapsto  \sum a\ut b\sw 1\ut 
	 1_{A}\ut 1_{A}\ut b\sw 2\ut \bar{a'}$;
	 \item the counit $\eH: a\ut b\ut \bar{a'}\mapsto \eps(b)a a'$;
 \end{rlist}

The definition of a bialgebroid we present here is by now generally accepted. 
On the other hand, there seems to be no consensus as to how
 a Hopf algebroid should 
be defined. In \cite{l-jh96},  an anti-algebra map
$\kappa :\sH \to \sH$ such that
\begin{rlist}
\item $\kappa \circ t =s$;
\item $\mu_\sH \circ (\kappa \ot\id_\sH) \circ \DH = t\circ \eH
\circ \kappa$;
\item there exists a section $\gamma :\sH\otimes_A \sH \to \sH\otimes_k \sH$
of
 the natural
projection  $\sH\otimes_k \sH \to \sH\otimes_A \sH$ such that
$\mu_\sH \circ (\id_\sH\ot \kappa)\circ \gamma \circ \DH = s\circ \eH$,
\end{rlist}
is called an {\em antipode}.  A bialgebroid with an antipode in this sense is often referred to as a {\em Lu-Hopf algebroid}. Another definition 
of a Hopf algebroid has been proposed by Schauenburg in \cite{s-p00a}. For any 
$A$-bialgebroid $\sH$, the category of its left modules is a monoidal category 
such that the forgetful functor to the category of $A$-bimodules is strict monoidal. 
Schauenburg defines a Hopf algebroid as a bialgebroid for which this functor 
preserves also the closed structure. In algebraic terms, one requires that the 
`canonical map' $\sH\ot_{\bar{A}}\sH\ra \sH  \otimes_A\sH$, 
$g\ot_{\bar A} h \mapsto g\sw 1\ot_A g\sw 2h$ be invertible. One refers to such bialgebroids as {\em $\times_A$-Hopf algebras}. This definition, however, does not lead to an 
antipode as a map $\sH\to \sH$.

The most symmetric and closest to the Hopf algebra case
 definition of a Hopf algebroid  was proposed 
by B\"ohm and Szlach\'anyi in \cite{bs03}, \cite{b-g05}. This definition starts with the observation that, in order to define an antipode which would have similar properties
to those of an antipode in a Hopf algebra, one needs to {\em symmetrise} the
definition of a bialgebroid. Note that the Definition~\ref{2:def.Lu} is not symmetric in the sense that it
considers $\sH$ as an $A$-bimodule with the actions obtained by the {\em left}
multiplication with the source and target maps. One therefore refers to the notion 
defined in Definition~\ref{2:def.Lu} more precisely as a {\em left $A$-bialgebroid}. Obviously, it is possible to define a similar object by using the right multiplication by
the source and target maps. Such an object is then called a {\em right $A$-bialgebroid}. The notion of a Hopf algebroid proposed in \cite{bs03}, \cite{b-g05}
requires existence of  both left f right bialgebroid structures on the same algebra.
\begin{Def}\label{2:def.BS}
A {\em Hopf algebroid} consists of a left $L$-bialgebroid $\sH$ with structure maps $s_L$, $t_L$, $\Delta_L$, $\eps_L$, a right $R$-bialgebroid $\sH$ with structure maps $s_R$, $t_R$, $\Delta_R$, $\eps_R$, and a map $S:\sH\to \sH$
 satisfying the following conditions:
\begin{blist}
\item source-target compatibility,
$$
s_L\circ \eps_L\circ t_R=t_R,\quad t_L\circ \eps_L\circ s_R=s_R,
\quad 
s_R\circ \eps_R\circ t_L=t_L,\quad t_R\circ \eps_R\circ s_L=s_L;
$$
\item colinearity of coproducts,
$$
(\Delta_L \tens_{R} \id_\sH)\circ \Delta_R=(\id_\sH\tens_{L} \Delta_R)\circ \Delta_L,
\qquad
(\Delta_R \tens_{L} \id_\sH)\circ \Delta_L=(\id_\sH\tens_{R} \Delta_L)\circ \Delta_R;
$$
\item the $R\ot L$-bilinearity of the antipode,
$$
S\big(t_L(l) h t_R(r)\big)=s_R(r) S(h) s_L(l),
\qquad
S\big(t_R(r) h t_L(l)\big)=s_L(l) S(h) s_R(r),
$$
for all $r\in R$, $l\in L$ and $h\in \sH$;
\item antipode axioms,
$$
\mu_\sH\circ (S\tens_{L} \id_\sH)\circ \Delta_L=s_R\circ \eps_R,
\qquad 
\mu_\sH\circ (\id_\sH\tens_{R} S)\circ \Delta_R=s_L\circ \eps_L.
$$
\end{blist}
\end{Def}

Note that the axiom (b) in Definition~\ref{2:def.BS} can be stated since condition (a) implies that $\Delta_L$ is $R$-bilinear and $\Delta_R$ is $L$-bilinear. Furthermore, Definition~\ref{2:def.BS} implies that $R$ is isomorphic to the opposite algebra of $L$,
 so that  $L=A$ and $R=\bar{A}$ with no loss of generality. It is also proven
in \cite[Proposition~2.3]{b-g05} that the antipode in a Hopf algebroid is an antimultiplicative and an anticomultiplicative map (in the latter case both coproducts must be used).  B\"ohm and Szlach\'anyi show that their definition 
is not equivalent to that of Lu, by constructing an explicit example of a Hopf algebroid which is not a Lu-Hopf algebroid (cf.\ \cite[Example~4.9]{bs03}).   Many constructions familiar in Hopf algebra can be extended to Hopf algebroids (see \cite{b-g08} for a review). From the geometric point of view, the discussion of strong connections in Hopf algebroid extensions and the construction of the 
relative Chern-Galois character in \cite{bb05} might be of particular interest.

Yet another definition of a Hopf algebroid has been proposed by Day and
 Street in \cite{ds} in the framework of  monoidal 
   bicategories. When specialised to the monoidal bicategory of 
   $k$-algebras,  bimodules and bimodule maps, this definition  is shown in \cite[Theorem~4.7]{bs03} to be very close to (but slightly more restrictive than) the 
B\"ohm-Szlach\'anyi definition.

There are many examples of Hopf algebroids in the sense of Definition~\ref{2:def.BS}.
It is shown in \cite{en01} that weak Hopf algebras (with a bijective antipode) of 
B\"ohm, Nill and Szlach\'anyi \cite{bns99} are examples of bialgebroids. This result is 
then refined in \cite{s-p02}, where it is shown that weak bialgebras are bialgebroids, 
while weak Hopf algebras are Hopf $\times_A$-algebras in the sense of Schauenburg.
In \cite[Example~4.8]{bs03}, weak Hopf algebras are shown to be Hopf algebroids. Furthermore, 
it is shown in \cite{ks03} that one can associate a bialgebroid to any depth-2 algebra 
extension. In case a depth-2 extension is Frobenius, the corresponding bialgebroid is a Hopf algebroid \cite{bs03}. A class of Lu-Hopf algebroids associated  to braided commutative algebras is 
constructed in \cite{bm02}. These are shown to be Hopf algebroids in \cite[Example~4.14]{bs03} (provided the antipodes are bijective).
Other examples of Hopf algebroids in the sense of Definition~\ref{2:def.BS} considered in \cite{bs03} include  quantum torus and the Connes-Moscovici  twisted Hopf algebras 
 \cite{cm01}, \cite{cm04} (cf.\ \cite{b-g03}).
 On the mathematical physics side, quantum groupoids arise from solutions to quantum dynamical Yang-Baxter equations (in the guise of {\em dynamical quantum groups} cf.\ \cite{ev98}) as well as symmetries of certain models in statistical physics (in the guise of {\em face algebras} cf.\ \cite{h-t93}). From the point of view of Galois-type extensions, the 
following example appears to be most significant.

\begin{Ex} \label{2:thm.coring.gauge}
Given 
a coalgebra-Galois $C$-extension $B\inc P$  with  
the translation map
 $\tau$,  consider
 a $B$-bimodule 
$$
\sC = \{ {\sum}_i\, p^i\ut \tilde{p}^i\in
P\otimes P\; |\; {\sum}_i\, p^i\sw 0\ut \tau(p^i\sw 1)
\tilde{p}^i = {\sum}_i\, p^i\ut \tilde{p}^i\otimes_B 1\}.
$$
If  $P$ is faithfully flat as a left 
$B$-module, then $\sC$ is a
$B$-coring with the coproduct and counit
$$
\DC ({\sum}_i\, p^i\ut \tilde{p}^i) = {\sum}_i\, p^i\sw 0\ut
\tau (p^i\sw 1) \ut \tilde{p}^i, \quad \eC ({\sum}_i\, 
a^i\ut
\tilde{p}^i) = {\sum}_i\, p^i\tilde{p}^i.
$$
The $B$-coring $\sC$ is called the {\em Ehresmann} or {\em gauge} 
coring associated to a (left faithfully flat) coalgebra-Galois $C$-extension $B\inc P$.

The Ehresmann coring associated to a (left faithfully flat) Hopf-Galois $H$-extension $B\inc P$ is a $B$-bialgebroid with the algebra structure of a subalgebra of  $P^{e}$,  the  source map $s: p\mapsto p\ut 1$ 
and the target map $t: p\mapsto 1\ut \bar{p}$.
\end{Ex}

For details of the proof we refer to \cite[Section~34]{bw03}. The Ehresmann bialgebroid corresponding to a Hopf-Galois extension was constructed in \cite{s-p98}. Both the Ehresmann coring and bialgebroid can be seen as a noncommutative version of the Ehresmann or gauge groupoid that can be associated to any principal bundle (cf.\ \cite[Example~5.1(8)]{mm03}).

\section{Connections and associated modules}

Recall that, from a geometric point of view, coalgebra-Galois extensions can be
viewed as noncommutative principal bundles. Motivated by
this relationship, one can develop geometric-type objects such as
connections, sections of associated vector bundles, etc. These have
physical meaning of gauge potentials (connections) or matter fields
(sections). The aim of this section is to present differential-geometric aspects of coalgebra-Galois and principal extensions.

We begin in Section~\ref{2:subs.con} by introducing connections in coalgebra-Galois extensions. We give various descriptions of such connections. We then proceed in Section~\ref{2:subs.ass} to define modules associated to coalgebra-Galois extensions, which  can be understood as modules of sections of noncommutative associated
vector bundles. In Section~\ref{2:subs.gau} gauge transformations of coalgebra-Galois extensions are discussed. Finally, in Section~\ref{2:subs.str} strong connections on principal extensions are studied. These are objects that induce connections on associated modules, and thus guarantee the projectivity of the latter.  This makes the desired link with $K$-theory.

\note{Although the theory of connections can be
developed for any coalgebra-Galois extensions it is much easier (and
fully sufficient from the point of view of examples which we have in
mind) to describe it for  $e$-coaugmented $C$-Galois extensions
$e$-coaugmented $C$-extension $B\inc P$. Thus we restrict ourself to such 
extensions only. Further
on we will make an additional assumption about these extensions.
}

\subsection{General coalgebra-Galois extensions}

\subsubsection{Connections}\label{2:subs.con}

If an algebra $P$ is a comodule of a coalgebra $C$, we would like
 to establish covariance properties of $\Op$, i.e., we wish to define
a $C$-coaction on $\Op$. Although in general the coaction
$\Delta_P:P \ra P \otimes C$ does not extend to $\Op$,
it turns out that such an extension
is possible for an entwining structure $\Ep$ with $P \in \Mp$.
\begin{Prop}(\cite[Proposition~2.2]{bm00}) \label{2:prop.univ.cov}
Consider an entwining structure $\Ep$ with $P \in \Mp$, and
define a map
\[\label{2:dsr}
\Delta_{P \otimes P}: P \otimes P \lra P \otimes P \otimes C \ ,
\ p \otimes p^{\prime} \longmapsto p\sw 0 \otimes \psi(p\sw 1 \otimes
p^{\prime}).
\]
The homomorphism $\Delta_{P \otimes P}$ 
is a $C$-coaction on $P \otimes P$, and it restricts to a coaction
$\Delta_{\Op}: \Op \ra \Op \otimes C \ .$
\end{Prop}
\begin{proof}
The fact that $\Delta_{P \otimes P}: P \otimes P \ra P \otimes P \otimes
C$ defines a $C$-coaction  can be
checked using the right hand
side of the bow-tie diagram in Definition \ref{2:ES}. The fact that
$\Delta_{P \otimes P}$ restricts to the coaction $\Delta_{\Op}: \Op \ra \Op \otimes C$
 follows from the left hand side of the bow-tie diagram. Explicitly,
since for all
$\sum_i p_i \otimes p^{\prime}_i \in \Op$,  $\sum_i p_i  p^{\prime}_i
=0$, we obtain
\[
\sum_i p_{i (0)} \psi(p_{i (1)} \otimes p^{\prime}_i) =
\Delta_P(\sum_i p_i p^{\prime}_i)= \Delta_P(0)=0.
\]
Hence $\Delta_{P\otimes P}(\Op)\subseteq \Op\otimes C$, as required.
\end{proof}

Thus, in the case of a canonical entwining structure associated to a
coalgebra-Galois 
$C$-extension $B\inc P$, both $P \otimes P$ and $\Op$ become right
$C$-comodules. 
This observation is crucial for the following:
\begin{Def}\label{2:def.con}
Let $B\inc P$ be a coalgebra-Galois 
 $C$-extension. 
\begin{zlist}
\item A {\em connection} 
is a left $P$-module projection $\Pi: \Op \ra \Op$ such that
\begin{rlist}
\item $\ker(\Pi)=P(\Omega^1B)P$ (horizontality),
\item $\Pi \circ \d: P \ra \Op \ $ is a right $C$-comodule map
(covariance property).
\end{rlist}
\item A {\em connection form} is a homomorphism
$\omega: C \ra \Op$ such that
\begin{rlist}
\item $1\sw0\omega(1\sw1)=0$,
\item for all $c\in C$, $(\widetilde{\can} \circ \omega)(c)= 1 \otimes c - 1\sw0 \ot 1\sw1\varepsilon(c)$,
\item $(\id\ot\psi)\ci(\psi\ot\id)\ci(\id \otimes \omega)\ci\hD
=(\omega\ot\id)\ci\hD$.
\end{rlist}
\end{zlist}
\end{Def}

In view of the fact that the definition of a coalgebra-Galois extension
is equivalent the exactness of the sequence (\ref{2:sequence}), we can interpret
connections as left-linear splittings
of the sequence (\ref{2:sequence}) with the covariance property Definition~\ref{2:def.con}(1)(ii). This leads to the following (cf.\ \cite[Proposition~3.3]{bm00}):
\begin{Thm}\label{2:con.form}
Let $B\inc P$ be coalgebra-Galois 
 $C$-extension. Write $\tau(c) = c\su 1\ot_B c\su 2$ for the action of the translation map on any $c\in C$. The following formulae
\begin{equation}\label{2:omegaPi}
\Pi\longmapsto\ho_\Pi,\;\;\; \ho_\Pi(c)=c\su 1\Pi(\d c\su 2),
\end{equation}
\begin{equation}\label{2:Piomega}
\ho\longmapsto\Pi^\omega,\;\;\; \Pi^\omega(p\d p')=pp'\sw0\ho(p'\sw1),
\end{equation}
define mutually inverse maps between the space of connections $\Pi$ and the  space
of connection forms $\omega$.
\end{Thm}
\begin{proof}
Let $\omega$ be a connection form. The map $\Pi^\omega$ is well 
defined because, for all $p\in P$, $\Pi^\omega(p\d(1)) = p1\sw 0\omega
(1\sw 1) =0$, by Definition~\ref{2:def.con}(2)(i). For any $p,p'\in P$, $b\in B$,
$$
\Pi^\omega(p\d(b)p')= \Pi^\omega(p\d (bp')) - \Pi^\omega(pb\d(p')) = p(bp')\sw 0 \omega
((bp')\sw 1) - pbp'\sw 0\omega(p'\sw 1) =0,
$$
as $\Delta_P$ is left $B$-linear.
On the other hand, if $\sum_i p^i\d( q^i)\in
\Ker\Pi^\omega$, then using Definition~\ref{2:def.con}(2)(ii) we can compute
$$
0 = \sum_i\widetilde{\can}(p^iq^i\sw 0\omega(q^i\sw 1)) = \sum_i(p^iq^i\sw 0
\otimes q^i\sw 1
- p^iq^i1\sw 0\otimes 1\sw 1) = \sum_i\widetilde{\can}(p^i\d( q^i)).
$$
Since $\Ker\widetilde{\can} = P(\Omega^1B) P$, we deduce that $\Ker \Pi^\omega \subseteq P
(\Omega^1 B) P$,
i.e., $\Ker\Pi^\omega = P(\Omega^1B)P$. Finally,  write $\psi^2$ for $(\id\ot\psi)\ci(\psi\ot\id)$, and compute, for all $p\in P$,
\begin{eqnarray*}
\Delta_{\Omega^1P}(\Pi^\omega(\d( p))) & = & p\sw 0\psi^2(p\sw 1
\otimes \omega(p\sw 2)) \\
& = & p\sw 0\omega(p\sw 1)\otimes p\sw 2 \qquad
\qquad \mbox{\rm (by Definition~\ref{2:def.con}(2)(iii))}\\
& = & \Pi^\omega(\d( p\sw 0))\otimes p\sw 1,
\end{eqnarray*}
so that $\Pi^\omega\circ\d$ is a right $C$-comodule map, as required.

Conversely, let $\Pi$ be a connection  in a coalgebra-Galois $C$-extension 
$B\inc P$. In particular, $\Pi$ is left $P$-linear and $\Ker\Pi = P (\Omega^1B)P$, 
so that, for all $b\in B$ and $p\in P$,
$$
\Pi(\d(bp)) = \Pi((\d b) p) + b\Pi(\d p) = b\Pi(\d p).
$$
Therefore, the map $\omega_\Pi$ in equation (\ref{2:omegaPi}) is well defined 
(despite the fact that the differential $\d$ is not a left $B$-module map). Using the 
Translation Map Lemma Lemma~\ref{2:translation}(iii), one immediately finds
$$
1\sw 0\omega_\Pi(1\sw 1) = 1\sw 0 1\sw 1\su 1\Pi(\d 1\sw 1\su 2) = \Pi(\d 1) =0.
$$
Hence $\omega_\Pi$ satisfies condition Definition~\ref{2:def.con}(2)(i).

Next, define a left $P$-linear map $\sigma_\Pi: P\ot C^+\to \Omega^1P$, 
$p\ot c \mapsto p\omega_\Pi(c)$. Again employ Lemma~\ref{2:translation}(iii) 
to note that, for all $p,p'\in P$,
$\Pi(p\d p') = pp'\sw 0\omega_\Pi(p'\sw 1)$. A short calculation  reveals that
$ \Pi (p\d p') = \sigma_\Pi(\widetilde{\can}(p \d p'))$, i.e.,
$$
\Pi = \sigma_\Pi\circ \widetilde{\can}.
$$
Since $\widetilde{\can}$ is an epimorphism and, by the assumption on $\Pi$ and the exactness of sequence (\ref{2:sequence}), $\Ker\Pi = P (\Omega^1B)P = \Ker\widetilde{\can}$, this implies that the map $\sigma_\Pi$ is a monomorphism. Indeed, suppose $\sigma_\Pi(x) =0$ for some $x\in P\ot C^+$. Then there exists $y\in \Omega^1P$ such that $x=\widetilde{\can}(y)$. Therefore,
$$
\Pi(y) = \sigma_\Pi(\widetilde{\can}(y)) = \sigma_\Pi(x) =0,
$$
i.e., $y\in \Ker\Pi = \Ker\widetilde{\can}$, so that
 $x = \widetilde{\can}(y) =0$. Now, as $\Pi$ is assumed to be a projection, we have
$$
\sigma_\Pi\circ \widetilde{\can} = \Pi =\Pi\circ\Pi = \sigma_\Pi\circ \widetilde{\can}\circ\sigma_\Pi\circ \widetilde{\can}.
$$
This means that $\widetilde{\can}\circ\sigma_\Pi = \id_{P\ot C^+}$, for $\sigma_\Pi$ is a monomorphism and $\widetilde{\can}$ is an epimorphism. Therefore, in view of the fact 
that $\omega_\Pi$ satisfies Definition~\ref{2:def.con}(2)(i), we obtain, for all $c\in C$,
\begin{eqnarray*}
\widetilde{\can}(\omega_\Pi(c)) &=& \widetilde{\can}(1\sw 0
\omega_\Pi(c\eps(1\sw 1) - \eps(c) 1\sw 1)) \\
&=& \widetilde{\can}(\sigma_\Pi(1\sw 0\ot (c\eps(1\sw 1) - \eps(c) 1\sw 1))) \\
&=& 1\sw 0\ot c\eps(1\sw 1) - 1\sw 0 \ot \eps(c) 1\sw 1 = 1\ot c - \eps(c)1\sw 0\ot 1\sw 1.
\end{eqnarray*}
Thus $\omega_\Pi$ satisfies condition Definition~\ref{2:def.con}(2)(ii) as well.

To prove that $\omega_\Pi$ satisfies condition Definition~\ref{2:def.con}(2)(iii), first apply $\id\ot_B\id\ot\tau$ to Lemma~\ref{2:translation}(iv), multiply the second and third factors, and use 
Lemma~\ref{2:translation}(iii) to deduce that, for all $c\in C$,
\begin{equation}\label{2:trans.6}
c\su 1\ot_B1\ot_B c\su 2 = c\sw 1\su 1\ot_Bc\sw 1\su 2 c\sw 2\su 1\ot_B c\sw 2\su 2.
\end{equation}
This then facilitates the following calculation (in which $\psi^2$ is the shorthand for $(\id\ot\psi)\ci(\psi\ot\id)$, as before):
\begin{eqnarray*}
\psi^2(c\sw 1\ot \omega_\Pi(c\sw 2)) &=& \psi^2(c\sw 1\ot c\sw 2\su 1\Pi(\d c\sw 2\su 2)) \\
&=& c\sw 2\su 1_\alpha \psi^2(c\sw 1^\alpha\ot \Pi(\d c\sw 2\su 2)) \qquad \mbox{(left pentagon for $\psi$)}\\
&=& c\sw 1\su 1(c\sw 1\su 2 c\sw 2\su 1)\sw 0\psi^2((c\sw 1\su 2 c\sw 2\su 1)\sw 1 \ot \Pi(\d c\sw 2\su 2)) \\
&=& c\su 11\sw 0\psi^2(1\sw 1 \ot \Pi(\d c \su 2)) \qquad \mbox{(equation~(\ref{2:trans.6}))}\\
&=& c\su1\Delta_{\Omega^1P}(\Pi(\d c \su 2)) \qquad \mbox{(def.\ of coaction $\Delta_{\Omega^1P}$)}\\
&=& c\su 1(\Pi(\d c \su 2\sw 0)) \ot c \su 2\sw 1\qquad \mbox{(colinearity of $\Pi\circ\d$)}\\
&=& c\sw 1\su 1(\Pi(\d c\sw 1 \su 2)) \ot c \sw 2\qquad \mbox{(Lemma~\ref{2:translation}(iv))}\\
&=& \omega_\Pi(c\sw 1)\ot c\sw 2.
\end{eqnarray*}
The third equality follows from the definition of the canonical entwining map in equation (\ref{2:canent}). Thus $\omega_\Pi$ satisfies condition Definition~\ref{2:def.con}(2)(iii), and we conclude that $\omega_\Pi$ is a connection form. 

Finally, a simple calculation that uses the Translation Map Lemma~\ref{2:translation} reveals that the maps defined by equations (\ref{2:omegaPi}) and (\ref{2:Piomega}) are inverses of each other. This completes the proof.
\end{proof}


The connection and connection form are constructed in analogy with their classical
counterparts that are adapted to de Rham differential forms. In our setting of
the universal calculus, it appears more convenient to use the following formulation
of the concept of connection.
\begin{Def}\label{2:def.lift}
Let $B\inc P$ be a \cge. A {\em connection lifting} is a homomorphism
$\ell:C\ra P\ot P$ such that
\begin{rlist}
\item $1\sw0\ell(1\sw1)=1\ot1$,
\item $\widetilde{can} \circ \ell= 1 \otimes \id $,
\item  $(\id\ot\psi)\ci(\psi\ot\id)\ci(\id \otimes \ell)\ci\hD
=(\ell\ot\id)\ci\hD$. 
\end{rlist}
\end{Def}
It is straightforward to verify that the connection lifting and connection form
are related by the following simple formula:
\[\label{2:lomega}
\ell=\ho+1\ot1\eps.
\]
This constant shift of a connection form allows one to view a connection as
a lifting of the translation map~$\tau$. Indeed, one can equivalently write the
condition Definition~\ref{2:def.lift}(ii) as $\pi_B\ci\ell=\tau$, where $\pi_B:P\ot P\ra P\ot_BP$ is the canonical
surjection.
\begin{Rem}\label{2:tom}
In an $e$-coaugmented  coalgebra-Galois $C$-extension $B\inc P$,
the external differential commutes with the coaction,
$(\d \otimes \id_C)\circ \Delta_P=\Delta_{\Op} \circ \d$. Indeed,
 for all $p\in P$, we have
\begin{eqnarray*}
\d(p)\sw 0\otimes d(p)\sw 1 &=& 1\otimes p_\alpha \otimes e^\alpha - p\sw 0
\otimes 1_\alpha\otimes p\sw 1^\alpha \\
&=& 1\otimes p\sw 0\otimes p\sw 1 -
p\sw 0\otimes 1\otimes p\sw 1 = \d(p\sw 0)\otimes p\sw 1.
\end{eqnarray*}
Here we used the left triangle in the bow-tie diagram and the
fact that $P$ is an entwined module, i.e., 
\[
p_\alpha\otimes e^\alpha = 1\sw 0p_\alpha\otimes 1\sw 1^\alpha = (p1)\sw
0\otimes p\sw 1 = \Delta_P(p).
\]
Thus, in this case,  the universal differential calculus  
 $\Op$ is a {\em
right-covariant differential calculus} on $P$ (cf.\ \cite{w-sl89} for
the definition of a right-covariant calculus). Moreover, we can now define
a {\em covariant differential} $D=\d-\Pi\ci\d$ which is right $C$-colinear.
\end{Rem}

The definition of a connection Definition~\ref{2:def.con} tries to follow
the classical definition of a connection in principal bundles, albeit in a
dual language. Obviously, not all the classical properties of
connections can be recovered in this general algebraic setup. For
example, if $P$ is an algebra of functions on a total space $X$ of a
classical principal bundle,  $B$ is an algebra of functions on a base
manifold $M$, and $\Omega^1(B)$ is the classical space of $1$-forms on
$M$, then horizontal forms have several equivalent descriptions
$$
\Omega^1_{hor}(P) = P(\Omega^1(B))P = P(\Omega^1(B)) =(\Omega^1(B))P.
$$
Obviously, no such relationships exist in a general noncommutative
setting, even more so in the case of the universal differential
calculus. In a non-universal calculus case, and in particularly nice
examples (e.g., in the case of the quantum Hopf fibering with the
$3D$-calculus discussed in \cite{bm93}), the above equalities can be
obtained. In a general situation with the universal differential
calculus, in order to come closer to the classical geometric intuition,
one is forced to consider also a stronger version of the  notion of a
connection (cf.\ \cite{h-pm96}).

\subsubsection{Associated modules}\label{2:subs.ass}

In noncommutative differential geometry \cite{c-a94} vector bundles
are identified with finitely generated projective modules, via the
extrapolation of the classical Serre-Swann theorem which states that
the category of finite dimensional vector bundles over a compact Hausdorff
space $M$ is equivalent to the category of finitely generated projective
modules over the algebra of functions $C(M)$. The equivalence is given
by assigning to each vector bundle its the module of continuous sections.

In classical geometry, all vector bundles (and hence projective modules of
sections) arise as bundles associated to
principal bundles. More precisely, consider a principal bundle with a total
space $X$, the structure group $G$, and the base space $M=X/G$. Take
any finite-dimensional representation of $G$, i.e., a finite-dimensional
vector space $V$ with a right $G$-action. Then sections of an
associated vector bundle with a standard fibre $V$ are identified with
vector-valued maps $\sigma: X\ra V$ such that, for all $x\in X$ and $g\in
G$, $\sigma(x\cdot g) = \sigma(c)\cdot g$.  Given a vector bundle $E$ with a standard
fibre $k^n$, one can construct a $GL_n(k)$-principal bundle (the frame
bundle) such that $E$ is isomorphic to a bundle associated to this
principal bundle. Since coalgebra-Galois extensions are to be
interpreted as noncommutative principal bundles, it makes sense to study
associated modules, which are to be interpreted as vector bundles. This
construction can be performed for any coalgebra-Galois extension, but
we restrict our attention to those extensions which lead to bona fide
noncommutative vector bundles, i.e., finitely generated projective
modules. 

Recall
that elements of $\Hom^C(V,P)$ are  linear maps $f:V\ra P$ such that
$\Delta_P\circ f = (f\otimes \id)\circ \Delta_V$, where $\Delta_V
:V\ra V\otimes C$ is a coaction of $C$ an $V$. 
\begin{Def}\label{2:def.sec}
For a   coalgebra-Galois   $C$-extension $B\inc P$ and a right
$C$-comodule $V$, the {\em module of sections} of a bundle associated to
the $C$-extension $B\inc P$ with standard fibre $V$ is defined as
the space of right $C$-colinear maps $E:=\Hom^C(V,P)$.
\end{Def}

 Since $\Delta_P$ is by
definition a left $B$-linear map we immediately obtain
$E$ is a left $B$-module with the action $(b\cdot f)(v) = bf(v)$.
 Recall that for a finite-dimensional comodule $V$, the associated module
$E=\Hom^C(V,P)$ is isomorphic to $P \square_C V^*$ in $_B{\Mc}$.
Note that, if $P$ is equvariantly projective, then $E$ is projective.
Indeed, $s \otimes \id: P \Box_C V^* \ra B \otimes P \Box_C V^*$ defines
a splitting of the multiplication map. Furthermore, one can prove:

\begin{Lem}
Let $B \subset P$ be a symmetric (bijectivity of the canonical entwining assumed)
 Galois $C$-extension with $P$ faithfully
flat in $\M_B$ and $V$ a left $C$-comodule with $\dim_k V < \infty$.
Then the associated module $E=\Hom^C(V,P)$ is finitely generated as a left
$B$-module.
\end{Lem}

\subsubsection{Gauge transformations}\label{2:subs.gau}

In  classical geometry, gauge transformations arise from automorphisms of principal bundles. In parallel to this and following \cite{b-t99}, we consider:
\begin{Def}\label{2:def.auto}
Given a coalgebra-Galois $C$-extension $B\inc P$, a left $B$-linear, right $C$-colinear automorphism $F:P\to P$ such that $F(1) = 1$ is called a {\em gauge automorphism} of $B\inc P$. Gauge autmorphisms form a group with respect to the opposite composition (i.e., $FG = G\circ F$) which is denoted by $GA^C(B\inc P)$.
\end{Def}

The following theorem (cf.\ \cite[Theorem~2.4]{b-t99}) provides one with  an equivalent description of gauge automorphisms in terms of certain maps $C\to P$, which play the role of gauge transformations in the coalgebra-Galois setting.
\begin{Thm}\label{2:thm.gauge.aut}
Let $B\inc P$ be a coalgebra-Galois $C$-extension, and let $\psi$ denote the canonical entwining. Then
\begin{zlist}
\item Convolution invertible maps $f:C\to P$ such that 
\begin{blist}
\item $1\sw 0f(1\sw 1) = 1$,
\item $\psi\circ(\id\ot f)\circ\Delta = (f\otimes \id)\circ \Delta$,
\end{blist}
form a group with respect to the convolution product. This group is denoted by $GT^C(B\inc P)$ and called the {\em group of gauge transformations} of $B\inc P$.
\item The assignments
\[ \label{2:Ff}
f\mapsto F_f,\qquad F_f(p) = p\sw 0f(p\sw 1)
\]
\[ \label{2:fF}
F\mapsto f_F, \qquad f_F(c) = c\su 1F(c\su 2)
\]
define the mutually inverse isomorphisms of groups of gauge transformations and gauge automorphisms. 
\end{zlist}
\end{Thm}
\begin{proof}
(1) It is clear that if $f, g:C\to P$ satisfy (a), then so does their convolution product $f*g$. Condition (b) reads explicitly, for any $c\in C$,
\[ \label{2:cond.b.gauge}
\psi(c\sw 1\otimes f(c\sw 2)) = f(c\sw 1)\ot c\sw 2.
\]
An easy calculation that uses the left pentagon in the bow-tie diagram reveals that if $f$ and $g$ satisfy (\ref{2:cond.b.gauge}), then so does the convolution product $f*g$. Thus $GT^C(B\inc P)$ is a semigroup. 

 It is clear that the map $\eta: c\mapsto \eps(c)1_P$ satisfies condition (a). The left triangle in the bow-tie diagram ensures that this map also satisfies condition (b). Thus $\eta\in GT^C(B\inc P)$, and therefore $GT^C(B\inc P)$ is a monoid. 

Now take any $f\in GT^C(B\inc P)$. We need to show that its convolution inverse $f^{-1}$ satisfies conditions (a) and (b). Since $P$ is an entwined module and $f$ satisfies conditions (a) and (b), we can compute
$$
1\sw 0\ot 1\sw 1 = \Delta_P (1) = \Delta_P(1\sw 0f(1\sw 1)) = 1\sw 0\psi(1\sw 1\ot f(1\sw 2)) = 1\sw 0 f(1\sw 1)\ot 1\sw 2.
$$
Apply $\id\ot f^{-1}$ and multiply to conclude that $1\sw 0f^{-1}(1\sw 1) = 1$, as required. Finally, the facts that $f$ satisfies (b) and the commutativity of   the left pentagon in the bow-tie diagram, imply that, for all $c\in C$,
$$
c\sw 1\tens 1\tens c\sw 2 = c\sw 1\tens\psi(c\sw 2\tens f(c\sw  
3)f^{-1}(c\sw 4)) 
= c\sw 1 \tens f(c\sw 2)\psi(c\sw 3\tens f^{-1}(c\sw  
4)).  
$$
Applying $f^{-1}\ot\id\ot\id$ to the above equality and multiplying  
the first two factors one finds that $f^{-1}$ satisfies (b). Thus $GT^C(B\inc P)$ is a group as claimed.

(2) The canonical isomorphism $\can: P\ot_B P \cong P\ot C$ induces an isomorphism 
of spaces of left $P$-linear maps, 
${}_P{\rm Hom}(P\ot_B P,P) \cong {}_P{\rm Hom}(P\ot C,P)$. This, in turn, reduces 
to an isomorphism
$$
{}_B{\rm Hom}(P,P) \cong {\rm Hom}(C,P).
$$
One easily checks that the explicit form of this isomorphism is given by the maps in equations (\ref{2:Ff}) and (\ref{2:fF}), and that the opposite composition of automorphisms corresponds to the convolution product. The $C$-colinearity of an automorphism of $P$ induces condition (1)(b) for the corresponding map on the right  hand side. Similarly, normalisation leads to condition (1)(a). 
\end{proof}

Gauge transformations induce  transformations of connections in coalgebra-Galois
extensions. More precisely, the gauge group acts on the space of connections. This is of particular 
importance in the case of strong connections in principal extensions.

\subsection{Strong connections on principal extensions}\label{2:subs.str}

The original motivation for defining strong connections was to have a concept
of connection that for a cleft \hge\ would be expressible only in terms of the
Hopf algebra $H$ and
the coaction-invariant subalgebra~\cite{h-pm96}. 
Roughly speaking, in differential geometry this
 corresponds to the fact that horizontal subspaces are always isomorphic to
tangent spaces of the base manifold, so that one can assemble a connection form
on the total space out of a collection of locally defined forms on the base space.
 This turned out to be a notion that allowed one to construct a {\em covariant
derivative} on the associated modules \cite{hm99},
and thus link the Hopf-Galois theory of
quantum principal bundles with connections on projective modules.

Here we first study strong connections in the general setting of \cg\ extensions.
The definition of a general connection given before called for a replacement of
a diagonal coaction, and only finding it allowed a definition analogous to its earlier
Hopf-Galois version. Now, 
having a general connection at hand, one can phrase the strongness
 condition precisely as  in  Hopf-Galois theory~\cite{h-pm96}.
\begin{Def}[\cite{bm00}]\label{2:strongd}
Let $\Pi$ be a connection on a \cge\ $B\inc P$. We call it {\em strong} if
$(\id-\Pi\ci\d)(P)\inc (\hO^1B)P$.
\end{Def}

One can alternatively define a strong connection in the following way:
\begin{Def}[\cite{dgh01}]
Let $B\inc P$ be a \cge. We call a unital
left $B$-linear right $C$-colinear (with respect to $\id\ot\hD_P$)
splitting of the multiplication map $B\ot P\ra P$
a {\em strong-connection splitting}.
\end{Def}
\begin{Lem}[\cite{dgh01}]\label{2:stronglem}
Let $B\inc P$ be a \cge. Then the equation
\[
s=\Pi\ci\d+\id\ot1
\]
defines a one-to-one correspondence between strong connections and
strong-connection splittings.
\end{Lem}
\begin{proof}
Let
$s$ be such a splitting, and $\Pi^s(r\d p):=r(s(p)-p\ot 1)$. One can verify
that this formula gives a well-defined left $P$-linear endomorphism of $\hO^1P$.
Furthermore, by the left $B$-linearity of $s$, for any
$\sum_i\d b_i.p_i\in(\hO^1B)P$, we have:
\begin{equation}
\Pi^s(\mbox{$\sum_i$}\d b_i.p_i)
= \mbox{$\sum_i$}\Pi^s\llp\d (b_i.p_i)-b_i\d p_i\lrp
= \mbox{$\sum_i$}s(b_ip_i)- b_ip_i\ot 1-b_i(s(p_i)- p_i\ot 1)=0.
\end{equation}
Hence $P(\hO^1B)P\inc \ker\Pi^s$ by the left $P$-linearity of $\Pi^s$.
On the other hand, since $m\ci s=\id$ and $s(P)\inc B\ot P$, we have
$\pi_B(s(p))=1\ot_Bp$, where $\pi_B:P\ot P\ra P\ot_BP$ is the canonical
surjection. Consequently,
\begin{equation}
\pi_B(\Pi^s(p'\d p))
=\pi_B(p'(s(p)-p\ot 1))
=p'\pi_B(s(p))-rp\ot_B 1
=p'\ot_Bp-rp\ot_B 1
=\pi_B(p'\d p).
\end{equation}
Therefore, since $P(\hO^1B)P= \ker\pi_B$,
we obtain $\ker\Pi^s\inc P(\hO^1B)P$.
Thus $\ker\Pi^s= P(\hO^1B)P$. Next, take any $p\in P$. It follows from
$s(P)\inc B\ot P$ that
\begin{equation}
\d p-\Pi^s(\d p)=1\ot p-p\ot 1-s(p)+p\ot 1=1\ot p-s(p)\in B\ot P.
\end{equation}
Since also $m(1\ot p-s(p))=0$, we have
$\d p-\Pi^s(\d p)\in (\hO^1B)P\inc \ker\Pi^s$.
By the left $P$-linearity of $\Pi^s$, we can conclude now that
$\Pi^s\ci(\id-\Pi^s)=0$, i.e., $(\Pi^s)^2=\Pi^s$.
It remains to show that
$\hD_{P\ot P}\ci\Pi^s\ci\d=\llp(\Pi^s\ci\d)\ot\id\lrp\ci\Delta_P$.
The property $\psi(c\ot 1)=1\ot c$ implies that
\begin{equation}
\hD_{P\ot P}(p\ot 1)
= p\sw0\ot\psi(p\sw1\ot 1)
= p\sw0\ot 1\ot p\sw1.
\end{equation}
Therefore,
\begin{eqnarray}
\hD_{P\ot P}(\Pi^s(\d p))
\!\!\!\!\!\!&&
=\hD_{P\ot P}(s(p))-\hD_{P\ot P}(p\ot 1)
\nonumber\\ &&
=s(p\sw0)\ot p\sw1-p\sw0\ot 1\ot p\sw1
\nonumber\\ &&
=\llp(\Pi^s\ci\d)\ot\id\lrp(\Delta_P(p)),
\end{eqnarray}
by the colinearity of $s$. Consequently, $\Pi^s$ is a connection, as claimed.
\end{proof}

\begin{Rem}
Within the framework of Hopf-Galois theory the right coaction
$\id\ot\Delta_P : B\ot P\ra B\ot P\ot H$ and the restriction $\hD_{B\ot P}$
of the diagonal coaction $\hD_{P\ot P}$ (\ref{2:dsr}) coincide.
Therefore, one can use either
of them to define the colinearity of a splitting $s$ of the multiplication map
$B\ot P\ra P$. In the general setting of coalgebra-Galois extensions 
 \begin{equation} 
\hD_{P\ot P}(b\ot p) =(\id\ot\psi)(\Delta_P(b1)\ot p)
=(\id\ot\psi)(b1\sw0\ot 1\sw1\ot p) =b1\sw0\ot\psi(1\sw1\ot p). 
\end{equation} 
On
the other hand, if $B\inc P$ is $C$-Galois and $\psi$ is its
canonical entwining structure \cite[(2.5)]{bh99}, then, by
\cite[Theorem 2.7]{bh99}, $P$ is a $(P,C,\psi)$-module
\cite{b-t99}, so that we have $$\Delta_P(p'p)=p'\sw0\psi(p'\sw1\ot p).$$ 
In
particular, $\Delta_P(p)=1\sw0\psi(1\sw1\ot p)$. Hence 
\begin{equation}
(\id\ot\Delta_P)(b\ot p)=b\ot 1\sw0\psi(1\sw1\ot p). 
\end{equation}
 Hence we need to distinguish between
$\hD_{B\ot P}$ and $\id\ot\Delta_P$ in the $C$-Galois case.
\end{Rem}

The just discussed problem of the diagonal coaction in the general \cg\ setting
was already encountered in Remark~\ref{2:tom}, where it obstructs the definition
of a covariant differential associated to a general connection. For strong connections
this problem disappears. First, we can define:
\begin{Def}\label{2:D}
Let $B\inc P$ be a \cge. A {\em strong covariant differential} is a homomorphism
$D:P\ra(\hO^1B)P$ satifying

(1) $D(bp)=bD(p)+(\d b)p,\,\fa b\in B,\,p\in P$ (the Leibniz rule),

(2) $(\id\ot\hD_P)\ci D=(D\ot\id)\ci\hD_P$ (covariance).
\end{Def}
Taking advantage of Lemma~\ref{2:stronglem} and reasoning as in \cite{cq95,dgh01},
now we can prove:
\begin{Lem}\label{2:Ds}
Let $B\inc P$ be a \cge. Then the equation
\[
D=(\id-\Pi)\ci\d
\]
defines a one-to-one correspondence between strong connections and
strong covariant differentials.
\end{Lem}
\begin{proof}
 Given $s$ as in (2),
   define the corresponding $D_{s}:P\to (\Omega^{1}B)P$ via $p\mapsto
   1\otimes p -s(p)$. Conversely, given $D$ as in (3), define $s_{D}:
 P\to B\otimes P$, $p\mapsto 1\otimes p - D(p)$. This establishes the
equivalence between descriptions (2) and (3). The equivalence
 (3)$\Leftrightarrow$  (1) is established as follows.
   Given $D$, define $\Pi_{D}: p\d (q)\mapsto p\d (q) - pD(q)$, while given
    $\Pi$ define $D_{\Pi} = \d- \Pi\circ \d :P\to (\Omega^{1}B)P$.
\end{proof}

\begin{Rem}In the light of Lemma~\ref{2:D},  the same arguments as in \cite[pp.\ 314--315]{dgh01}  establish correspondence between strong connections  on coalgebra-Galois extensions and
Cuntz-Quillen connections on bimodules \cite[p.283]{cq95}. 
\end{Rem}

Let us now pass to {\em strong-connection forms}. One can define them simply as connection
forms corresponding to strong connections via Theorem~\ref{2:con.form}. 
However, it turns out 
that one needs to assume coaugmentation to define the strongness condition for
connection forms in a more intrinsic way.
\begin{Lem}\label{2:thm.str.inv}
Let  $B\inc P$ be an $e$-coaugmented \cge\  such that the canonical entwining map is injective. For a
connection $1$-form $\omega$, the following are equivalent:
\begin{blist}
\item $\omega$ is a strong connection one-form;
\item  $(\id \otimes \Delta_P)\circ \omega(c)=
  1 \otimes 1 \otimes c - \varepsilon(c)1 \otimes 1 \otimes e +
  \omega(c\sw 1) \otimes c\sw 2$,
\end{blist}
\end{Lem}
\begin{proof}
The injectivity of $\psi$ implies that $x\in B\ot P$ if and only if $\Delta_{P\ot P}(x) = (\id\otimes \Delta_P)(x)$. Indeed, it is clear that if $x\in B\ot P$, then $\Delta_{P\ot P}(x) = (\id\otimes \Delta_P)(x)$. Conversely, write $x=\sum_{i} r^i\ot p^i$. In view of the definition of the coaction $\Delta_{P\ot P}$ and the fact that $\Delta_P(p) = \psi(e\ot p)$, the condition   $\Delta_{P\ot P}(x) = (\id\otimes \Delta_P)(x)$ explicitly reads
$$
((\id\ot\psi)\circ (\Delta_P\ot \id))(\mbox{$\sum_i$} r^i\ot p^i) = (\id\ot\psi)(\mbox{$\sum_i$} r^i\ot e\ot p^i).
$$
Since $\psi$ is injective, this implies that $\sum_{i} \Delta_P(r^i)\ot p^i = \sum_{i} r^i\ot e\ot p^i$, i.e., for each of the $r^i$s, we have $\Delta_P(r^i) = r^i\ot e$, so that all of them
are elements of $B$, as claimed.

Let $\omega$ be a  connection form and let $D$ be the corresponding covariant differential. By definition, $\omega$ is strong if and only if, for all $p\in P$, $D(p)\in (\Omega^1B)P$. In view of the discussion above, this is equivalent to the condition
\begin{equation}\label{2:str.inj}
 (\id\otimes \Delta_P)\circ D(p) =
\Delta_{\Omega^1P}\circ D(u), \qquad
\forall p\in P.
\end{equation}
Using the explicit
definition of $\d$ and $D$, Theorem~\ref{2:con.form}(iii),
 as well as the fact that $\Omega^1P\in
\Mc_{\Omega P}^C(\psi^\otimes)$,  one finds that
(\ref{2:str.inj}) implies that
\[
(\id\otimes\Delta_P)(p\sw 0\omega(p\sw 1)) = p\sw 0\otimes 1\otimes p\sw
1-p\otimes 1\otimes e +p\sw 0\omega(p\sw 1)\otimes p\sw 2.\]
Next, for all $c$, let $\tau(c) = c^\su 1\otimes_B c\su 2$ be the
translation map. 
 Using the Translation Map Lemma~\ref{2:translation}, we compute
\begin{eqnarray*}
(\id\otimes\Delta_P)\circ\omega(c) & = & (\id\otimes\Delta_P)
(c\su 1c\su 2\sw
0\omega(c\su 2\sw 1)) = c\su 1(\id\otimes\Delta_P)(c\su 2\sw
0\omega(c\su 2\sw 1))\\
& = & c\su 1c\su 2\sw 0\otimes 1\otimes c\su 2\sw
1-c\su 1c\su 2\otimes 1\otimes e \\
&&+c\su 1c\su 2\sw 0\omega(c\su 2\sw 1)\otimes
 c\su 2\sw 2 \\
& = & 1\otimes 1\otimes c-\eps(c)1\otimes 1\otimes e+\omega(c\sw 1)
\otimes c\sw 2 .
\end{eqnarray*}
Thus if $\omega$ is a strong-connection form, then the assertion (b)  holds. 
Conversely, an easy calculation reveals
that the assertion (b)  implies equation (\ref{2:str.inj}), i.e.,
the connection is strong, as required.
\end{proof}

The structure of coalgebra-Galois extensions is even richer when the
canonical entwining map is bijective. In the case of a Hopf-Galois $H$-extension $B\inc P$, 
the canonical entwining map has the form $\psi: h\ot p\mapsto p\sw 0\ot hp\sw1$.
Hence 
it is bijective, provided the antipode is bijective. The inverse of $\psi$ then reads, 
$\psi^{-1}: p\ot h\to hS^{-1}(p\sw 1)\ot p\sw 0$ (cf.\ Example~\ref{2:ex.can.Hopf}).
Thus, whenever the bijectivity of $\psi$ is assumed one should keep in
mind that this corresponds to the bijectivity of the antipode in the
case of a Hopf-Galois extension. This heuristic understanding can be
extended even further, once one realises that also in more general case
of a Doi-Koppinen datum $(H,C,A)$ in Example~\ref{2:ex.Doi.Koppinen}, the
corresponding entwining map $\psi$ is bijective provided $H$ is a Hopf
algebra with a bijective antipode (in fact, suffices it assume that $H$
is a bialgebra with a twisted antipode, i.e., $H^{op}$ is a Hopf
algebra). 

Much as for strong connection forms, we can simply define {\em strong-connection
liftings} as connection liftings corresponding to strong connections. This time,
to obtain a more intrinsic characterisation of the strongness condition, we need
additionally to assume that the canonical entwining map $\psi$ 
of our $e$-coaugmented \cge\ is bijective. In this case,
we can first define a left coaction 
\[\label{2:leftco}
{}_P\Delta:P\lra C\ot P,\;\;\; {}_P\Delta(p):=\psi^{-1}(p\ot e).
\]
\begin{Rem}\label{2:symmetric} That ${}_P\Delta$ is a coaction
 can be verified directly, but it can also be seen as follows.
In Definition~\ref{2:ES} we defined a right-right entwining structure,
in the sense that the structures defined in Lemma~\ref{2:lemma.entw.gro}
are a  right comodule and a right  module structures. One can easily define
a
left-left entwining structure,  by flipping all tensor
products in the bow-tie diagram (a left-left entwining  would then be
a map $A\otimes C\ra C\otimes A$). One then immediately has the
left-handed version of Lemma~\ref{2:lemma.entw.gro}.
Now, if $\psi$ is a right-right
entwining map, then its inverse is a left-left entwining map. Thus the
left-handed version of Lemma~\ref{2:lemma.entw.gro} implies that
${}_P\Delta$ given by equation (\ref{2:leftco})  is a left coaction. 
\end{Rem}
 
In addition to Lemma~\ref{2:thm.str.inv}, we  obtain the following characterisation of strong connection forms.
\begin{Lem} \label{2:lemma.scfs}
Let $B\inc P$ be an $e$-coaugmented \cge\ whose canonical entwining map $\psi$
is bijective. For a connection form $\omega$ the following are equivalent:
\begin{blist}
\item $\omega$ is a strong connection one-form;
\item $({}_P\Delta \otimes \id) \circ \omega(c)=
  c \otimes 1 \otimes 1  - e \otimes 1 \otimes \varepsilon(c)1 +
  c\sw 1 \otimes \omega(c\sw 2)$.
\end{blist}
\end{Lem}
\begin{proof}
Obviously, since $\psi$ is bijective, Lemma~\ref{2:thm.str.inv} holds. 
Hence it suffices to apply  the bijection $\psi^{-2} = (\psi^{-1}\ot \id)\circ(\id\ot \psi^{-1})$ to assertion (b) in Lemma~\ref{2:thm.str.inv}.
\end{proof}

Now a strong-connection lifting can be characterised as follows:
\begin{Lem}[\cite{bhms}]\label{2:scl}
Let $B\inc P$ be an $e$-coaugmented \cge\ whose canonical entwining map $\psi$
is bijective, and let $\ell$ be a homomorphism from $C$ to $P\ot P$.
It is a strong-connection lifting {\em if and only if} it satisfies the following conditions:

(1) $\ell(e)=1\ot1$ (unitality),

(2) $\pi_B\circ \ell= \tau$, $\pi_B:P\ot P\ra P\ot_B P$, (lifting property)

(3) $({}_P\Delta\ot\id)\ci\ell=(\id\ot\ell)\ci\hD$ and  $(\id\ot\hD_P)\ci\ell=(\ell\ot\id)\ci\hD$
(bicolinearity).
\end{Lem}

Cougmented coalgebra-Galois extensions whose canonical entwining map is bijective
are very symmetric. Following the usual convention of differential geometry, where
one considers right rather than left bundles, we formulated the right-sided version
of extensions. However, much as in differential geometry, any
 $e$-coaugmented \cge\ $B\inc P$ enjoying the existence of $\psi^{-1}$ can
be equivalently formulated as a {\em left} \cge.
To begin with, there already exists a left $C$-coaction on $P$ given by the formula
(\ref{2:leftco}). Thus we can define the left coaction invariants:
\[\label{2:leftcoin}
^{coC}\!P:=\{b\in p\;|\;{}_P\Delta(pb)={}_P\Delta(p)b,\,\fa p\in P\}.
\]
Next, we need to check that the left and right coaction invariants
 coincide: 
\begin{Lem}[\cite{bm00}]
Let $B\inc P$ be an $e$-coaugmented \cge\ whose canonical entwining map $\psi$
is bijective. Then $^{coC}\!P=B$.
\end{Lem}
\begin{proof}
Since $b\in B$ if and only if $\psi(e\ot b) = \Delta_P(b) = b\ot e$ (cf.\ Proposition~\ref{2:prop.ecoin=coin}), applying $\psi^{-1}$, we immediately conclude that $b\in B$ if and if
$$
b \in {}^{coC}\!P_e \equiv \{p\in P\; |\; {}_P\Delta(p) = p\ot e\}.
$$
The left handed version of Lemma~\ref{2:coinvariants1} implies that $^{coC}\!P\inc ^{coC}\!P_e$. Now, since $P$ is a right-right $(P,C,\psi)$-entwined module by Theorem~\ref{2:canonical}, one easily checks that $P$ is also a left-left $(P,C,\psi^{-1})$-entwined module with coaction ${}_P\Delta$, i.e., $P\in {}^C_P\Mc(\psi^{-1})$ (cf.\ Remark~\ref{2:symmetric}). Thus, if $b\in ^{coC}\!P_e$, we can compute for any $p\in P$,
$$
{}_P\Delta(pb) = \psi^{-1}(p\ot b\sw{-1})\ot b\sw 0 = \psi^{-1}(p\ot e)b = {}_P\Delta(p)b.
$$
Therefore, $b\in ^{coC}\!P$, so that $^{coC}\!P = ^{coC}\!P_e$. This completes the proof.
\end{proof}

The right coaction is left $B$-linear, whereas the left coaction is right $B$-linear.
(In the Hopf-Galois case, they are both $B$-bimodule homomorphisms.) One should also
bear in mind that, in general, even under the assumption of commutativity, $P$ is {\em 
not} a bicomodule with respect to $\dl$ and $\dr$.

Now, the {\em left} canonical map
can be defined as follows:
\[\label{2:leftcan}
\can_L: P \otimes_B P \lra C \otimes P,\;\;\;\can_L(x \otimes y):=\dl(x)y.
\]
It is straightforward to verify that the left and right canonical maps are related
by the commutative diagram
\newline
\[
\xymatrix{
    & P \otimes_B P \ar[ld]_{\can_L} \ar[rd]^{\can_P} &      \\
C \otimes P  \ar[rr]^{\psi} &   & P \otimes C\,.
}
\]

\noindent
This implies $\tau(c)= \can_P^{-1}(1 \otimes c)= can_L^{-1}(c \otimes 1)=
c^{[1]} \otimes_B c^{[2]}$. (Here we used a more symmetric convention
$\can_P=\can$.) Therefore, even though there are left and right canonical
maps, the left and right translation maps coincide. Consequently, the lifting formulation
of a strong connection is also independent of the choice of left or right-sided formulation.

So far we restricted our attention to formulating the concept of a strong connections
without asking when such a connection exists. It turns out that the existance of a strong connection
both forces and is guaranteed by the equivariant projectivity. More precisely, we have:
\begin{Lem}[\cite{bhms}]\label{bhmslem}
An $e$-coaugmented \cge\ $B\inc P$ admits a strong connection {\em if and only if}
$P$ is $C$-equivariantly projective as a left $B$-module.  
\end{Lem}
Therefore, to have available all 5 formulations of a strong connection and to ensure their existence,
we need to demand that our \cg\ extension is equivariantly projective, coaugmented and
with bijective canonical entwining. Thus we arrive at the concept of a principal extension
from postulating rich strong-connection theory.
The following theorem summarises and completes this section on strong-connection
theory. It is a generalisation of Theorem~2.3 and Theorem~4.1 in~\cite{dgh01}.
\begin{Thm}\label{2:claim2}{\sc (Strong-Connection Theorem)}
Let $B\inc P$ be a principal $C$-extension. For any $c\in C$, write the translation map $\tau(c) = c\su1\ot_B c\su 2$. 
The following maps
\[\label{2:Dseq}
s\mapsto D_s, \qquad D_s(p) = 1\ot p -s(p), \quad \forall p\in P,
\]\[\label{2:PiD}
D\mapsto \Pi^D, \qquad \Pi^D(p\d p') = p\d p' - pD(p'), \quad \forall p,p'\in P,
\]\[\label{2:omegaPi-pr}
\Pi \mapsto \omega_\Pi, \qquad \omega_\Pi (c) = c\su 1\Pi(\d c\su 2), \quad \forall c\in C,
\]\[\label{2:lomega-pr}
\omega\mapsto \ell^\omega, \qquad \ell^\omega(c) = \omega(c) +\eps(c)1\ot 1, \quad \forall c\in C,
\]\[\label{2:sl}
\ell\mapsto s^\ell, \qquad s^\ell(p) = p\sw 0\ell(p\sw 1), \quad \forall p\in P,
\]
give bijective correspondences between sets of strong-connection splittings $s$, strong covariant differentials $D$, strong connections $\Pi$, strong-connection forms $\omega$, and strong-connection liftings $\ell$. The explicit form of inverses of these maps is obtained by cyclic composition, e.g., the inverse of map (\ref{2:PiD}) is obtained by composing maps (\ref{2:omegaPi-pr}) with (\ref{2:lomega-pr}), with (\ref{2:sl}) with (\ref{2:Dseq}), etc.
\end{Thm}
\begin{proof}
The equivalent descriptions of strong connections given by maps (\ref{2:Dseq}) and (\ref{2:PiD}) 
are contained in Lemma~\ref{2:stronglem} and Lemma~\ref{2:Ds}, while the map 
(\ref{2:omegaPi-pr}) is described in Theorem~\ref{2:con.form}. The  map (\ref{2:lomega-pr}) 
originates from equation (\ref{2:lomega}). That the connection splitting corresponding to a 
strong connection  is a strong-connection splitting follows from Lemma~\ref{2:scl}. Thus 
we obtain maps between claimed spaces  as required. One easily checks that cyclic 
compositions provide inverses, as described.
\end{proof}

Since a gauge automorphism 
$F$ is unital, i.e., $F(1)=1$, its left $B$-linearity implies that $F(b) = b$ for all $b\in B$. 
Furthermore, $F$ is $C$-colinear, so that given a strong-connection 
splitting $s: P\to B\ot P$, 
the map $s^F = (\id\ot F^{-1})\circ s\circ F: P\to B\ot P$ is again a strong-connection 
splitting. 
It is clear from the form of $s^F$ that the assignment $F\mapsto s^F$ defines a 
left action of 
the group of gauge automorphisms on strong-connection splittings (remember 
that we use conventions in which the product of gauge automorphisms is given 
by the opposite composition). In view of the description 
of gauge automorphisms in terms of gauge transformations in Theorem~\ref{2:thm.gauge.aut} 
as well as various descriptions of strong connections in Theorem~\ref{2:claim2}, we are led 
to the following.
\begin{Thm}\label{2:thm.gauge.conn}
The group of gauge transformations of  a principal $C$-extension $B\inc P$ 
acts on the 
spaces of strong connection splittings,  covariant differentials, connections, 
connection 
forms and connection liftings in the following ways, for all $f\in GT^C(B\inc P)$, $p,r\in P$ 
and $c\in C$:
\begin{zlist}
\item Strong-connection splittings $s:P\to B\ot P$: 
$$(f\vt s) (p):= s \llp p\sw0 f(p\sw1 )\lrp f^{-1} (p\sw2 );$$
\item Strong covariant differentials $D:P\to \Omega^1BP$: $$(f\vt D) (p):= D \llp p\sw0 f(p\sw1 )\lrp f^{-1} (p\sw2 );$$
\item Strong connections $\Pi:\Omega^1P\to \Omega^1P$: $$(f\vt \Pi) (r\d p):= r \Pi \llp\d (p\sw0 f(p\sw1 ))\lrp f^{-1} (p\sw2 )
+ r p\sw0 f(p\sw1 )\d f^{-1} (p\sw2 ) ;$$
\item Strong connection forms $\ho:C\to \Omega^1P$: $$(f\vt \ho ) (c):= f(c\sw1 )\ho (c\sw2 ) f^{-1} (c\sw3 )
+ f(c\sw1 )\d f^{-1} (c\sw2 ) ;$$
\item Strong connection liftings $\ell:C\to P\ot P$: $$(f\vt \ell ) (c):= f(c\sw1 )\ell (c\sw2 ) f^{-1} (c\sw3 ).$$
\end{zlist}
All these actions are compatible with the isomorphisms described in Theorem~\ref{2:claim2}, i.e., the maps given by (\ref{2:Dseq})--(\ref{2:sl}) are left $GT^C(B\inc P)$-module maps.
\end{Thm}
\begin{proof}
In view of Theorem~\ref{2:thm.gauge.aut} and the discussion preceding the theorem, to show 
that item (1) describes a left action of $GT^C(B\inc P)$ on the space of strong-connection splittings, it
suffices  to show that $f\vt s = s^{F_f} := (\id\ot F_f^{-1})\circ s\circ F_f$, where $F_f$ is given by 
equation (\ref{2:Ff}). Since $s$ is right $C$-colinear, and, for all $p\in P$,  $F_f(p) = p\sw 0f(p\sw 1)$ 
and $F_f^{-1}(p) = p\sw 0f^{-1}(p\sw 1)$, one immediately finds that $ ((\id\ot F_f^{-1})\circ s)(p) = 
s(p\sw 0)f^{-1}(p\sw1)$. Furthermore, $P$ is an entwined module, so that
\begin{eqnarray*}
s^{F_f}(p) &=& s((p\sw 0f(p\sw 1))\sw 0)f^{-1}((p\sw 0f(p\sw 1))\sw 1)\\
&=& s(p\sw 0f(p\sw 2)_\alpha)f^{-1}(p\sw 1^\alpha)\\
&=&  s(p\sw 0f(p\sw 1))f^{-1}(p\sw 2).
\end{eqnarray*}
Note that the final equality follows from the fact that $f$ is a gauge transformation, i.e., it satisfies
 equation (\ref{2:cond.b.gauge}). The formulae for the action in the case of other descriptions 
of a strong connection (items (2)--(5)) are obtained by applying maps described in 
Theorem~\ref{2:claim2}. In particular, this implies that descriptions of actions of the gauge 
group are compatible with these maps. The reader can directly check these formulae, 
noting that the element $1\ot 1\in P\ot P$ is fixed under the gauge transformations.
\end{proof}

\subsubsection{Covariant derivatives on associated modules}

\begin{Thm}
Modules associated to equvariantly projective symmetric (bijectivity of the
canonical entwining assumed) coalgebra-Galois $C$-extensions
via finite-dimensional corepresentations  are finitely generated projective.
\end{Thm}

Once vector bundles are identified with projective modules one can study
connections in such modules. In general \cite{c-a94}, for an algebra $B$
and a left
$B$-module $E$, a {\em connection} is defined as a linear map $\nabla:
E\ra \Ob\otimes_B E$, which satisfies the Leibniz rule in the form
$$
\nabla(b\cdot f) = \d(b)\otimes_B f + b\nabla(f),
$$
for all $b\in B$ and $f\in E$. The theory of connections is of
particular interest, and indeed, meaningful, in the case of projective
modules, since  
 a module admits a connection if and only if it is a projective
module ~\cite{cq95}. The connection constructed directly from an
idempotent is known as the {\em Grassmann}  connection. Note that
there is no need for a projective module to be  finitely generated in order
to have a connection.

Given a strong connection $1$-form $\omega$ in a symmetric
coalgebra-Galois $e$-coaugmented $C$-extension $B\inc P$,
the corresponding covariant differential induces a map
 on the
associated module of sections $\Hom^C(V,P)$:
$$\nabla: \Hom^C(V,P) \lra \Hom^C(V,(\Ob)P) \ , \  f \longmapsto \nabla f,$$
where $\nabla f(v):=\d f(v)-f(v\sw 0) \omega(v\sw 1)$.
\begin{Prop} \label{2:prop.nabla}
If $V$ is finite dimensional then $\nabla$ is a connection on $\Hom^C(V,P)$.
\end{Prop}
\begin{proof}
The first crucial observation here is that if $V$ is finite dimensional,
then $\Hom^C(V, (\Ob)P)$ $\cong \Ob \otimes_B \Hom^C(V,P)$. 
Indeed, if a right $C$-comodule
$V$ is finite dimensional, then
the dual space $V^*$ is a left $C$-comodule with the coaction given
explicitly, for all $v^*\in V^*$,
${}^{V^*\!}\!\varrho (v^*) = \sum_{i=1}^n v^*(e_i\sw 0)e_i\sw 1\otimes e^i$,
where
$\{e_i\in V\}_{i=1,\ldots n}$ and $\{e^i\in V^*\}_{i=1,\ldots n}$ are dual to each other
bases of $V$ and $V^*$, respectively. For
any right $C$-comodule $W$, we have the canonical identification
$$
\Hom^C(V, W) \cong W\square_C V^*.
$$
Here 
$$
W\square_C V^*:= \left\{ \sum _i w_i\otimes v^*_i\in W\otimes V^*\; |\;
\sum_i \Delta_W(w_i)\otimes v^*_i = \sum_i w_i\otimes
{}^{V^*\!}\varrho(v^*_i)\right\}
$$
is the cotensor product of a right and a left $C$-comodule. We then have
the following chain of identifications
\begin{eqnarray*}
\Hom^C(V, (\Ob)P) &= &\Hom^C(V, \Ob\otimes_BP) \cong
(\Ob\otimes_BP)\square_C V^*\\
&= & \Ob\otimes_B(P\square_C V^* )\cong
\Ob \otimes_B \Hom^C(V,P).
\end{eqnarray*}
The redistribution of brackets in the penultimate  equality is possible
because $\Ob$ is a flat right $B$-module.  Thus the map $\nabla$ can be
viewed as a map $\nabla: \Hom^C(V,P)\ra \Ob\otimes_B \Hom^C(V,P)$, and has a right range
for a connection. Hence only the Leibniz rule needs to be verified. For
all $b\in B$ and $f\in \Hom^C(V,P)$, we have
\begin{eqnarray*}
\nabla(b\cdot f)(v) &=& \d (bf(v)) - bf(v\sw 0)\omega(v\sw 1) =\d(b) f(v) +
b(\d f(v)-f(v\sw 0) \omega(v\sw 1))\\
& =& \d(b) f(v) + b\nabla(f)(v),
\end{eqnarray*}
as required.
\end{proof}

\begin{Rem}\label{2:pm}
\note{The gauge transformations on $H$-Galois extension $B\inc P$ are in on-to-one
correspondence with the gauge automorphisms
understood as unital left $B$-linear right $H$-colinear automorphisms of $P$
\cite[Proposition~5.2]{b-t96}.
If $f : H\ra P$ is a gauge  transformation,
then $F:P\ra P$, $F(p):= p\sw0 f(p\sw1 )$ is a  gauge automorphism.
Analogously, for $\alpha\in \hO^1\! P$, we put
$F(\alpha ) := \llp (\id\ot m)\ci (\id\ot\id\ot f )\ci \dsr \lrp (\alpha )$.
(The other way round we have $f(h) = h\o F(h\t )$.)}
Due to the right $C$-colinearity of the covariant differential $D$,
we can re-write point (2) of the above theorem in terms of the gauge automorphisms $F$
to obtain the formula
$(F\vt D) (p) = F^{-1} (D F(p))$. 
This formula coincides with the usual formula for the action of
gauge transformations on  projective-module connections, 
cf.\ \cite[p.554]{c-a94}. Note, however, that since we use the opposite composition as a group operation in $GA^C(B\inc P)$, we have a left rather than right action here.
\end{Rem}

\subsubsection{Strong connections on pullback constructions}

Let ${}_C\mathbf{A}_C$ be the category of unital algebras equipped
with left and right (not necessarily commuting) coactions ${}_A\hD$
and $\hD_A$ of an $e$-coaugmented coalgebra $C$ such that
${}_A\hD(1)=e\otimes 1$ and $\hD_A(1)=1\otimes e$. Morphisms in
this category are bicolinear algebra homomorphisms. Since we work
over a field, this category is evidently closed under any pullback
\[
\xymatrix{
& P  \ar[rd]^{\mathrm{pr}_2} \ar[ld]_{\mathrm{pr}_1}   & \\
P_1\ar[rd]_{\pi_1} && P_2 \ar[ld]^{\pi_2} \\
 & P_{12} \,. &  
}
\]

The aim of this section is to show that the subcategory of principal extensions is closed under one-surjective pullbacks. Here the right coaction
is the coaction defining a principal extension and the left coaction is the
one defined by the inverse of the canonical entwining. The following
theorem generalises the two-surjective pullback Hopf-Galois result of
\cite{hkmz}:
\begin{Thm}[\cite{hw}]                                                \label{prop-pe}
Let $C$ be an $e$-coaugmented coalgebra and let $P$ be the pullback of
$\pi_1: P_1 \rightarrow P_{12}$ and  $\pi_2:P_2\rightarrow P_{12}$ in
the category ${}_C\mathbf{A}_C$. If $\pi_1$ or $\pi_2$ is surjective
and both $P_1$ and $P_2$ are principal $C$-extensions, then $P$ is
a principal $C$-extension.
\end{Thm}
\begin{proof}
Without the loss of generality, let us assume that $\pi_2$ is surjective.
First step is to prove that any surjective morphism in ${}_C\mathbf{A}_C$
whose domain is a principal extension can be split by a left colinear map and by a right colinear map (not necessarily by a bicolinear map). This can be proved much the same way as in the
Hopf-Galois case \cite{hkmz}. 

Let $\alpha^2_L$ and $\alpha^2_R$ be a left colinear splitting and a right 
colinear splitting of $\pi_2$, respectively. Also, let $\alpha^1_R$ be
a right colinear splitting of $\pi_1$ viewed as a map onto $\pi_1(P_1)$.
\[
\xymatrix{
P_1\ar[rrdd]^{\pi_1}\ar@<.6ex>[dd] &&P  \ar[rr]^{\mathrm{pr}_2} \ar[ll]_{\mathrm{pr}_1} && P_2 \ar@{-}[ld] \\
&&& {\mbox{\footnotesize$\pi_2$}} \ar@{->>}[ld] &\\
 \pi_1(P) \ar@<.6ex>[uu]^{\alpha^1_R}\ar@^{(->}[rr] && P_{12} \ar@<1.5ex>[uurr]^{\alpha^2_L} \ar@<-1.5ex>[uurr]_{\alpha^2_R}\,. &&  
}
\]
Since $P_1$ and $P_2$ are principal, they admit strong-connection
liftings $\ell_1$ and $\ell_2$, respectively. For brevity, let us introduce the notation
\[
\alpha^{12}_L:=\alpha^2_L\circ\pi_1,
\quad
\alpha^{12}_R:=\alpha^2_R\circ\pi_1,
\quad
\alpha^{21}_R:=\alpha^1_R\circ\pi_2|_{\pi_2^{-1}(\pi_1(P_1))}\,,
\quad
L:=m_{P_2}\circ(\alpha^{12}_L\otimes\alpha^{12}_R)\circ\ell_1\,,
\]
where $m_{P_2}$ is the multiplication of $P_2$. In the light of Lemma~\ref{bhmslem}, the proof boils down to verifying that the following formula defines a strong-connection lifting on $P$:
\begin{align}
\ell\;:=\;&\llp(\id+\alpha^{12}_L)\otimes(\id+\alpha^{12}_R)\lrp\circ\ell_1\\
&+\;(\mathrm{pr}_2\circ\he-L)*\Llp\llp\id\otimes(\id+\alpha^{21}_R)\lrp\circ
\llp\ell_2-\ell_2*L+(\alpha^{12}_L\otimes\alpha^{12}_R)\circ\ell_1\lrp
\Lrp.\nonumber
\end{align}
Here $\mathrm{pr}_2$ is the canonical pullback map on the second
component and $*$, as usual, stands for the convolution product.
\end{proof}

\subsubsection{Strong connections on extensions by coseparable coalgebras}\label{2:sec.cosep} 

In view of Theorem~\ref{2:thm.main.princ}, an $e$-coaugmented bijectively
entwined extension $B\subseteq P$ by a coseparable coalgebra $C$ is principal, provided the lifted canonical map is surjective. Following \cite{bb08}, we construct now  an explicit form of  a connection lifting in this case. 

Assume that $C$ is a coseparable coalgebra with a cointegral $\delta$.
Take  an entwining structure $(P,C,\psi)$ such that the map $\psi$ is bijective. Suppose that $e\in C$ is a group-like element and view $P$ as a right $C$-comodule with the coaction $\Delta_P:P\ra P\ot C$, $p\mapsto \psi(e\ot p)$, and as a left $C$-comodule with coaction ${}_P\Delta: P\ra C\ot P$, $p\mapsto \psi^{-1}(p\ot e)$. Let $\widetilde\sigma$ be a $k$-linear section of the lifted canonical map
$$
\widetilde{\can}: P\ot P{\lra} P\ot  C,\quad p\ot q\longmapsto p\Delta_P(q).
$$
Since $\widetilde{\can} (1\ot 1) = 1\ot e$,  the linear map 
$\sigma:=\widetilde\sigma(1\otimes\cdot)$ can always be normalised (so that $\sigma(e) = 1\ot 1$) by making the linear change
$$
\sigma \longmapsto \sigma +1\ot 1\eps - \sigma(e)\eps.
$$
We thus choose $\sigma$ that already is normalised in this way. By Theorem~\ref{2:thm.main.princ}, $B\subseteq P$ is a principal $C$-extension. 
Define 
\begin{equation}\label{2:gamma.alpha}
\gamma = (\delta\ot \id_P)\circ (\id_C\ot {}_P\Delta), \qquad \alpha = (\id_P\ot \delta)\circ (\Delta_P\ot \id_C),
\end{equation}
and
\begin{equation}\label{2:ell.cosep}
\ell =  (\gamma\ot \alpha)\circ (\id_C\ot\sigma\ot \id_C)\circ (\Delta\ot \id_C)\circ \Delta .
\end{equation}
\begin{Thm}[\cite{bb08}] The map $\ell$ given by \eqref{2:ell.cosep}
is a strong-connection lifting. 
\end{Thm}
\begin{proof}
 Using \eqref{2:coint} one easily checks that the map $\gamma$ is left $C$-colinear, where $C\ot P$ as understood as a left $C$-comodule via $\Delta\ot \id$, and $\alpha$ is right $C$-colinear, where $P\ot C$ is a right $C$-comodule via $\id\ot \Delta$. By the colinearity of $\gamma$ and $\alpha$, the map $\ell$ is $C$-bicolinear. 

To prove that $\ell$ has a lifting property, we start with the following simple calculation, for all $p,q \in P$,
$$
\psi^{-1}(p\Delta_P(q)) = 
 \psi^{-1}\left(p\psi(e\ot q)\right)
= \psi^{-1}(p\ot e)q ={}_P\Delta( p)q.
$$
Here the first and last equalities follow from the definitions of the right and left $C$-coactions on $P$, and the second equality follows by the fact that $\psi^{-1}$ is the inverse of the entwining map $\psi$. Thus we obtain the equality
\begin{equation}
\psi^{-1}(pq\sw 0\ot q\sw 1)\ot q\sw 2 = {}_P\Delta(p)\Delta_P(q). \label{2:key}
\end{equation}
For any $c\in C$,  write explicitly
$
c\sco 1 \ot c\sco 2 := \sigma(c),
$
so that 
$
c\sco 1c\sco 2\sw 0\ot c\sco 2\sw 1 = 1\ot c.
$
This leads to the equality
$$
c\sw 1\ot c\sw 2\sco 1c\sw 2\sco 2\sw 0  \ot c\sw 2\sco 2\sw 1\ot c\sw 3 = c\sw 1\ot 1\ot c\sw 2 \ot c\sw 3.
$$
Apply $(\id_C\ot \psi^{-1}\ot \id_C\ot \Delta)\circ (\id_C\ot \id_P\ot\Delta\ot \id_C)$ and then use \eqref{2:key} on the left hand side, and the unitality of the entwining map (the left triangle in the bow-tie diagram)  on the right hand side, to obtain
$$
c\sw 1\ot {}_P\Delta(c\sw 2\sco 1)\Delta_P(c\sw 2\sco 2) 
 \ot   c\sw 3\ot c\sw 4
 = c\sw 1\ot c\sw 2\ot 1\ot c\sw 3 \ot c\sw 4\ot c\sw 5.
$$
Now apply $\delta\ot \id_P\ot\delta\ot \id_C$ and use the definitions of maps $\gamma$ and $\alpha$ in terms of $\delta$ on the left hand side, and the properties of the cointegral \eqref{2:coint} on the right, to conclude that
$$
\gamma(c\sw 1 \ot c\sw 2\sco 1)\alpha(c\sw 2\sco 2\ot c\sw 3)\ot c\sw 4 = 1\ot c.
$$
By the right $C$-colinearity of $\alpha$, this implies that $\widetilde{\can} \circ \ell = 1\ot \id_C$. Hence, as $\can^{-1} \circ \widetilde{\can} = \pi_B$ is the standard projection $P\ot P\ra P\ot_B P$, we conclude that  $\pi_B\circ \ell= \tau$,  as required.

Finally, the definitions of left and right $C$-coactions on $P$ an \eqref{2:coint} imply that $\alpha(1\ot e) =1$ and $\gamma(e\ot 1) = 1$. These equalities together with the chosen normalisation for $\sigma$ yield $\ell(e)= 1\ot 1$.

Therefore, the map $\ell$ defined in \eqref{2:ell.cosep} satisfies all the properties of Lemma~\ref{2:scl}, i.e.\ it is a strong connection lifting, as stated.
\end{proof}

\subsubsection{Strong connections on homogeneous Galois extensions}

As a further illustration of the theory of strong connections in symmetric
(bijectivity of the canonical entwining assumed)
coalgebra-Galois extensions, we consider such connections in a
coalgebra-Galois extension of a quantum homogeneous space. This is a
preparation for an explicit example in the next section.

 Let $P$ be a Hopf algebra and  $B \subset P$ a left coideal subalgebra.
Consider the homogeneous coalgebra-Galois $P/B^+P$-extension as in Section~\ref{2:qspaces}. Write $C:=P/B^+P$ and $\pi: P\to C$ for the canonical epimorphism.
If the antipode $S$ of $P$ is bijective, then the canonical entwining map
$\psi$ given by $\psi(c \otimes p)=p\sw 1 \otimes \pi(p'p\sw 2) =
p\sw 1 \otimes c\cdot p\sw 2$ for $p' \in \pi^{-1}(c)$ is bijective with
inverse 
$\psi^{-1}(p \otimes c)=c\cdot S^{-1}(p\sw 2) \otimes p\sw 1$, so that
$B \subset P$ is a symmetric coalgebra-Galois extension.

Now, if we consider the Hopf algebra $P$ as a $P$-bicomodule by the
comultiplication $\Delta$ (regular coactions), then the 
 universal differential calculus $\Omega^1 P$ is bicovariant
(cf.\ \cite{w-sl89}). More precisely, the diagonal $P$-coactions on $P\otimes P$
can be restricted to a right $\Delta_{\Omega^1P}$ and a left ${}_{\Omega^1P}\hD$
coaction on $\Omega^1 P$.
%
The coactions $\Delta_{\Omega^1
P}$ and ${}_{\Omega^1
P}\Delta$ make $\Omega^1 P$ into a $P$-bicomodule. Furthermore, one
easily checks that the universal differential $\d: P\ra \Omega^1 P$ is a
$P$-bicomodule map, i.e.,
$$
{}_{\Omega^1
P}\Delta\circ \d = (\id\otimes \d)\circ\Delta , \qquad \Delta_{\Omega^1
P}\circ \d = (\d\otimes\id)\circ \Delta.
$$
Since $\Omega^1 P$ is a bicovariant calculus, one can, in particular,
consider {\em left-invariant forms}, i.e., elements $\omega\in \Omega^1
P$ such that ${}_{\Omega^1
P}\Delta (\omega) = 1\otimes \omega$. Any bicovariant calculus on $P$,
which necessarily is obtained as a quotient of the universal calculus
$\Omega^1 P$, is generated by left-invariant forms (as a left or right
$P$-module). 

In the case of a coalgebra-Galois extension $B \subset P$, one can
also study strong connections whose connection forms $\omega$ are
left-invariant,
i.e.,  such that, for all $c\in C$, ${}_{\Omega^1
P}\Delta (\omega(c)) = 1\otimes \omega(c)$.
The following theorem classifies all left-invariant strong-connection 
forms in the symmetric case \cite[Proposition~4.4]{bm00}.

\begin{Thm} \label{2:thm.str.con.hom}
Consider a  homogeneous coalgebra-Galois $C$-extension
$B \subset P$. Assume that the antipode $S$ is bijective, i.e.,
$B \subset P$ is a  symmetric coalgebra-Galois extension.
View the Hopf algebra $P$ as a right $C$-comodule
as  $\Delta_P = (\id_P\otimes \pi)\circ\Delta$ and as a left $C$-comodule via
${}_P\Delta = (\pi\otimes\id_P)\circ\Delta$. (Note that this is not the induced left coaction \eqref{2:leftco}.)
 Then there is a one-to-one
correspondence between left-invariant strong-connection
 forms and $C$-bicomodule maps $i:C \ra P$ such that
$ \pi\circ i=\id_C$, $i(\pi(1)) = 1$, and $\varepsilon_P \circ i =
\varepsilon_C$. 
A connection $1$-form is given by
$\omega(c)=S(i(c)\sw 1)\d(i(c)\sw 2)$.
\end{Thm}
\begin{proof}
Given such a splitting $i:C\ra P$ of $\pi$, consider
$\omega(c)= Si(c)\sw 1\d (i(c)\sw 2)$,
as stated. The normalisation conditions imply that
 $\omega(\pi(1)) = 0$ and
$\widetilde{\can}\circ\omega(c) = 1\otimes c -
\eps(c)1\otimes\pi(1)$. Use the short-hand notation $\psi^2:= (\id_P \otimes \psi)\circ (\psi\ot \id_P)$ and compute
\begin{eqnarray*}
\psi^{2}(c\sw 1\otimes \omega(c\sw 2)) & = & Si(c\sw 2)\sw 2\d
i(c\sw 2)\sw 3\otimes\pi(i(c\sw 1) Si(c\sw 2)\sw 1i(c\sw 2)\sw 4)\\
& = &
Si(c)\sw 3 \d
i(c)\sw 4\otimes\pi(i(c)\sw 1S(i(c)\sw 2)i(c)\sw 5) \quad \mbox{\rm ($i$
is left-colinear)}\\
& = & Si(c)\sw 1\d i(c)\sw 2\otimes \pi(i(c)\sw 3)\\
& = & Si(c\sw 1)\sw 1\d i(c\sw 1)\sw 2\otimes \pi(i(c\sw 2)) \qquad
\mbox{\rm ($i$ is right-colinear)}\\
& = & \omega(c\sw 1)\otimes c\sw 2 \qquad \qquad\qquad\qquad\qquad
\mbox{\rm ($\pi$ is split by $i$).}
\end{eqnarray*}
Theorem~\ref{2:con.form} implies that $\omega$ is a
connection  one-form. Finally, compute
\begin{eqnarray*}
(\id\otimes\Delta_P)(\omega(c)) & = & Si(c)\sw 1\otimes i(c)\sw
2\otimes\pi(i(c)\sw 3) - \eps(c) 1\otimes 1\otimes \pi(1)\\
& = & Si(c\sw 1)\sw 1\otimes i(c\sw 1)\sw 2\otimes c\sw 2 - \eps_C(c)
1\otimes 
1\otimes\pi(1)\\
& = & \omega(c\sw 1)\otimes c\sw 2+1\otimes 1\otimes c - \eps_C(c) 1\otimes
 1\otimes
\pi(1),
\end{eqnarray*}
where the use of the fact that $i$ is a right colinear splitting was
made in the derivation of the second equality.
Lemma~\ref{2:thm.str.inv} now implies that the connection
corresponding to $\omega$ is strong.

Conversely, assume that there is a  strong connection with the
left-invariant connection form $\omega$. Then  the left-invariance
of $\omega$ implies that there exists a splitting $i: C\ra P$ of $\pi$
such that $\eps_P\circ i =\eps_C$ and
$\omega(c)=Si(c)\sw 1\d i(c)\sw 2$  (cf.\
\cite[Proposition~3.5]{bm98b}). The fact that $\omega(\pi(1)) =0$
implies that $i(\pi(1))=1$. Applying $(\id\otimes \Delta_P)$ to
this $\omega$ and using Lemma~\ref{2:thm.str.inv}, one deduces that
$i$ is right-colinear. Since we are dealing with a symmetric
coalgebra-Galois extension the entwining map is bijective.
 The left coaction \eqref{2:leftco} induced by
$\psi^{-1}$ is ${}_P\hD(p) = \pi(S^{-1}p\sw 2)\otimes p\sw 1$. By
Lemma~\ref{2:lemma.scfs},
\[ ({}_P\hD\otimes \id_P)\omega(c) = \pi(i(c)\sw 1) \otimes
Si(c)\sw 2\otimes i(c)\sw
3 -\eps_C(c) \pi(1)\otimes 1\otimes 1
\]
must be equal to
\[
c\sw 1\otimes S(i(c\sw 2)\sw 1)\otimes i(c\sw 2)\sw 2
-\eps_C(c)\pi(1)\otimes
1\otimes 1.
\]
Applying $\id\otimes S^{-1}\otimes \eps_C$ to this equality, one deduces that
$i$ must be left-colinear (with respect to the coaction ${}_P\Delta$). 
This completes the proof.
\end{proof}

Theorem~\ref{2:thm.str.con.hom}  shows that strong connections in a
symmetric coalgebra-Galois extensions over a coideal subalgebra can be
obtained from purely (co)algebraic data. This observation allows one to
construct concrete examples of strong connections.

\subsubsection{Dirac monopoles over the Podle\'s 2-spheres}

Consider the quantum Hopf fibration described in Section~\ref{2:qsphere.coalg}.
In this case, the coalgebra $C$ is spanned by group-like elements $g_\mu$, $\mu\in \Z$, given by equations \eqref{Hopf.fibr.gen}, and the bicolinear splitting
$i$ of the projection $\pi$  can be
relatively easily computed. Explicitly, for all positive integers $n$, it comes out as
\begin{eqnarray}
i(g_n) &=& \prod_{k=0}^{n-1} \frac{\alpha + q^k s (\beta + \gamma) +
q^{2k} s^2 \delta}{1 + q^{2k} s^2}, \nonumber\\
i(g_{-n}) &=& \prod_{k=0}^{n-1} \frac{\delta - q^{-k} s (\beta + \gamma) +
q^{-2k} s^2 \alpha}{1 + q^{-2k} s^2},
\end{eqnarray}
where  the multiplication increases from  left to right. Thus, in
view of Theorem~\ref{2:thm.str.con.hom}, we
have constructed a strong left-invariant connection in the quantum Hopf
fibration with connection lifting $\ell=(S\ot\id)\ci\hD\ci i$. Such a connection in the classical Hopf fibration is known as
the {\em Dirac magnetic monopole}, as it has a
physical interpretation of a point particle which is a source of a
magnetic field. (See \cite{l-g00} for very nice description of
classical
monopoles from the point of view of noncommutative geometry.)
Motivated by this correspondence, the strong
connection constructed from $i$ via Theorem~\ref{2:thm.str.con.hom} is
called the {\em Dirac $q$-monopole}.

Furthermore, one can study the module of sections of a line bundle
associated to the quantum Hopf fibration. As a right $C$-comodule we
take the one-dimensional space
$V=k\,$ with the coaction $ \Delta_V(1)= 1 \otimes g_1$. Then
the module of sections turns out to be $E =\Hom^C(V,H)=
\{x(\alpha + s \beta) + y(\gamma + s \delta) \mid x,y \in \mathcal{O}(S^2_{q,s})\}$.
 Explicitly, $E$ is given by the  following idempotent matrix:
\[
E \cong (S_{q,s}^2)^2 \mathbf{p}, \quad \mathbf{p} =
\frac{1}{1+s^2} 
\left(\begin{array}{cc} 1- \zeta &\xi \\ -\eta &s^2+q^{-2}\zeta\end{array}\right).
\]



\end{document}